\documentclass[11pt]{amsart}
\usepackage{geometry}                
\usepackage{graphicx}
\usepackage{amssymb,upgreek}
\usepackage{amsmath,amsthm}
\numberwithin{equation}{section}

\usepackage{hyperref}
\hypersetup{
pdftex,colorlinks,citecolor=red,linkcolor=blue
}

\newtheorem{thm}{Theorem}[section]

\newtheorem{prop}[thm]{Proposition}
\newtheorem{lemma}[thm]{Lemma}

\newtheorem{thmintro}{Theorem}

\theoremstyle{definition}

\newcommand{\ZZ}{\mathbb{Z}}

\newcommand{\CC}{\mathbb{C}}   
\newcommand{\QQ}{\mathbb{Q}}

\newcommand{\PH}{\mathbb{P}}

\newcommand{\mb}{\mathbf}
\newcommand{\mh}{\mathbb}
\newcommand{\mr}{\mathrm}
\newcommand{\mc}{\mathcal}
\newcommand{\mf}{\mathfrak}
\newcommand{\N}{\mathbb N}
\newcommand{\Z}{\mathbb Z}
\newcommand{\Q}{\mathbb Q}

\newcommand{\sfa}{\mathsf{a}}
\newcommand{\sfb}{\mathsf{b}}

\def\Frob{\ensuremath{\mathrm{Frob}}}

\DeclareMathOperator{\fdeg}{\mathrm{fdeg}}

\newcommand{\der}{\mathrm{der}}

\newcommand{\ad}{\mathrm{ad}}

\newcommand{\SC}{\mathrm{sc}}
\newcommand{\Irr}{\mathrm{Irr}}

\newcommand{\nr}{\mathrm{nr}}
\newcommand{\cusp}{\mathrm{cusp}}
\newcommand{\KK}{K_\nr}
\newcommand{\unip}{\mathrm{unip}}
\newcommand{\Wr}{\mathrm{wr}}
\newcommand{\matje}[4]{\left(\begin{smallmatrix} #1 & #2 \\ 
#3 & #4 \end{smallmatrix}\right)}

\begin{document}

\title[Supercuspidal unipotent representations]{
Supercuspidal unipotent representations:\\ 
L-packets and formal degrees}

\author{Yongqi Feng}
\address[Y.~Feng]
{Department of Mathematics\\Shantou University\\Daxue Road 243, 515063 Shantou, China}
\email{yqfeng@stu.edu.cn}

\author{Eric Opdam}
\address[E.~Opdam]
{Korteweg-de Vries Institute for Mathematics\\Universiteit van Amsterdam\\Science Park 105-107\\ 
1098 XG Amsterdam, The Netherlands}
\email{e.m.opdam@uva.nl} 

\author{Maarten Solleveld}
\address[M.~Solleveld]
{Institute for Mathematics, Astrophysics and Particle Physics\\
Radboud Universiteit Nijmegen\\Heyendaalseweg 135, 
6525 AJ Nijmegen, The Netherlands}
\email{m.solleveld@science.ru.nl}

\date{\today}
\thanks{The first and third author were supported by NWO Vidi grant nr.~639.032.528.
The second author was supported by ERC--advanced grant no.~268105. The first author thanks the support from KdVI at the University of Amsterdam and from IMAPP at the Radboud University Nijmegen.}

\begin{abstract}
Let $K$ be a non-archimedean local field and let $\mb G$ be a connected reductive
$K$-group which splits over an unramified extension of $K$. We investigate supercuspidal
unipotent representations of the group $\mb G (K)$. We establish a bijection between
the set of irreducible $\mb G (K)$-representations of this kind and the set of cuspidal 
enhanced L-parameters for $\mb G(K)$, which are trivial on the inertia subgroup of 
the Weil group of $K$. The bijection is characterized by a few simple equivariance
properties and a comparison of formal degrees of representations with adjoint 
$\gamma$-factors of L-parameters.

This can be regarded as a local Langlands correspondence for all supercuspidal unipotent
representations. We count the ensuing L-packets, in terms of data from the affine
Dynkin diagram of $\mb G$. Finally, we prove that our bijection satisfies the conjecture 
of Hiraga, Ichino and Ikeda about the formal degrees of the representations. 
\end{abstract}

\maketitle
\tableofcontents

\section*{Introduction}

Let $K$ be a non-archimedean local field and let $\mb G$ be a connected reductive $K$-group.
Roughly speaking, a representation of the reductive $p$-adic group $\mb G (K)$ is unipotent if it 
arises from a unipotent representation of a finite reductive group associated to a parahoric 
subgroup of $\mb G (K)$. Among all (irreducible) smooth $\mb G (K)$-representations, this is a
very convenient class, which can be studied well with classification, parabolic induction and
Hecke algebra techniques. The work of Lusztig \cite{Lusztig-unirep,Lusztig-unirep2} and Morris
\cite{Mor} goes a long way towards a local Langlands correspondence for such representations, 
when $\mb G$ is simple and adjoint.

In this paper we focus on supercuspidal unipotent $\mb G (K)$-representations. For this to work
well, we assume throughout that $\mb G$ splits over an unramified extension of $K$. Our main
goal is a local Langlands correspondence for such representations, with as many nice properties
as possible. We will derive that from the following result, which says that one can determine 
the L-parameters of supercuspidal unipotent representations of a simple algebraic group by 
comparing formal degrees and adjoint $\gamma$-factors.

Denote the Weil group of $K$ by $\mb W_K$ and let
$\Frob \in \mb W_K$ be a geometric Frobenius element. A Langlands parameter is called 
unramified if it is trivial on the inertia subgroup of $\mb W_K$ (so that it is determined
by the image of $\Frob$ and by one unipotent element).

\begin{thmintro}\label{thm:A}
Consider a simple $K$-group $\mb{G}$ which 
splits over an unramified extension. For each irreducible supercuspidal unipotent 
$\mb G (K)$-representation $\pi$, there exists a discrete unramified 
local Langlands parameter $\lambda\in \Phi (\mb G (K))$ such that 
\begin{equation}\label{eq:fdegisgammafactor}
\fdeg(\pi) = C_\pi \gamma (\lambda) \text{ for some } C_\pi \in \QQ^\times
\end{equation}
as rational functions in $q_K$ with $\QQ$-coefficients. (Here $q_K$ denotes the cardinality of
the residue field of $K$, and one makes the terms of \eqref{eq:fdegisgammafactor} into
functions of $q_K$ by simultaneously considering unramified extensions of the field $K$.) Furthermore:
\begin{itemize}
\item $\lambda$ is essentially unique, in the sense that its image in the collection 
$\Phi (\mb G_\SC (K))$ of L-parameters for the simply connected cover of $\mb G (K)$ is unique.
\item When $\mb G$ is adjoint, the map $\pi \mapsto \lambda$ agrees with 
a parametrization of supercuspidal unipotent representations obtained in 
\cite{Mor,Lusztig-unirep,Lusztig-unirep2}.
\end{itemize}
\end{thmintro} 

The credits for Theorem \ref{thm:A} belong to several authors. The larger part of it, namely  
all cases with classical groups, was proven in \cite[Theorem 4.6.1]{FeOp}.  Quite generally, whenever
$\mb G$ is adjoint, \cite[Theorem 4.11]{Opd2} shows that the Langlands parameters from 
\cite{Lusztig-unirep,Lusztig-unirep2} satisfy \eqref{eq:fdegisgammafactor}. Hence the L-parameters 
from \cite{FeOp} coincide with those found by Lusztig \cite{Lusztig-unirep,Lusztig-unirep2}.
A little before that, Morris \cite[\S 5--6]{Mor} already associated L-parameters to supercuspidal unipotent 
representations of inner forms of split simple groups. We note that the parametrizations from \cite{Mor} 
and \cite{Lusztig-unirep} are presented in combinatorial fashion and do not involve formal degrees.
Instead, they are motivated (and nearly determined) by considerations with character sheaves and 
cuspidal local systems on unipotent orbits \cite{Lus-Intersect}. For that reason, the
L-parameters from \cite{Mor} and \cite{Lusztig-unirep} agree. Then \cite{FeOp,Opd2} show that these
parametrizations can be characterized uniquely by the equality \eqref{eq:fdegisgammafactor}.

For split exceptional groups the formal degrees in Theorem \ref{thm:A} were computed in 
\cite[\S 7]{Ree0} and \cite[\S 10--13]{Ree1}, and it was shown that they determine essentially unique 
Langlands parameters. Next \eqref{eq:fdegisgammafactor} was proven in \cite[\S 3.4]{HII}.
The essential uniqueness in the cases of the non-split inner forms of $E_6$ and $E_7$ is easy by the 
extremely small number of instances \cite[\S 4.4]{Opd2}. Hence the Langlands parameters determined by 
formal degrees agree with those from \cite{Mor,Lusztig-unirep} for inner forms of exceptional split groups. 
For outer forms of exceptional groups all this follows from the explicit computations in
\cite[\S 4.4]{Fen} and a comparison with \cite{Lusztig-unirep2}.

We will make the above parametrization of supercuspidal unipotent representations 
more precise and generalize it to connected reductive $K$-groups.
Let $\Irr (\mb{G}(K))_{\cusp, \unip}$ be the collection of irreducible supercuspidal 
unipotent representations of $\mb{G}(K)$, modulo isomorphism. 
Let $\Phi (\mb G (K))_\cusp$ be the set of cuspidal enhanced L-parameters for $\mb G(K)$,
considered modulo conjugation by the dual group $\mb{G}^\vee$. 
We denote its subset of unramified parameters by $\Phi_\nr (\mb G (K))_\cusp$.
(See Section \ref{sec:prelim} for the definitions of these and related objects.)
Our main result can be summarized as follows:

\begin{thmintro}\label{thm:C}
Let $\mb G$ be a connected reductive $K$-group which splits over an unramified extension. 
There exists a bijection
\[
\begin{array}{ccc}
\Irr (\mb G (K))_{\cusp,\unip} & \longrightarrow & \Phi_\nr (\mb G (K))_\cusp \\
\pi & \mapsto & (\lambda_\pi, \rho_\pi)
\end{array}
\]
with the properties:
\begin{enumerate}
\item When $\mb G$ is semisimple, the formal degree of $\pi$ equals the adjoint
$\gamma$-factor of $\lambda_\pi$, up to a rational factor which depends only on $\rho_\pi$.
\item Equivariance with respect to tensoring by weakly unramified characters.
\item Equivariance with respect to $\mb W_K$-automorphisms of the root datum.
\item Compatibility with almost direct products of reductive groups.
\item Let $Z(\mb G)_s$ be the maximal $K$-split central torus of $\mb G$ 
and let $\mb H$ be the derived group of $\mb G / Z(\mb G)_s$.
When $Z(\mb G)_s (K)$ acts trivially on $\pi \in \Irr (\mb G (K))_{\cusp,\unip}$, 
we can regard $\pi$ as a representation of $(\mb G / Z(\mb G)_s)(K)$ and
restrict to a representation $\pi_{\mb H}$ of $\mb H (K)$. Then $\lambda_\pi$ has image 
in the Langlands $L$-group of $\mb G / Z(\mb G)_s$ and the canonical map
\[
\big( \mb G / Z(\mb G)_s \big)^\vee \rtimes \mb W_K \longrightarrow 
\mb H^\vee \rtimes \mb W_K 
\]
sends $\lambda_\pi$ to $\lambda_{\pi_{\mb H}}$.
\item The map in (5) provides a bijection between the intersection of 
$\Irr (\mb G (K))_{\cusp,\unip}$ with the L-packet of $\lambda_\pi$ and the intersection
of $\Irr (\mb H (K))_{\cusp,\unip}$ with the L-packet of $\lambda_{\pi_{\mb H}}$.
\end{enumerate}
For a given $\pi$ the properties (1), (2), (4) and (5) determine $\lambda_\pi$ uniquely, 
modulo tensoring by weakly unramified characters of $(\mb G / Z(\mb G)_s)(K)$.
\end{thmintro}

Here a character of a group like $\mb G (K)$ is called weakly unramified if its kernel 
contains all parahoric subgroups of $\mb G (K)$.
Property (3) is important for the generalization of such a correspondence to all unipotent 
representations of reductive $p$-adic groups, which is carried out in \cite{Sol}.

The bijection exhibited in Theorem \ref{thm:C} is of course a good candidate for a local
Langlands correspondence (LLC) for supercuspidal unipotent representations, and we will 
treat it as such. The second bullet of Theorem \ref{thm:A} says that comparing formal 
degrees and adjoint $\gamma$-factors completely characterizes the L-parameters of 
supercuspidal unipotent representations of simple adjoint $K$-groups exhibited by Lusztig and 
Morris. In fact the method with formal degrees from \cite{Ree1,FeOp,Fen} provides a little 
more information, which we use to fix a few arbitrary choices in 
\cite{Mor,Lusztig-unirep,Lusztig-unirep2}.
In particular our LLC is determined already by formal degrees of supercuspidal unipotent 
representations in combination with the functoriality properties (2) and (4).

We point out that our correspondence is constructive. Indeed, for inner twists of simple adjoint 
unramified groups the enhanced L-parameters $(\lambda_\pi,\rho_\pi)$ can already be found in 
\cite{Mor}. For simple adjoint groups that split over an unramified extension the elements 
$\lambda_\pi (\Frob)$ are known explicitly from \cite{Lusztig-unirep,Lusztig-unirep2}, while
the unipotent class from $\lambda_\pi$ is given in \cite{Ree1,FeOp,Fen}. The enhancements 
$\rho_\pi$ are not uniquely determined, but there are only very few possibilities and those are 
given by the 
classification of cuspidal local systems on simple complex groups in \cite{Lus-Intersect}.
Further, our methods to generalize from simple adjoint to reductive groups are constructive, 
so that for any given supercuspidal unipotent representation one can in principle write down 
the enhanced L-parameter.\\

When $\mb G$ is semisimple we obtain finer results than Theorem \ref{thm:C}, summarized in 
Theorem \ref{thm:B}. In that setting we explicitly describe the number of cuspidal 
enhancements of $\lambda_\pi$ and the number of supercuspidal representations in the 
L-packet of $\lambda_\pi$, with combinatorial data coming from the affine Dynkin diagrams 
of $\mb G$ and $\mb G^\vee$.

Strengthening and complementing Theorem \ref{thm:C}, we will prove a conjecture by Hiraga, 
Ichino and Ikeda (cf.~\cite[Conjecture 1.4]{HII}) for unitary supercuspidal unipotent 
representations $\mb{G}(K)$. It relates formal degrees and adjoint $\gamma$-factors more
precisely than Theorem \ref{thm:A}.

Fix an additive character $\psi : K \to \CC^\times$ of order 0 and endow $K$ with the 
Haar measure that gives the ring of integers volume 1. Using these data, we normalize 
the Haar measure on $\mb G (K)$ as in \cite{HII}. The adjoint $\gamma$-factor 
$\gamma (s,\mr{Ad} \circ \lambda, \psi)$ involves the adjoint representation Ad of 
${}^L \mb G$ on $\mr{Lie} \big( (\mb G / Z(\mb G)_s )^\vee \big)$. Then $\gamma (\lambda)$
from Theorem \ref{thm:A} equals $\gamma (0,\mr{Ad} \circ \lambda, \psi)$. We will prove:

\begin{thmintro}\label{thm:HII}
Let $\mb G$ be a connected reductive $K$-group which splits over an unramified extension.
Let $\pi \in \Irr (\mb G (K))_{\cusp,\unip}$ be unitary and let $(\lambda_\pi ,
\rho_\pi)$ be the enhanced L-parameter assigned to it by Theorem \ref{thm:C}. Then
\[
\fdeg(\pi) = \frac{\dim (\rho_\pi) \; |\gamma (0,\mr{Ad} \circ \lambda_\pi, \psi)|}{
|Z_{(\mb G / Z(\mb G)_s)^\vee} (\lambda_\pi)|} .
\]
\end{thmintro}

Theorem \ref{thm:HII} shows in particular that all  supercuspidal members of one unipotent
L-packet have the same formal degree (up to some rational factor), as expected in the
local Langlands program. \\

Let us discuss the contents of the paper and the proofs of the main results in more detail. 
In Section \ref{sec:prelim} we fix the notations and we recall some facts about reductive
groups, enhanced Langlands parameters and cuspidal unipotent representations.
Let $\Omega$ be the fundamental group of $\mb{G}$, interpreted as a group of 
automorphisms of the affine Dynkin diagram of $\mb{G}$. We denote the action of $\Frob \in 
\mb W_K$ on $\mb G^\vee$ by $\theta$, so that the group of weakly unramified
characters of $\mb G (K)$ can be expressed as $Z(\mb G^\vee)_{\mb W_K}$ and as the dual group
$(\Omega^\theta)^*$ of $\Omega^\theta$. In Section \ref{sec:mainss} we make Theorem \ref{thm:C} 
more precise for semisimple $K$-groups, counting the involved objects in terms of subquotients 
of the finite abelian group $(\Omega^\theta )^*$. A large part of the paper is dedicated 
to proving Theorem \ref{thm:B}, in bottom-up fashion.

In Sections \ref{sec:innAn}--\ref{sec:E7} we consider simple adjoint groups case-by-case. 
The majority of our claims can be derived quickly from \cite[\S 5--6]{Mor} and the 
tables \cite[\S 7]{Lusztig-unirep} and \cite[\S 11]{Lusztig-unirep2}, which contain a lot 
of information about the parametrization of $\Irr (\mb G (K))_{\cusp,\unip}$ from 
Theorem \ref{thm:A}. A simple group of type $E_8, F_4$ or $G_2$ is both simply 
connected and adjoint, so $\Omega$ is trivial. Then Theorem \ref{thm:B} is contained
entirely in \cite{Lusztig-unirep}, and we need not spend any space on it. For other simple
adjoint groups we compute several data that cannot be found in the works of Morris and Lusztig.

The main novelty in Sections \ref{sec:innAn}--\ref{sec:E7} is the equivariance of the LLC with 
respect to $\mb W_K$-automorphisms of the root datum (part (3) of Theorem \ref{thm:C}), that 
was not discussed in the sources on which we rely here. In some remarks we already take a 
look at certain non-adjoint simple groups. This concerns cases where we can only check 
Theorem \ref{thm:C} by direct calculations. 
In Section \ref{sec:adjoint} we explain in detail which parts of Sections \ref{sec:innAn}--\ref{sec:E7} 
are needed where, and we complete the proof of the main theorem for adjoint groups. 

In Sections \ref{sec:semisimple} and \ref{sec:proofss} we generalize Theorem \ref{thm:B} 
from adjoint semisimple to all semisimple groups. In particular, we investigate 
what happens when an adjoint $K$-group $\mb G_\ad$ is replaced by a covering group $\mb G$. 
It is quite easy to see how $\Irr(\mb{G}(K))_{\cusp, \unip}$ behaves. Namely, several 
unipotent cuspidal representations of $\mb{G}_\ad(K)$ coalesce upon pullback to $\mb G (K)$, 
and then decompose as a direct sum of a few irreducible unipotent cuspidal representations of 
$\mb{G}(K)$. With some technical work, we prove that the same behaviour (both qualitatively 
and quantitatively) occurs for enhanced L-parameters.

The proof of the main theorem for reductive $K$-groups (Section \ref{sec:mainred}) can roughly 
be divided into two parts. First we deal with the case where the connected centre of $\mb{G}$ 
is anisotropic. We reduce to the derived group of $\mb{G}$, which is semisimple, and use the 
already established results for semisimple groups. To deal with general connected reductive groups, 
we note that the connected centre is an almost direct product of its maximal split and maximal
anisotropic subtori. Applying Hilbert's theorem 90 to the maximal split torus, we obtain a 
corresponding decomposition of the group of $K$-rational points. This enables us to reduce
to the cases of tori (well-known) and of reductive $K$-groups with anisotropic connected 
centre.

We attack the HII conjecture in Section \ref{sec:HII}. For simple adjoint groups, the second 
author already proved Theorem \ref{thm:HII} in \cite{Opd2}. Starting from that and using
the proof of Theorem \ref{thm:A}, we extend Theorem \ref{thm:HII} to all reductive $K$-groups
that split over an unramified extension. 

Finally, in the appendix we explore the behaviour of L-parameters and adjoint 
$\gamma$-factors under Weil restriction. Whereas L-functions are always preserved, it turns
out that adjoint $\gamma$-factors sometimes change under Weil restriction. Nevertheless,
we can use these computations to prove that the HII conjectures are always stable under 
restriction of scalars. That is, if $L/K$ is a finite separable extension of non-archimedean
local fields and the HII conjectures hold for a reductive $L$-group, then they also hold
for the reductive $K$-group obtained by restriction of scalars (and conversely).\\

\textbf{Acknowledgements.}

We thank the referees for their helpful comments and careful reading.

\section{Preliminaries}
\label{sec:prelim}

Throughout this paper we let $K$ be a non-archimedean local field with finite residue field 
$\mf F$ of cardinality $q_K = |\mf F|$. We fix a separable closure $K_s$ of $K$ and we
let $\KK\subset K_s$ be the maximal unramified extension of $K$. 
The residue field $\overline{\mf F}$ of $\KK$ is an algebraic closure of $\mf F$. 
There are isomorphisms of Galois groups 
$\mathrm{Gal}(\KK/K) \simeq \mathrm{Gal}(\overline{\mf F} / \mf F) \simeq \hat{\ZZ}$. 
The geometric Frobenius element $\Frob$, whose inverse induces the automorphism 
$x \mapsto x^{q_K}$ for any $x \in \overline{\mf F}$, is a topological generator 
of $\mathrm{Gal}(\overline{\mf F}/\mf F)$. Let $\mb I_K =\mathrm{Gal}(K_s/\KK)$ be the 
inertia subgroup of $\mathrm{Gal}(K_s/K)$ and let $\mb W_K$ be the Weil group of $K$.
We fix a lift of Frob in $\mathrm{Gal}(K_s/K)$, so that $\mb W_K = \mb I_K \rtimes \langle
\Frob \rangle$.\\
 
Unless otherwise stated, $\mb{G}$ denotes an unramified connected reductive linear 
algebraic group over $K$. By unramified we mean that $\mb{G}$ is a quasi-split group defined 
over $K$ and that $\mb{G}$ splits over $\KK$. The group $\mb{G}(\KK)$ of $\KK$-points of 
$\mb{G}$ is often denoted by $G=\mb{G}(\KK)$. Let $Z(\mb{G})$ be the centre of $\mb{G}$, 
and write $\mb{G}_\ad := \mb{G}/Z(\mb{G})$ for the adjoint group of $\mb{G}$. 

We fix a Borel $K$-subgroup $\mb{B}$ and maximally split maximal $K$-torus $\mb{S}\subset\mb{B}$ 
which splits over $\KK$. We denote by $\theta$ the finite order automorphism of $X_*(\mb S)$ 
corresponding to the action of $\Frob$ on $S=\mb{S}(\KK)$. Let $R^\vee$ be the coroot system
of $(\mb G, \mb S)$ and define the abelian group 
\[
\Omega = X_* (\mb S) / \ZZ R^\vee .
\] 
Let $\mb G^\vee$ be the complex dual group of $\mb G$. Then $Z(\mb G^\vee)$ can be 
identified with $\Irr (\Omega) = \Omega^*$, and $\Omega$ is naturally isomorphic to the 
group $X^* (Z(\mb G^\vee))$ of algebraic characters of $Z(\mb G^\vee)$. In particular
\begin{equation}\label{eq:2.1}
\begin{aligned}
& \Omega_\theta \cong X^* \big( Z(\mb G^\vee) \big)_\theta = X^* \big( Z(\mb G^\vee)^\theta \big) ,\\
& \Omega^\theta \cong X^* \big( Z(\mb G^\vee) \big)^\theta = X^* \big( Z(\mb G^\vee)_\theta \big) .
\end{aligned}
\end{equation}
The isomorphism classes of inner twists of $\mb{G}$ over $K$ are naturally parametrized 
by the elements of the continuous Galois cohomology group 
\[
H_c^1(K,\mb{G}_\ad) \cong H_c^1(F,G_\ad) ,
\]
where $F$ denotes the automorphism of $G_\ad := \mb{G}_\ad (\KK)$ by which $\Frob$ acts on $G_\ad$.
A cocycle in $Z_c^1(F, G_\ad)$ is determined by the image $u \in G_\ad$ of $F$. 
The $K$-rational structure of $\mb{G}$ corresponding to such a $u \in G_\ad$ 
is given by the action of the inner twist $F_u:=\mathrm{Ad}(u) \circ F \in \mathrm{Aut}(G)$ 
of the $K$-automorphism $F$ on $G$. We will denote this $K$-rational form of $\mb{G}$ by 
$\mb{G}^u$, and the corresponding group of $K$-points by $G^{F_u}$. 

The cohomology class $\omega \in H_c^1(F, G_\ad)$ of the cocycle is represented by the 
$F$-twisted conjugacy class of $u$. Let $\Omega_\ad$ be the fundamental group of $\mb G_\ad$.
By a theorem of Kottwitz \cite{Kot1,Tha} and by \eqref{eq:2.1} there is a natural isomorphism
\begin{equation}\label{eq:2.22}
H^1_c (F, G_\ad) \cong H^1_c (F,\Omega_\ad) \cong (\Omega_\ad)_\theta 
\cong X^* \big( Z({G_\ad}^\vee)^\theta \big) .
\end{equation}
This works out to mapping $\omega$ to $u \in (\Omega_\ad)_\theta$. For each class 
$\omega \in H^1_c (F, G_\ad)$ we fix an inner twist $F_u$ of $F$ representing $\omega$, 
and we denote this representative by $F_\omega$. Then
\[
\mb G^\omega (K) = G^{F_\omega} . 
\]
Let $G_1$ be the kernel of the Kottwitz homomorphism $G \to X^* (Z(\mb G^\vee))$ \cite{Ko,HR}. 
This map is $\mb W_K$-equivariant and yields a short exact sequence
\[
1 \to G_1^{F_\omega} \to G^{F_\omega} \to X^*(Z(\mb G^\vee))^\theta \cong \Omega^\theta \to 1.
\]
We say that a character $\chi$ of $G^{F_\omega}$ is weakly unramified if $\chi$ is trivial 
on $G^{F_\omega}_1$, and we denote by $X_\Wr  (G^{F_\omega})$ the abelian group of 
weakly unramified characters. Since $\mb G$ is unramified there are natural isomorphisms
\cite[\S 3.3.1]{Hai}
\begin{equation}\label{eq:2.2}
\Irr (G^{F_\omega} / G_1^{F_\omega}) = X_\Wr  (G^{F_\omega}) \cong (\Omega^\theta)^*
\cong Z(\mb G^\vee)_\theta.
\end{equation}
This can be regarded as a special case of the local Langlands correspondence.
The identity components of the groups in \eqref{eq:2.2} are isomorphic to the group of 
unramified characters of $G^{F_\omega}$ (which is trivial whenever $\mb G$ is semisimple).\\

Let ${}^L \mb G = \mb G^\vee \rtimes \mb W_K$ be the L-group of $\mb G$. Recall that a 
L-parameter for $\mb G^\omega (K) = G^{F_\omega}$ is a group homomorphism
\[
\lambda : \mb W_K \times SL_2 (\CC) \to \mb G^\vee \rtimes \mb W_K
\]
satisfying certain requirements \cite{Bor}. We say that $\lambda$ is unramified if
$\lambda (w) = (1,w)$ for every $w\in \mb I_K$ and we say that $\lambda$ is discrete if
the image of $\lambda$ is not contained in the L-group of any proper Levi subgroup of 
$G^{F_\omega}$. We denote the set of $\mb G^\vee$-conjugacy classes of L-parameters (resp. 
unramified L-parameters and discrete L-parameters) for $G^{F_\omega}$ by $\Phi (G^{F_\omega})$ 
(resp. $\Phi_\nr (G^{F_\omega})$ and $\Phi^2 (G^{F_\omega})$). The group
$Z(\mb G^\vee)$ acts naturally on the set of L-parameters, by
\begin{equation}\label{eq:2.3}
(z \lambda) (\Frob^n w, x) = (z \lambda (\Frob))^n \lambda (w,x)
\quad z \in Z(\mb G^\vee), n \in \ZZ, w \in \mb I_K, x \in SL_2 (\CC) .
\end{equation}
This descends to an action of
\[
Z(\mb G^\vee)_\theta \cong (\Omega^*)_\theta = (\Omega^\theta)^* 
\]
on $\Phi (G^{F_\omega})$.

For any $\lambda \in \Phi (G^{F_\omega})$ the centralizer
$A_\lambda := Z_{\mb G^\vee}(\mr{im} \, \lambda)$ satisfies
\[
A_\lambda \cap Z(\mb G^\vee) = Z({}^L \mb G) = Z(\mb G^\vee)^\theta ,
\] 
and $A_\lambda / Z(\mb G^\vee)^\theta$ is finite if and only if $\lambda$ is discrete.
Let $\mathcal{A}_\lambda$ be the component group of the full pre-image of 
\begin{equation}\label{eq:2.11}
A_\lambda / Z( \mb G^\vee)^\theta \cong 
A_\lambda Z(\mb G^\vee) / Z(\mb G^\vee) \subset {\mb G^\vee}_\ad
\end{equation}
in the simply connected covering $(\mb G^\vee)_\SC$ of the derived group of $\mb G^\vee$. 
Equivalently, $\mc A_\lambda$ can also be described as the component group of
\begin{equation}\label{eq:2.12}
Z_{{\mb G^\vee}_\SC}^1 (\lambda) = \big\{ g \in {\mb G^\vee}_\SC : g \lambda g^{-1} = 
\lambda b \text{ for some } b \in B^1 (\mb W_K, Z(\mb G^\vee)) \big\} .
\end{equation}
Here $B^1 (\mb W_K, Z(\mb G^\vee))$ denotes the group of 1-coboundaries for group cohomology,
that is, the set of maps  $\mb W_K \to Z(\mb G^\vee)$ of the form $w \mapsto z w(z^{-1})$
for some $z \in Z(\mb G^\vee)$.

An enhancement of $\lambda$ is an irreducible representation $\rho$ of $\mc A_\lambda$. 
The group $G^\vee$ acts on the set of enhanced L-parameters by
\[
g \cdot (\lambda,\rho) = (g \lambda g^{-1}, \rho \circ \mr{Ad}(g^{-1})) .
\]
We write 
\[
\Phi_e ({}^L \mb G) = \{ (\lambda,\rho) : \lambda \text{ is an L-parameter for } \mb G (K), 
\rho \in \Irr (\mc A_\lambda) \} / G^\vee .
\]
Fix a complex character $\zeta$ of the centre $Z({\mb G^\vee}_\SC)$ of ${\mb G^\vee}_\SC$ whose 
restriction to $Z({}^L \mb G_\ad ) = Z({\mb G^\vee}_\SC)^\theta$ corresponds to $\omega$ via 
the Kottwitz isomorphism. If $\omega$ is given as an element of $\Omega_\ad$ (not just in 
$(\Omega_\ad)_\theta$), then there is a preferred way to define a character of 
$Z({\mb G^\vee}_\SC)$, namely via the Kottwitz isomorphism of the $K$-split form of $\mb G$. 
In particular $\omega = 1$ corresponds to the trivial character. 

Let $\Irr (\mathcal{A}_\lambda, \zeta)$ be the set of irreducible representations of 
$\mathcal{A}_\lambda$ whose restriction to $Z({\mb G^\vee}_\SC)$ is a multiple of $\zeta$. 
The set of enhanced L-parameters for $G^{F_\omega}$ is
\begin{equation}\label{eq:2.9}
\Phi_e (G^{F_\omega}) := \big\{ (\lambda, \rho) \in \Phi_e ({}^L \mb G) \mid 
\rho \in \mathrm{Irr}(\mathcal{A}_\lambda, \zeta) \big\} .
\end{equation}
We note that the existence of a $\rho \in \Irr (\mathcal{A}_\lambda, \zeta)$ is equivalent
to $\lambda$ being relevant \cite[\S 8.2.ii]{Bor} for the inner twist $\mb G^\omega$ 
of the quasi-split $K$-group $\mb G$ \cite[Proposition 1.6]{ABPS1}.

Let $Z^1_{{\mb G^\vee}_\SC} (\lambda (\mb W_K))$ be the inverse image of 
$Z_{\mb G^\vee}(\lambda (\mb W_K)) / Z(\mb G^\vee)^{\mb W_K}$ in ${\mb G^\vee}_\SC$. 
The unipotent element $u_\lambda := \lambda \big( 1, \big( \begin{smallmatrix} 1 & 1 \\ 0 & 1
\end{smallmatrix}\big) \big) \in \mb G^\vee$ can also be regarded as an element of the unipotent
variety of ${\mb G^\vee}_\SC$, and then 
\begin{equation}\label{eq:2.13}
\mc A_\lambda = \pi_0 \big( Z_{Z^1_{{\mb G^\vee}_\SC} (\lambda (\mb W_K))} (u_\lambda) \big) .
\end{equation}
We say that $\rho$ is a cuspidal representation of $\mc A_\lambda$, or that 
$(\lambda,\rho)$ is a cuspidal (enhanced) L-parameter for $G^{F_\omega}$, if $(u_\lambda, \rho)$ 
is a cuspidal pair for $Z^1_{{\mb G^\vee}_\SC} (\lambda (\mb W_K))$ \cite[Definition 6.9]{AMS1}.
Equivalently, $\rho$ determines a $Z^1_{{\mb G^\vee}_\SC} (\lambda (\mb W_K))$-equivariant
cuspidal local system on the conjugacy class of $u_\lambda$. This is only possible if
$\lambda$ is discrete (but not every discrete L-parameter admits cuspidal enhancements). 
We refer to \cite{Lus-Intersect} for more information about cuspidal local systems, and in 
particular their classification for every simple complex group.
We denote the set of $\mb G^\vee$-conjugacy classes of cuspidal enhanced L-parameters for 
$G^{F_\omega}$ by $\Phi (G^{F_\omega})_\cusp$.

The $(\Omega^\theta)^*$-action \eqref{eq:2.3} extends to enhanced L-parameters by
\begin{equation}\label{eq:2.4}
z \cdot (\lambda,\rho) = (z \lambda,\rho) \qquad z \in (\Omega^\theta)^*, (\lambda,\rho)
\in \Phi_e (G^{F_\omega}).
\end{equation}
The extended action preserves both discreteness and cuspidality.\\

Let $\Irr (G^{F_\omega})$ be the set of irreducible smooth $G^{F_\omega}$-representations
on complex vector spaces. The group $(\Omega^\theta)^*$ acts on $\Irr (G^{F_\omega})$
via \eqref{eq:2.2} and tensoring with weakly unramified characters. It is expected that
under the local Langlands correspondence (LLC) this corresponds precisely to the action
\eqref{eq:2.4} of $(\Omega^\theta)^*$ on $\Phi_e (G^{F_\omega})$. In other words, the
conjectural LLC is $(\Omega^\theta)^*$-equivariant.

Furthermore, the LLC should behave well with respect to direct products. Suppose that
$\mb G^\omega$ is the almost direct product of $K$-subgroups $\mb G_1$ and $\mb G_2$. 
Along the quotient map
\[
q : \mb G_1 \times \mb G_2 \to \mb G^\omega 
\]
one can pull back any representation $\pi$ of $\mb G^\omega (K)$ to a representation 
$\pi \circ q$ of $\mb G_1 (K) \times \mb G_2 (K)$. Since $q$ need not be surjective on
$K$-rational points, this operation may destroy irreducibility.
Assume that $\pi$ is irreducible and 
that $\pi_1 \otimes \pi_2$ is any irreducible constituent of $\pi \circ q$. Then the 
image of the L-parameter $\lambda_\pi$ of $\pi$ under the map 
\[
q^\vee : (\mb G^\omega)^\vee \to \mb G_1^\vee \times \mb G_2^\vee 
\]
should be the L-parameter $\lambda_{\pi_1} \times \lambda_{\pi_2}$ of $\pi_1 \otimes \pi_2$. 
In this case $\mc A_{\lambda_\pi}$ is naturally a subgroup of 
$\mc A_{\lambda_{\pi_1}} \times \mc A_{\lambda_{\pi_2}}$. We say that a LLC 
(for some class of representations) is compatible with almost direct products if, when 
$(\lambda_\pi,\rho_\pi)$ denotes the enhanced L-parameter of $\pi$ and $\mb G^\omega =
\mb G_1 \mb G_2$ is an almost direct product of reductive $K$-groups,
\begin{equation}\label{eq:2.20}
\lambda_{\pi_1} \times \lambda_{\pi_2} = q^\vee (\lambda_\pi) \text{ and } \big( \rho_{\pi_1} 
\otimes \rho_{\pi_2} \big) |_{\mc A_{\lambda_\pi}} \text{ contains } \rho_\pi .
\end{equation}
We also want the LLC to be equivariant with respect to automorphisms of the root datum, 
in a sense which we explain now. Let 
\[
\mc R (\mb G, \mb S) = (X^* (\mb S), R, X_* (\mb S),R^\vee ,\Delta)
\]
be the based root datum of $\mb G$, where $\Delta$ is the basis determined by the Borel
subgroup $\mb B \subset \mb G$. Since $\mb S$ and $\mb B$ are defined over $K$, the
Weil group $\mb W_K$ acts on this based root datum. 

When $\mb G$ is semisimple, any automorphism of $\mc R (\mb G,\mb S)$ is completely
determined by its action on the basis $\Delta$. Then we call it an automorphism of the
Dynkin diagram of $(\mb G,\mb S)$, or just a diagram automorphism of $\mb G$.
When $\mb G$ is simple and not of type $D_4$, the collection of such diagram automorphisms 
is very small: it forms a group of order 1 (type $A_1,B_n,C_n,E_7,E_8,F_4,G_2$ 
or a half-spin group) or 2 (type $A_n,D_n,E_6$ with $n>1$, except half-spin groups).

Suppose that $\tau$ is an automorphism of $\mc R (\mb G, \mb S)$ which commutes with the 
action of $\mb W_K$. Via the choice of a pinning of $\mb G^\vee$ (that is, the choice of a
nontrivial element in every root subgroup for a simple root), $\tau$ acts on $\mb G^\vee$
and ${}^L \mb G$. Then it also acts on enhanced L-parameters, by 
\[
\tau \cdot (\lambda,\rho) = (\tau \circ \lambda, \rho \circ \tau^{-1}).
\]
Then $\tau$ also acts on $\Phi_e ({}^L \mb G)$.
The action of $\tau$ on $\mb G^\vee$ is uniquely determined up to 
inner automorphisms, so the action on $\Phi_e ({}^L \mb G)$ is canonical. Considering 
$\omega \in \Omega_\ad$ as an element of $\Irr (Z ({\mb G^\vee}_\SC))$, we can define 
$\tau (\omega) = \omega \circ \tau$. Then $\tau$ maps enhanced L-parameters relevant 
for $G^{F_{\tau (\omega)}}$ to enhanced L-parameters relevant for $G^{F_\omega}$. 

From \cite[Lemma 16.3.8]{Spr} we see that the automorphism $\tau$ of $\mc R (\mb G,\mb S)$ 
can be lifted to a $K_\nr$-automorphism of $\mb G_\ad$. That uses only the diagram
automorphism induced by $\tau$. As $\tau$ also gives an automorphism of $\mb S$, it 
determines an automorphism of $\mb S$ stabilizing $Z(\mb G)$. The proof of 
\cite[Lemma 16.3.8]{Spr} also works for $\mb G$, when we omit the condition that the connected 
centre must be fixed and instead use the automorphism of $Z(\mb G)^\circ$ determined 
by $\tau$. Then $\tau$ lifts to a $K_\nr$-automorphism $\tau_{K_\nr}$ of $\mb G$ which
\begin{itemize}
\item stabilizes $\mb S$ and $\mb B$,
\item is unique up to conjugation by elements of $\mb S_\ad (K_\nr)$, where
$\mb S_\ad = \mb S / Z(\mb G)$.
\end{itemize}
Further, $\tau$ determines a permutation of the affine Dynkin diagram of
$(G,S)$. This in turn gives rise to a permutation of the set of vertices of a 
standard alcove in the Bruhat--Tits building of $(\mb G, K_\nr)$. For every such 
vertex $v$, we can require in addition that $\tau_{K_\nr}$ maps the $G$-stabilizer 
$G_v$ to $G_{\tau (v)}$. Since $\Omega_\ad$ acts faithfully on this standard alcove
and the image of $S \to G_\ad$ contains the kernel of $S_\ad \to \Omega_\ad$ 
\cite[\S 2.1]{Opd2}, this determines $\tau_{K_\nr} \in \mr{Aut}_{K_\nr}(\mb G)$ 
up conjugation by elements of $G_1 \cap S$. 

Let $u \in G_\ad$ represent $\omega$, so that 
\begin{equation}\label{eq:defGFu}
G^{F_\omega} \cong G^{F_u} = \{ g \in G : \mr{Ad}(u) \circ F (g) = g \} .
\end{equation}
By the functoriality of the Kottwitz isomorphism $\tau_{K_\nr} (u)$ represents $\tau (\omega)$. 
For $g \in G^{F_u}$:
\begin{equation}\label{eq:actionTau}
\mr{Ad}(\tau_{K_\nr} (u)) \circ F \circ \tau_{K_\nr} (g) = 
\tau_{K_\nr} \circ \mr{Ad}(u) \circ F (g) = \tau_{K_\nr} (g) ,
\end{equation}
so $\tau_{K_\nr} (g) \in G^{F_{\tau_{K_\nr} (u)}}$. Thus we obtain an isomorphism of $K$-groups
\[
\tau_K : \mb G^\omega \to \mb G^{\tau (\omega)} .
\]
Since $\tau_{K_\nr}$ was unique up to $G_1 \cap S$, $\tau_K$ is unique up to
conjugation by elements of $\mb S (K) \cap G_1$. (Not merely up to $\mb S_\ad (K)$
because $\tau_K (G_v) = G_{\tau (v)}$.) In particular, for every representation 
$\pi$ of $G^{F_{\tau (\omega)}}$ we obtain a representation $\pi \circ \tau_K$ 
of $G^{F_\omega}$, well-defined up to isomorphism.

Equivariance with respect to $\mb W_K$-automorphisms of the root datum means:
if $(\lambda_\pi,\rho_\pi)$ is the enhanced L-parameter of $\pi$ then
\begin{equation}\label{eq:}
(\tau \cdot \lambda_\pi, \rho_\pi \circ \tau^{-1}) \text{ is the enhanced L-parameter of } 
\pi \circ \tau_K,
\end{equation} 
for all $\tau \in \mr{Aut}(\mc R (\mb G,\mb S))$ which commute with $\mb W_K$. When $\mb G$ 
is semisimple, we also call this equivariance with respect to diagram automorphisms.

We note that it suffices to check this for automorphisms of $\mc R(\mb G, \mb S)$ which fix 
$\omega \in \Omega_\ad$. Indeed, if we know all those cases, then we can get equivariance
with respect to diagram automorphisms by defining the LLC for other groups $G^{F_{\omega'}}$
via the LLC for $G^{F_\omega}$ and a $\tau$ with $\tau (\omega) = \omega'$.\\

We define a parahoric subgroup of $G$ to be the stabilizer in $G_1$ of a facet (say $\mf f$)
of the Bruhat--Tits building of $(\mb G,K_\nr)$, and we typically denote it by $\mh P$.
Then $\mh P$ fixes $\mf f$ pointwise. If $\mf f$ is $F_\omega$-stable, it determines a
facet of the Bruhat--Tits building of $(\mb G^\omega,K)$, and $\mh P^{F_\omega}$ is the
associated parahoric subgroup of $G^{F_\omega}$. All parahoric subgroups of $G^{F_\omega}$
arise in this way.

Let $\mh P_u$ be the pro-unipotent radical of $\mh P$, that is, the kernel of the 
reduction map from $\mh P$ to the associated reductive group $\overline{\mh P}$ over 
$\overline{\mf F}$. Then $\mh P_u^{F_\omega}$ is the pro-unipotent radical of 
$\mh P^{F_\omega}$, and the quotient 
\begin{equation}\label{eq:2.5}
\mh P^{F_\omega} / \mh P_u^{F_\omega} = \overline{\mh P^{F_\omega}} \cong 
\overline{\mh P}^{F_\omega}
\end{equation}
is a connected reductive group over $\mf F$. Unipotent representations of
finite reductive groups like \eqref{eq:2.5} were classified in \cite[\S 3]{Lus-Chevalley}.
We call an irreducible representation of $\mh P^{F_\omega}$ unipotent (resp. cuspidal) 
if it arises by inflation from an irreducible unipotent (resp. cuspidal) representation of  
$\overline{\mh P}^{F_\omega}$.

An irreducible representation $\pi$ of $G^{F_\omega}$ is called unipotent if there
exists a parahoric subgroup $\mh P^{F_\omega}$ such that the restriction of $\pi$
to $\mh P^{F_\omega}$ contains a unipotent representation of $\mh P^{F_\omega}$. 
We denote the set of irreducible unipotent $G^{F_\omega}$-representations by
$\Irr (G^{F_\omega})_\unip$.

In this paper we are mostly interested in supercuspidal 
$G^{F_\omega}$-representations, which form a collection denoted $\Irr (G^{F_\omega})_\cusp$.
Among these, the supercuspidal unipotent representations form a subset  
$\Irr (G^{F_\omega})_{\cusp,\unip}$ which was described quite explicitly in 
\cite{Mor,Lusztig-unirep}. Every such $G^{F_\omega}$-representation arises from a cuspidal
unipotent representation $\sigma$ of a maximal parahoric subgroup $\mh P^{F_\omega}$.
For a given finite reductive group there are only few cuspidal unipotent representations, 
and the number of them does not change when \eqref{eq:2.5} is replaced by an isogenous 
$\mf F$-group. From the classification one sees that, when $\overline{\mh P}$ is simple, any 
cuspidal unipotent representation $(\sigma, V_\sigma)$ of $\overline{\mh P}^{F_\omega}$ is 
stabilized by every algebraic automorphism of $\overline{\mh P}^{F_\omega}$. 

By \cite{Opd2} there is a natural isomorphism
\begin{equation}\label{eq:2.6}
N_{G^{F_\omega}}(\mh P^{F_\omega}) / \mh P^{F_\omega} \cong \Omega^{\theta,\mh P},
\end{equation}
where the right hand side denotes the stabilizer of $\mh P$ in the abelian group $\Omega^\theta$. 
Morris shows in \cite[Proposition 4.6]{Mor} that, when $\mb G$ is adjoint, any unipotent 
$\sigma \in \Irr (\mh P^{F_\omega})$ can be extended to a representation of the normalizer of 
$\mh P^{F_\omega}$ in $G^{F_\omega}$. When $\mb G$ is semisimple, the group $\Omega^\theta$ 
embeds naturally in $\Omega_\ad^\theta = (\Omega_\ad)^\theta$.
Then $N_{G^{F_\omega}}(\mh P^{F_\omega}) / \mh P^{F_\omega}_u Z(G^{F_\omega})$ can be identified with 
a subgroup of $N_{G_\ad^{F_\omega}}(\mh P_\ad^{F_\omega}) / \mh P^{F_\omega}_{\ad,u}$, and in that 
way $\sigma$ can be extended to a representation of $N_{G^{F_\omega}}(\mh P^{F_\omega})$, on the same 
vector space. (The same conclusion holds when $\mb G$ is reductive, we will show that in
Section \ref{sec:mainred}.)

We fix one such extension, say $\sigma^N$. By \eqref{eq:2.6}, at least when $\mb G$ is semisimple:
\begin{equation}\label{eq:2.7}
\mr{ind}_{\mh P^{F_\omega}}^{N_{G^{F_\omega}}(\mh P^{F_\omega})}(\sigma) =
\bigoplus\nolimits_{\chi \in (\Omega^{\theta,\mh P})^*} \chi \otimes \sigma^N .
\end{equation}
When $Z(G^{F_\omega})$ is not compact, \eqref{eq:2.7} remains true if the right hand 
side is replaced by a direct integral over $(\Omega^{\theta,\mh P})^*$.
Furthermore it is known from \cite{Mor,Lusztig-unirep} that every representation
$\mr{ind}_{N_{G^{F_\omega}}(\mh P^{F_\omega})}^{G^{F_\omega}} (\chi \otimes \sigma^N)$
is irreducible and supercuspidal. Hence (when $\Omega^\theta$ is finite)
\begin{equation}\label{eq:2.8}
\mr{ind}_{\mh P^{F_\omega}}^{G^{F_\omega}} (\sigma) =
\bigoplus\nolimits_{\chi \in (\Omega^{\theta,\mh P})^*} 
\mr{ind}_{N_{G^{F_\omega}}(\mh P^{F_\omega})}^{G^{F_\omega}} (\chi \otimes \sigma^N) .
\end{equation}
Every element of $\Irr (G^{F_\omega})_{\unip,\cusp}$ arises in this way, from a pair 
$(\mh P,\sigma)$ which is unique up to $G^{F_\omega}$-conjugation.  
We denote the packet of irreducible supercuspidal unipotent $G^{F_\omega}$-representations 
associated to the conjugacy class of $(\mh P,\sigma)$ via \eqref{eq:2.7} and \eqref{eq:2.8} 
by $\Irr (G^{F_\omega})_{[\mh P,\sigma]}$. In other words, these are precisely 
the irreducible quotients of $\mr{ind}_{\mh P^{F_\omega}}^{G^{F_\omega}} (\sigma)$.
The group $(\Omega^{\theta,\mh P})^*$ acts simply transitively on 
$\Irr (G^{F_\omega})_{[\mh P,\sigma]}$, by tensoring with weakly unramified characters. 
The choice of $\sigma^N$ determines an equivariant bijection
\begin{equation}\label{eq:2.10}
(\Omega^{\theta,\mh P})^* \to \Irr (G^{F_\omega})_{[\mh P,\sigma]} : \chi \mapsto 
\mr{ind}_{N_{G^{F_\omega}}(\mh P^{F_\omega})}^{G^{F_\omega}} (\chi \otimes \sigma^N) .
\end{equation}
We normalize the Haar measure on $G^{F_\omega}$ as in \cite{GrGa,HII}. Recall that the 
formal degree of $\mr{ind}_{\mh P^{F_\omega}}^{G^{F_\omega}} (\sigma)$ equals
$\dim (\sigma) / \mr{vol}(\mh P^{F_\omega})$. When $(\Omega^\theta)^*$ is finite, 
\eqref{eq:2.8} implies that
\begin{equation}\label{eq:2.21}
\mr{fdeg}(\pi) = \frac{\dim (\sigma)}{|\Omega^{\theta,\mh P}| \, \mr{vol}(\mh P^{F_\omega})} 
\quad \text{for any } \pi \in \Irr (G^{F_\omega})_{[\mh P,\sigma]} .
\end{equation}
We will make ample use of Lusztig's arithmetic diagrams $\mathsf I / \mathsf J$ 
\cite[\S 7]{Lusztig-unirep}. This means that $\mathsf I$ is the affine Dynkin diagram of 
$\mb G$ (including the action of $\mb W_K$), and that $\mathsf J$ is a $\mb W_K$-stable 
subset of $\mathsf I$. This provides a convenient way to parametrize parahoric subgroups of 
$G$ up to conjugacy. The $\mb W_K$-action on $\mathsf I$ boils down to that of the 
Frobenius element, and the maximal $\Frob$-stable subsets $\mathsf J \subsetneq \mathsf I$ 
correspond to maximal parahoric subgroups of $G^{F_\omega}$. Recall that only those parahorics
can give rise to supercuspidal unipotent $G^{F_\omega}$-representations.

The above entails that $\Irr (G^{F_\omega})_{\cusp,\unip}$ depends only on some combinatorial
data attached to $\mb G$ and $F_\omega$: the affine Dynkin diagram $\mathsf I$, the Lie
types of the parahoric subgroups of $G$ associated to the subsets of $\mathsf I$, the
group $\Omega^\theta$ and its action on $\mathsf I$.

\section{Statement of main theorem for semisimple groups}
\label{sec:mainss}

Consider a semisimple unramified $K$-group $\mb G$ with data $\mh P, \sigma$ as in \eqref{eq:2.8}. 
Theorem \ref{thm:A} and compatibility with direct products of simple groups determine a map
\begin{equation}\label{eq:1.4}
\Irr (G^{F_\omega})_{\cusp,\unip} \to (\Omega^\theta)^* \backslash \Phi_\nr^2(G^{F_\omega}),
\end{equation} 
such that the image of $\Irr (G^{F_\omega})_{[\PH,\sigma]}$ is an orbit 
$(\Omega^\theta)^* \lambda$ where $\lambda$ satisfies the requirement 
\eqref{eq:fdegisgammafactor} about formal degrees and adjoint $\gamma$-factors.

In this section we count the number of enhancements of L-parameters in \eqref{eq:1.4},
and we find explicit formulas for the numbers of supercuspidal representations in
the associated L-packets. To this end we define four numbers:
\begin{itemize}
\item $\sfa$ is the number of $\lambda' \in \Phi_\nr^2(G^{F_\omega})$ which admit 
a $G^{F_\omega}$-relevant cuspidal enhancement and for each $K$-simple factor 
$\mb G_i$ of $\mb G$ satisfy
\[
\gamma (0,\mr{Ad}_{\mb G_i^\vee} \circ \lambda', \psi) = 
c_i \, \gamma (0,\mr{Ad}_{\mb G_i^\vee} \circ \lambda,\psi)
\] 
for some $c_i \in \Q^\times$ (as rational functions of $q_K$);
\item $\sfb$ is the number of  $G^{F_\omega}$-relevant cuspidal enhancements of $\lambda$;
\item $\sfa'$ is defined as $|\Omega^{\PH,\theta}|$ times the number of $G^{F_\omega}$-conjugacy 
classes of $F_\omega$-stable maximal parahoric subgroups $\PH' \subset G$ for which there exists a 
$\sigma' \in \Irr_{\cusp,\unip} (\PH^{F_\omega})$ such that the components 
$\sigma_i, \sigma'_i$ corresponding to any $K$-simple factor $\mb G_i$ of $\mb G$ satisfy
\[
\mr{fdeg} \big( \mr{ind}_{\PH_i^{' F_\omega}}^{G^{F_\omega}} \sigma'_i \big) = 
c'_i \, \mr{fdeg} \big( \mr{ind}_{\PH_i^{F_\omega}}^{G^{F_\omega}} \sigma_i \big)
\] 
for some $c'_i \in \Q^\times$ (as rational functions of $q_K$);
\item $\sfb'$ is the number of cuspidal unipotent representations $\sigma'$ of 
$\overline{\PH}^{F_\omega}$ such that $\textup{deg}(\sigma')=\textup{deg}(\sigma)$.
\end{itemize}

\begin{lemma}\label{lem:3.1}
When $\mb{G}$ is adjoint, simple and $K$-split, the above numbers $\sfa,\sfb,\sfa',\sfb'$ 
agree with those introduced (under the same names) in \cite[6.8]{Lusztig-unirep}.
\end{lemma}
\begin{proof}
Our $\sfb'$ is defined just as that of Lusztig.

Under these conditions on $\mb{G}$, all $\mh P'$ as above are conjugate to $\mh P$, 
so $\sfa' = |\Omega^{\mh P,\theta}|$. From \cite[1.20]{Lusztig-unirep} we see that 
$\Omega^{\mh P,\theta}$ equals $\bar \Omega^u$ over there, so the two versions of $\sfa'$ agree.

With $\sfb$ Lusztig counts pairs $(\mc C, \mc F)$ consisting of a unipotent conjugacy
class in $\mc C$ in $Z_{\mb G^\vee}(\lambda (\Frob))$ and a cuspidal local system $\mc F$ on
$\mc C$, such that $Z(\mb G^\vee)$ acts on $\mc F$ according to the character defined by
$\mb G^\omega$ via the Kottwitz isomorphism \eqref{eq:2.22}. The set of such $(\mc C,\mc F)$ 
is naturally in bijection with the set of extensions of $\lambda |_{\mb W_F}$ to a 
$G^{F_\omega}$-relevant cuspidal L-parameter \cite{AMS1}. To equate Lusztig's $\sfb$
to ours, we need to show the following. Given $\mb G^\omega$ and $s = \lambda (\Frob)$, there 
exists at most one unipotent class in $Z_{\mb G^\vee}(s)$ supporting a
$G^{F_\omega}$-relevant cuspidal local system.

Recall from \cite[\S 8.2]{Ste} that $Z_{\mb G^\vee}(s)$ is a connected reductive 
complex group (because $\mb G^\vee$ is simply connected). For the existence of cuspidal local 
system on unipotent classes $Z_{\mb G^\vee}(s)$  has to be semisimple, so the semisimple element 
$s = \lambda (\Frob)$ must have finite order and must correspond to a single node in the affine 
Dynkin diagram of $\mb G^\vee$ \cite[\S 2.4]{Retorsion}. As $\mb G^\vee$ is simple, 
this implies that $Z_{\mb G^\vee}(s)$ has at most two simple factors. 

For every complex simple group which is not a (half-)spin group, there exists at most one 
unipotent class supporting a cuspidal local system, whereas for (half-)spin groups there are 
at most two such unipotent classes \cite{Lus-Intersect}. (There are two precisely when the 
vector space to which the spin group 
is associated has as dimension a square triangular number bigger than 1.) It follows that the 
required uniqueness holds whenever $\mb G^\vee$ does not have Lie type $B_n$ or $D_n$. The 
$G^{F_\omega}$-relevance of the cuspidal local system (i.e. the $Z(\mb G^\vee)$-character 
$\omega$) imposes another condition, limiting the number of possibilities even further. 
Going through all the cases \cite[\S 5.4--5.5, \S 6.7--6.11]{Mor}, or equivalently 
\cite[\S 7.38--7.53]{Lusztig-unirep}, one can see that in fact the uniqueness of unipotent classes 
holds for all simple adjoint $\mb G$. Alternatively, this can derived from Theorem \ref{thm:A}.

This uniqueness of unipotent classes also means that our $\sfa$ just counts the number
of possibilities for $\lambda \big|_{\mb W_F}$, or equivalently for $s = \lambda (\Frob)$. 
The geometric diagram in \cite[\S 7]{Lusztig-unirep} determines a unique node $v(s)$ of the 
affine Dynkin diagram $I$ of $\mb G^\vee$, and hence completely determines the image of $s$ in 
${\mb G^\vee}_\ad$. Then the possibilities for $s \in \mb G^\vee$ modulo conjugacy are parametrized by
the orbit of $v(s)$ in $I$ under the group $\Omega$ for $\mb G_\ad$, see \cite[\S 2.2]{Retorsion}
and \cite[\S 2]{Lusztig-unirep}. Since $\mb G^\vee$ is simple, this coincides with the orbit of 
$v(s)$ under the group of all automorphisms of $I$. The cardinality of the latter orbit is 
used as the definition of $\sfa$ in \cite{Lusztig-unirep}, so it agrees with our $\sfa$.
\end{proof}

Assume for the moment that $\mb{G}$ is simple (but not necessarily split or adjoint). Then 
$s \theta = \lambda (\Frob) \in \mb G^\vee \theta$ has finite order, and $s$ determines a vertex 
$v(s)$ in the fundamental domain for the Weyl group $W(\mb G^\vee, \mb S^\vee)^\theta$ acting on 
$\mb S^\vee$. The order $n_s$ of $v(s)$ is indicated by the label in the corresponding Kac diagram 
\cite{Kacbook,Retorsion}. We can also realize $v(s)$ as a node in Lusztig's geometric diagrams
\cite[\S 7]{Lusztig-unirep}. They are denoted as ``$\tilde{I}/J$'', where $\tilde I$ is a 
basis of the affine root system of the complex group $(\mb G^\vee)^\theta$. The complement of $J$ 
in $\tilde{I}$ is one node, the one corresponding to $v(s)$. We point out that $v(s)$ determines a 
unique $\mb G^\vee$-conjugacy class in ${\mb G^\vee}_\ad \theta$. Thus the geometric diagram
$J$ determines the conjugacy class of $s \theta$ up to $Z(\mb G^\vee)$. 

In first approximation, the semisimple group $\mb{G}$ is a product of simple groups, and 
thus the above yields a description of the possibilities for $\lambda (\Frob) = s \theta$,
$v(s) \in {\mb G^\vee}_\ad$ and $n_s = \mr{ord}(v(s))$.

In the setting of \eqref{eq:1.4}, let $(\Omega^\theta)^*_{\lambda}$ be the isotropy group of 
$\lambda$ in $(\Omega^\theta)^*$. We define 
\[
g = \big| (\Omega^\theta)^*_{\lambda} \big| \, [\Omega^\theta : \Omega^{\PH,\theta}]^{-1}
\quad \text{and} \quad 
g' = [ \Omega_\ad^\theta/\Omega_\ad^{\PH,\theta}:\Omega^\theta/\Omega^{\PH,\theta} ]. 
\]
We say that $\pi \in \Irr (G^{F_\omega})_{[\mh P,\sigma]}$ and $\lambda$ satisfy 
\eqref{eq:fdegisgammafactor} with respect to a $K$-simple factor $\mb G_i$ of $\mb G$ if, 
in the notations from page \pageref{eq:1.4}, there exists a $c_i \in \Q^\times$ such that
\begin{equation}\label{eq:2.14}
\mr{fdeg} (\mr{ind}_{\PH_i^{F_\omega}}^{G^{F_\omega}} \sigma_i) = 
c_i \, \gamma (0,\mr{Ad}_{\mb G_i^\vee} \circ \lambda,\psi) 
\end{equation}
as rational functions of $q_K$. Now we are ready to state our main result.

\begin{thm} \label{thm:B}  
Let $\mb G$ be an unramified semisimple $K$-group.
\begin{enumerate}
\item There exists an $(\Omega^\theta)^*$-equivariant bijection between 
$\Irr (G^{F_\omega})_{\cusp,\unip}$ and \\
$\Phi_\nr (G^{F_\omega})_\cusp$, which is equivariant with respect to diagram automorphisms,
compatible with almost direct products and matches formal degrees with adjoint 
$\gamma$-factors as in \eqref{eq:fdegisgammafactor}.
\item The set of L-parameters associated in part (1) to 
$\Irr (G^{F_\omega})_{[\mh P,\sigma]}$ is canonically determined. 
\end{enumerate}
Now we fix a $F_\omega$-stable parahoric subgroup $\PH \subset G$ and a cuspidal 
unipotent representation $\sigma$ of $\PH^{F_{\omega}}$. Let 
$\lambda \in \Phi_\nr^2(G^{F_\omega})$ be an L-parameter associated to 
$[\PH^{F_\omega}, \sigma]$ via part (1).
\begin{itemize}
\item[(3)] The $(\Omega^\theta)^*$-stabilizer of any $\pi \in \Irr (G^{F_\omega})_{\unip,\cusp}$ 
and of any $(\lambda,\rho) \in \Phi_\nr (G^{F_\omega})_\cusp$ which satisfies 
\eqref{eq:fdegisgammafactor} with respect to any $K$-simple factor of $\mb G$ is 
$( \Omega^\theta/ \Omega^{\PH, \theta})^*$. In particular 
$g = [ (\Omega^\theta)^*_{\lambda} : ( \Omega^\theta/ \Omega^{\PH, \theta})^*] \in \N$.
\item[(4)] $\sfb' = \phi (n_s)$, where $\phi$ denotes Euler's totient function. 
In particular, $\phi (n_s)$ is identically equal to 1 for groups isogenous to 
classical groups. 
\item[(5)] We have $\mathsf{ab = a'b'}$, which is equal to the total number of 
supercuspidal unipotent representations $\pi$ satisfying \eqref{eq:fdegisgammafactor}
with respect to any $K$-simple factor of $\mb G$, for this $\lambda$. Furthermore $\sfa = 
[(\Omega^\theta)^*:(\Omega^\theta)^*_{\lambda}]$, $\sfa' = g'\, |\Omega^{\PH,\theta}|$, 
and thus $\sfb = gg' \phi (n_s)$. 
\item[(6)] The number of $(\Omega^\theta)^*$-orbits on the set of $\pi \in 
\Irr (G^{F_\omega})_{\unip,\cusp}$ satisfying \eqref{eq:fdegisgammafactor} is
$g' \phi (n_s)$. These orbits can be parametrized by $G^{F_\omega}$-conjugacy classes
of pairs $(\mh P,\sigma)$, or (on the Galois side) by cuspidal enhancements of $\lambda$ 
modulo $(\Omega^\theta)^*_\lambda$.
\end{itemize}
\end{thm}

In \eqref{eq:2.10} we saw that $\Irr (G^{F_\omega})_{[\mh P,\sigma]}$ can be parametrized
with the group $(\Omega^{\theta,\mh P})^*$. By \eqref{eq:2.2} that is a quotient of 
$Z(\mb G^\vee)_\theta$, and via \eqref{eq:2.3} it acts naturally on the set of involved
L-parameters. Thus part (2) can also be formulated as: the L-parameter of any  $\pi \in 
\Irr (G^{F_\omega})_{[\mh P,\sigma]}$ is canonically determined up to the action of
$(\Omega^\theta )^* \cong Z(\mb G^\vee)_\theta$.

In the upcoming nine sections we will collect the data that are needed to establish 
Theorem \ref{thm:B} for simple adjoint groups and cannot readily be found in the literature 
yet. In particular this concerns the behaviour under diagram automorphisms of reductive
$p$-adic groups. The actual proof for adjoint groups is written down in Section \ref{sec:adjoint}.

\section{Inner forms of projective linear groups}
\label{sec:innAn}

We consider $\mb G = PGL_n$, of adjoint type $A_{n-1}$. Then $\mb G^\vee = SL_n (\CC)$,
$\Omega^* = Z (\mb G^\vee) \cong \ZZ / n \ZZ$ and $\Omega = \mr{Irr}(Z(\mb G^\vee))$.

Cuspidal unipotent representations of $G^{F_\omega}$ can exist only if $\mathsf J \subset
\widetilde{A_{n-1}}$ is empty and $\omega \in \Omega$ has order $n$. Then $G^{F_\omega}$ is
an anisotropic form of $PGL_n (K)$, so isomorphic to $D^\times / K^\times$ where $D$
is a division algebra of dimension $n^2$ over $Z(D) = K$.

The parahoric $\PH^{F_\omega}$ is the unique maximal compact subgroup of $G^{F_\omega}$, 
so $\Omega^\PH = \Omega$ and $\sfa = |\Omega^\PH| = n$.
The cuspidal unipotent representations of $G^{F_\omega}$ are precisely its weakly
unramified characters. There are $n$ of them, naturally parametrized by $Z(\mb G^\vee)$ via 
the LLC. Hence $\sfa' \sfb' = n$ and $\sfb' = 1$. 

The associated Langlands parameter $\lambda$ sends Frob to an
element of $Z(\mb G^\vee)$, while $u_\lambda$ is a regular unipotent element of $\mb G^\vee$.
Hence $\mc A_\lambda = Z(\mb G^\vee)$, which supports exactly one cuspidal local system
relevant for $G^{F_\omega}$, namely $\omega \in \mc A_\lambda^*$. In particular $\sfb = 1$.
The group $(\Omega^\theta)^* = Z(\mb G^\vee)$ acts simply transitively on 
$\Phi_\nr (G^{F_\omega})_\cusp$, so $\sfa = n$ and 
\[
(\Omega^\theta)^*_\lambda = 1 = (\Omega^\theta / \Omega^{\theta,\mh P})^*.
\]
Let $\tau$ be the unique nontrivial automorphism of $A_{n-1}$. It acts on $G$ and
$\mb G^\vee$ by the inverse transpose map, composed with conjugation by a suitable matrix $M$.
Consequently $\tau (\lambda,\omega)$ is equivalent with $(\lambda^{-T},\omega^{-1})$.
On the $p$-adic side $\tau$ sends $g \in G^{F_\omega}$ to $M g^{-T} M^{-1} \in 
G^{F_{\omega^{-1}}}$. Thus $\tau$ sends a weakly unramified character $\chi$ of $G^{F_\omega}$ 
to $\chi^{-1} \in \mr{Irr}(G^{F_{\omega^{-1}}})$. If $(\chi,G^{F_\omega})$ corresponds to
$(\lambda,\omega)$, then $(\chi^{-1},G^{F_{\omega^{-1}}})$ corresponds to
$(\lambda^{-T},\omega^{-1})$. This says that the LLC is $\tau$-equivariant in this case.

\section{Projective unitary groups}
\label{sec:outAn}

Take $\mb G = PU_n$, of adjoint type ${}^2 \! A_{n-1}$, with $\mb G^\vee = SL_n (\CC)$. 
Now $\theta = \tau$ is the unique nontrivial diagram automorphism of $A_{n-1}$. When $n$ is 
odd, the groups $\Omega^\theta, \Omega_\theta, (\Omega^*)^\theta$ and $(\Omega^\theta)^*$ 
are all trivial. When $n$ is even, 
\[
\Omega^\theta = \{1, z \mapsto z^{n/2} \}, \Omega_\theta = \Omega / \Omega^2,
(\Omega^*)^\theta = Z(\mb G^\vee)^\theta = \{1,-1\}, (\Omega^\theta)^* = 
Z(\mb G^\vee) / Z(\mb G^\vee)^2
\]
and all these groups have order 2. When $n$ is even, the nontrivial element of 
$\Omega^\theta$ acts on $\widetilde{{}^2 \! A_{n-1}}$ by a rotation of order 2.

When $n$ is not divisible by four, there is a canonical way to choose the $\omega \in \Omega$
defining the inner twist, namely $\omega \in \Omega^\theta$. When $n$ is divisible by four,
the non-quasi-split inner twist $G^{F_\omega}$ cannot be written with a $\theta$-fixed $\omega$.
For that group we just pick one $\omega \in \Omega \setminus \Omega^2$. Then the diagram
automorphism $\tau$ sends $G^{F_\omega}$ to $G^{F_{\omega^{-1}}}$, a different group which
counts as the same inner twist. So equivariance with respect to diagram automorphisms 
is automatic, unless $n$ is congruent to 2 modulo 4.\\

The subset $\mathsf J \subset \widetilde{{}^2 \! A_{n-1}}$ has to consist of two (possibly 
empty) $F_\omega$-stable subdiagrams ${}^2 \! A_s$ and ${}^2 \!A_t$, with $s+t+2 = n$
(or $s + 1 = n$ if $t=0$ and $n$ is even). The analysis
depends on whether or not $s$ equals $t$, so we distinguish those two possibilities.\\

\textbf{The case $\mathsf J = {}^2 \! A_s {}^2 \!A_t$ with $s \neq t$}

When $n$ is odd, no parahoric subgroup associated to another subset of $\widetilde{{}^2 \! 
A_{n-1}}$ gives rise to a cuspidal unipotent representation with the same formal degree as 
that coming from $\mathsf J$. When $n$ is even, the parahoric subgroup associated to
$\mathsf J' = {}^2 \! A_t \, {}^2 \!A_s$ does have such a cuspidal unipotent representation, 
and the subsets $\mathsf J, \mathsf J'$ of $\widetilde{{}^2 \! A_{n-1}}$ form one orbit for 
$\Omega^\theta$. This leads to $\sfa' = |\Omega^{\theta,\mh P}|  = 1$.

The group $G^{F_\omega}$ has only one cuspidal unipotent 
representation with the given formal degree, so that one is certainly fixed by $\tau$. 

The cuspidal enhancements of $\lambda$ are naturally in bijection with the 
cuspidal local systems supported on unipotent classes in $Z_{SL_n (\CC)}(\lambda (\Frob))$.
The centralizer of the semisimple element $\lambda (\Frob) = y \theta \in {}^L G$ 
in $SL_n (\CC)$ is the classical group associated to the bilinear form given by $y$
times the antidiagonal matrix with entries 1 on the antidiagonal. This implies an isomorphism
\begin{equation}\label{eq:5.1}
Z_{SL_n (\CC)}(\lambda (\Frob)) \cong Sp_{2q}(\CC) \times SO_p (\CC) ,
\end{equation}
where the Lie type depends on the index of the bilinear form and can be read off from 
\cite[\S 11.2--11.3]{Lusztig-unirep2}. The unipotent element $u_\lambda$ is given in 
\cite[4.7.(i)]{FeOp}: it has Jordan blocks $2,4,6,\ldots$ and $1,3,5,\ldots$.

To get $\mc A_\lambda$, we have to add $Z(SL_n (\CC))$
to \eqref{eq:5.1}, and then to take the centralizer of $\lambda (SL_2 (\CC))$. The 
inclusion of $Z(SL_n (\CC))$ does not make a difference, because in \eqref{eq:2.9} we
already fixed the restriction of representations of $\mc A_\lambda$ to that group. Since 
both $Sp_{2q}(\CC)$ and $SO_p (\CC)$ admit at most one cuspidal pair $(u,\rho)$ 
\cite[\S 10]{Lus-Intersect}, $\lambda$ has at most one cuspidal enhancement. 
In other words, $\sfb = 1$. 

When $n$ is odd, Theorem \ref{thm:A} produces a unique L-parameter.

When $n$ is even, Theorem \ref{thm:A} gives one or two L-parameters. The action of 
$(\Omega^\theta )^* \cong Z(\mb G^\vee) / Z(\mb G^\vee)^2$ on L-parameters is by multiplying 
$\lambda (\Frob)$ with an element of $Z(\mb G^\vee)$. 
An element $z I_n \in Z(\mb G^\vee) \setminus Z(\mb G^\vee)^2$ can be written as
$(1 - \theta)(z^{1/2} U)$, where $U \in GL_n (\CC)^\theta$ has determinant $z^{-n/2} = -1$. 
When \eqref{eq:5.1} contains a nontrivial special orthogonal group, we can choose $U$ in 
$Z_{GL_n (\CC)}(\lambda (\Frob))$, which shows that $z \lambda (\Frob)$ is conjugate to 
$\lambda (\Frob)$ within \eqref{eq:5.1}. By \cite[\S 11.3]{Lusztig-unirep2}, this condition 
on \eqref{eq:5.1} is equivalent to $s \neq t$ (and $n$ even), which we already assumed here. 
With Theorem \ref{thm:A} it follows that in that case there is only one L-parameter with the 
required adjoint $\gamma$-factor. Notice that here
\[
(\Omega^\theta / \Omega^{\theta,\mh P})^* = (\Omega^\theta)^* = (\Omega^\theta)^*_\lambda 
\quad \text{and} \quad \sfa = \sfb = 1 = \sfa' = \sfb'.
\]
\textbf{Remark.}
When $n$ is even, some groups isogenous to $\mb G = PU_n$ have trivial $\Omega^\theta$, for instance 
$\mb H = SU_n$. In other words, the image of $H^{F_\omega} \to G^{F_\omega}$ does not contain
representatives for the nontrivial element of $\Omega^\theta$. For $s \neq t$,
the pullback of the $G^{F_\omega}$-representation $\pi$ associated to 
$\mathsf J = {}^2 \! A_s \, {}^2 \!A_t$ to $H^{F_\omega}$ decomposes as a direct sum of two 
irreducible representations, associated to $\mathsf J$ and to 
$\mathsf J' = {}^2 \! A_t \, {}^2 \!A_s$. Since $\mathsf J$ and $\mathsf J'$ are stable under
$\tau$, $\tau$ stabilizes both these $H$-representations. 

The isotropy group of $\lambda$ as a L-parameter $\lambda_H$ for $H^{F_\omega}$ is bigger than 
for $G^{F_\omega}$, for $z \lambda_{\mb H} = \lambda_{\mb H}$ and elements of $SL_n (\CC)$ 
which send $\lambda$ to $z \lambda$ also stabilize $\lambda_{\mb H}$. From the above we see 
that one such new element in the isotropy group is $z^{1/2} U$, where $U \in O_p (\CC) 
\setminus SO_p (\CC)$ and $p = n -2q \in 2 \ZZ_{>0}$. Thus \eqref{eq:5.1} becomes
\begin{equation}\label{eq:5.2}
Z_{SL_n (\CC)}(\lambda_{\mb H} (\Frob)) \cong Sp_{2q}(\CC) \times O_p (\CC) .
\end{equation}
The group $\mc A_{\lambda_{\mb H}}$ can be obtained from \eqref{eq:5.2} in the same way as 
described after \eqref{eq:5.1}. The group \eqref{eq:5.2} has precisely two cuspidal pairs, 
which should be matched with the two direct summands of the pullback of $\pi$. Note that the 
action of $\tau$ on \eqref{eq:5.1} is (up to some inner automorphism) the unique nontrivial 
diagram automorphism of that group. In $Sp_{2q}(\CC) \times O_p (\CC)$ that diagram automorphism 
becomes inner, which implies that $\tau$ fixes both cuspidal pairs for this group. In particular, 
the aforementioned matching of these with $H^{F_\omega}$-representation is automatically 
$\tau$-equivariant.\\

\textbf{The case $\mathsf J = {}^2 \! A_s {}^2 \!A_s$ (with $2s + 2 = n$)}

Now $\Omega^{\theta,\mh P} = \Omega^\theta$ is nontrivial and $\sfa' = 2$. There
are two cuspidal unipotent representations containing $\sigma$, parametrized by the two
extensions $\sigma_1, \sigma_2$ of $\sigma$ to $N_{G^{F_\omega}}(\PH^{F_\omega})$. Then
$\sigma_1 (g) = -\sigma_2 (g)$ for all $g \in N_{G^{F_\omega}}(\PH^{F_\omega}) \setminus 
\PH^{F_\omega}$. 

Consider the action of $\tau = \theta$ on $G$. We may take it to be the action of $F$,
only without the Frobenius automorphism of $K_\nr/K$. It stabilizes $G^{F_\omega}$, unless
$n$ is divisible by four and $G^{F_\omega}$ is not quasi-split (a case we need not consider,
for there equivariance with respect to diagram automorphisms is automatic). 
Then \eqref{eq:actionTau} shows that the action of $\tau$ on $G^{F_\omega}$ reduces to the 
action of this Frobenius automorphism on the matrix coefficients. 

Since $N_{G^{F_\omega}}(\PH^{F_\omega}) \setminus \PH^{F_\omega}$ contains 
$\tau$-fixed elements (they are easy to find knowing the explicit form of $\tau$), $\tau$ 
fixes $\sigma_1$ and $\sigma_2$. Thus $\tau$ fixes both cuspidal unipotent representations 
under consideration.

The Jordan blocks of $u_\lambda$ are again given in \cite[4.7.(i)]{FeOp}.
The same reasoning as in the case $\mathsf J = {}^2 \! A_s {}^2 \!A_t$ with $s \neq t$
shows that the L-parameters $\lambda$ and $z \lambda$ are not equivalent and that
$Z_{SL_n (\CC)}(\lambda (\Frob)) \cong Sp_n (\CC)$ admits just one cuspidal pair.
Hence $\sfb = 1$,  $(\Omega^\theta)^*_\lambda = 1$ and $(\Omega^\theta)^*_\lambda = 
(\Omega^\theta / \Omega^{\theta,\mh P})^*$. As $\sfb' = 1$, we conclude that
\[
\sfa = \sfb \sfa = |(\Omega^\theta)^* \lambda| = 2 = |\Omega^{\theta,\mh P}| = 
\sfa' = \sfa' \sfb'.
\]
We can take for $y = \lambda (\Frob) \theta^{-1}$ the diagonal matrix with alternating 1 and
$-1$ on the diagonal. Considering the eigenvalues of $y$ and $\theta (y)$, it is clear that 
$\theta (\lambda (\Frob)) = \theta (y) \theta$ and $z \lambda (\Frob) = z y \theta$ are not
conjugate. So $\tau$ fixes both these L-parameters.\\

We checked that the diagram automorphism $\tau = \theta$ fixes all L-parameters
under consideration in this section. Every such L-parameter has only one cuspidal 
enhancement. Hence $\tau$ fixes everything on the Galois side, which means that 
our LLC is $\tau$-equivariant for the representations in this section.

\section{Odd orthogonal groups}
\label{sec:innBn}

Here $\mb G = SO_{2m+1} = PSO_{2m+1}$, of type $B_m$. Now $\mb G^\vee = Sp_{2m}(\CC)$ and 
$|\Omega^\theta| = |\Omega| = 2$. From \cite[\S 5.3]{Mor} or \cite[\S 7]{Lusztig-unirep} 
we see that $\mathsf J = D_s B_t$ and hence
\[
\Omega^{\theta,\mh P} = \left\{ \begin{array}{ll}
\Omega^\theta & \text{if } s > 0,\\
1 & \text{if } s = 0.
\end{array}
\right.
\]
Further Lusztig's geometric diagram $J$ has two (possibly empty) components of type 
$C_{n_\pm}$. Then 
\[
Z_{\mb G^\vee}(\lambda (\Frob)) \cong Sp_{2 n_+}(\CC) \times Sp_{2 n_-}(\CC),
\] 
and this determines $\lambda (\Frob))$ up to $Z(\mb G^\vee)$. The L-parameter $\lambda$ is described 
in \cite[\S 5.3, \S 6.6]{Mor} and \cite[4.7 (ii)]{FeOp}: the unipotent class in $Sp_{2 n_\pm}(\CC)$ 
has Jordan blocks of sizes $2, 4, 6 \ldots$, which forces $n_\pm$ to be a square. One observes that
\[
(\Omega^\theta )^*_\lambda = \left\{ \begin{array}{ll}
(\Omega^\theta )^* & \text{if } n_+ = n_-,\\
1 & \text{if } n_+ \neq n_- .
\end{array}
\right.
\]
By \cite[7.54--7.56]{Lusztig-unirep} $\sfb = \sfb' = 1$ and $s=0$ is equivalent to $n_+ = n_-$.
We conclude that
\[
\sfa \sfb = \sfa = |(\Omega^\theta )^*_\lambda| = |(\Omega^\theta / \Omega^{\theta,\mh P})^*| 
= \sfa' = \sfa' \sfb' .
\]

\section{Symplectic groups}
\label{sec:innCn}

We consider $\mb G = PSp_{2n}$, the adjoint group of type $C_n$. The group $\Omega=\Omega^\theta$ 
has two elements and 
\[
(\Omega^\theta )^* = Z(\text{Spin}_{2n+1}(\CC)) = \{1,-1\} .
\]
The subset $\mathsf J \subset \widetilde{C_n}$ can be of three kinds.\\

\textbf{The case $\mathsf J = C_s C_t$ with $s \neq t$}

Here $\Omega^{\theta,\mh P} = 1$, so $\omega = 1$ and $\sfa' = 1$. 
By \cite[7.48--7.50]{Lusztig-unirep} $\sfb = \sfb' = 1$ and the geometric diagram is of type 
$D_p B_q$ with $p > 0$. More precisely, \cite[\S 5.4]{Mor} says that
\[
Z_{\mb G^\vee}(\lambda (\Frob)) \cong \big( \text{Spin}_{2p}(\CC) \times 
\text{Spin}_{2q+1}(\CC) \big) \big/ \langle (-1,-1) \rangle .
\]
The unipotent class from $\lambda$ is given in \cite[4.7 (iii)]{FeOp}: it has Jordan blocks of
sizes $1,3,\ldots,2N_p - 1$ and $1,3,\ldots,2N_q - 1$, where $N_p^2 = 2p$ and $N_q^2 = 2q+1$.
This shows that $(\Omega^\theta )^*_\lambda = (\Omega^\theta)^* = 
(\Omega^\theta / \Omega^{\theta,\mh P})^*$ and $\sfa = 1$.\\

\textbf{The case $\mathsf J = C_s C_s$ (with $2s = n$)} 

Now $\Omega^{\theta,\mh P} = \Omega^\theta$, $\sfa' = 2$ and $\omega$ can be both elements
of $\Omega$. The geometric diagram has type $B_q$ and one checks that
$(\Omega^\theta )^*_\lambda = 1$. (This corrects \cite[\S 7.50]{Lusztig-unirep}.) 

The group $Z_{\mb G^\vee}(\lambda (\Frob))$ is just $\mb G^\vee = \text{Spin}_{2n+1}(\CC)$.
By \cite[\S 14]{Lus-Intersect} it has (at most) one cuspidal pair on which
$Z(\mb G^\vee)$ acts as $\omega$, so $\sfb = 1$. Thus
\[
\sfa \sfb = \sfa = 2 = \sfa' = \sfa' \sfb' .
\]

\textbf{The case $\mathsf J = C_s {}^2 A_t C_s$ with $t>0$}

Here $\omega$ must be nontrivial. Now $\Omega^{\theta,\PH} = \Omega^\theta$ and $\sfb' = 1$, 
so $\sfa' = \sfa' \sfb' = 2$. Also, $(\Omega^\theta / \Omega^{\theta,\mh P})^*$ 
is automatically contained in $(\Omega^\theta )^*_\lambda$.

By \cite[\S 7.51--7.53]{Lusztig-unirep} the geometric diagram is of type $D_p B_q$ 
with $p,q > 0$. The L-parameters are given explicitly in \cite[\S 6.7]{Mor} and 
\cite[4.7 (v)]{FeOp}. The group $(\Omega^\theta )^*$ stabilizes $\lambda$ and $\sfa = 1$. Here
\[
Z_{\mb G^\vee}(\lambda (\Frob)) \cong \big( \text{Spin}_{2p}(\CC) \times \text{Spin}_{2q+1}(\CC) \big)
\big/ \langle (-1,-1) \rangle .
\]
The unipotent element $u_\lambda$ has two factors, both with Jordan blocks of the types
$1,5,9,\ldots$ or $3,7,11,\ldots$. Its conjugacy class only admits cuspidal local systems on 
which both $-1 \in \text{Spin}_{2p}(\CC)$ and $-1 \in \text{Spin}_{2q+1}(\CC)$ act nontrivially. 
From \cite[\S 14]{Lus-Intersect} we know that there are precisely two such cuspidal local
systems, differing only by the action of $Z(\text{Spin}_{2p}(\CC))$. Hence $\sfb = 2$.

Let us also look at the action of $(\Omega^\theta )^*_\lambda$ on this pair of
enhancements of $\lambda$. For this we need to exhibit a $g \in \mb G^\vee$ such that
$g \lambda (\Frob) g^{-1} = -1 \cdot \lambda (\Frob)$. For that we can look at the
$\mb G^\vee$-centralizer of the image $v(s)$ of $s = \lambda (\Frob)$ in $\mb G^\vee / \{1,-1\} = 
SO_{2n+1}(\CC)$. As
\[
Z_{SO_{2n+1}(\CC)}(v(s)) \cong S (O_{2p}(\CC) \times O_{2q+1}(\CC)) ,
\]
we find
\[
Z_{\mb G^\vee}(v(s)) \cong S \big( \text{Pin}_{2p}(\CC) \times \text{Pin}_{2q+1}(\CC) \big) 
\big/ \langle (-1,-1) \rangle .
\]
The required $g$ must lie in $Z_{\mb G^\vee}(v(s)) \setminus Z_{\mb G^\vee}(s)$, so its image in
$\mr{Pin}_{2p}(\CC)$ does not lie in $\mr{Spin}_{2p}(\CC)$. Therefore conjugation by $g$
is an outer automorphism of $Z_{\mb G^\vee}(\lambda (\Frob))$. Every outer automorphism of
$\mr{Spin}_{2q}(\CC)$ acts nontrivially on the centre of that group (but fixes $-1$), 
and hence exchanges the two cuspidal local systems supported by the unipotent class from
$\lambda$. Thus $(\Omega^\theta )^*_\lambda$ acts transitively on the set of relevant
cuspidal enhancements of $\lambda$.

\section{Inner forms of even orthogonal groups}
\label{sec:innDn}
 
We consider $\mb G = PSO_{2n}$, of adjoint type $D_n$. 
Then $\mb G^\vee = \mr{Spin}_{2n}(\CC)$ and
\[
\Omega^* = Z(\mr{Spin}_{2n}(\CC)) = 
\left\{ \begin{array}{cc}
(\ZZ / 2 \ZZ )^2 & n \text{ even,} \\
\ZZ / 4 \ZZ & n \text{ odd.} \\
\end{array} \right.
\]
Let $\tau$ be the standard diagram automorphism of $D_n$ of order 2.
Then $(\Omega^* )^\tau = \{1,-1\}$ is the kernel of the projection
$\mr{Spin}_{2n}(\CC) \to SO_{2n}(\CC)$. Apart from that $\Omega^*$ contains elements $\epsilon$ 
and $-\epsilon$. In the associated Clifford algebra, $\epsilon$ is the product of the
elements of the standard basis of $\CC^{2n}$.

We write $\Omega = \mr{Irr} (\Omega^*) = \{1,\eta,\rho,\eta \rho\}$, where $\eta$ is fixed 
by $\tau$ and $\eta (-1) = 1$. Furthermore we decree that $\rho (\epsilon) \neq 1$. So 
$\rho$ has order 2 if $n$ is even and order 4 
if $n$ is odd, while $\tau$ interchanges $\rho$ and $\eta \rho$. The action of $\Omega \rtimes
\{1,\tau\}$ on the affine Dynkin diagram $\widetilde{D_n}$ can be pictured as

\includegraphics[width=10cm]{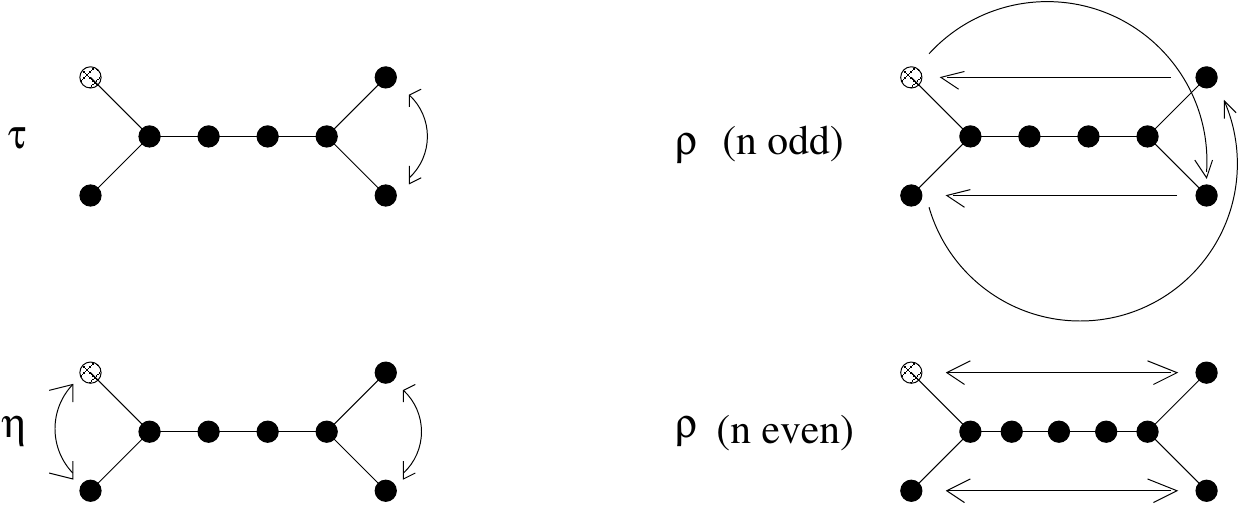}

To check $\tau$-equivariance, the following elementary lemma is useful.

\begin{lemma}\label{lem:tauOmega*}
Let $n$ be even and let $X$ be a set with a simply transitive $\Omega^*$-action. 
Suppose that $\{1,\tau\}$ acts on $X$, $\Omega^*$-equivariantly in the sense that
$\tau (\lambda x) = \tau (\lambda) \tau (x)$ for all $x \in X, \lambda \in \Omega^*$. 
Then $X \cong \Omega^*$ as $\Omega^* \rtimes \{1,\tau\}$-spaces.
\end{lemma}
\begin{proof}
First we show that $\tau$ fixes a point of $X$. Take any $x \in X$ and consider
$\tau (x) \in X$. If $\tau (x) = x$, we are done. When $\tau (x) = -x$, the element
$\epsilon x$ is fixed by $\tau$, for
\[
\tau (\epsilon x) = \tau (\epsilon) \tau (x) = -\epsilon \cdot -x = \epsilon x .
\]
Suppose that $\tau (x) = \epsilon x$. We compute
\[
x = \tau (\epsilon x) = \tau (\epsilon) \tau (x) = -\epsilon \cdot \epsilon x ,
\]
so $\epsilon^2 = -1$. But $\Omega^* \cong (\ZZ / 2 \ZZ)^2$ since $n$ is even, so we
have a contradiction. For similar reasons $\tau (x) = -\epsilon x$ is impossible. 

Thus $X$ always contains a $\tau$-fixed point, say $x_0$. Then the map
\[
\Omega^* \to X : \lambda \mapsto \lambda x_0
\]
is an isomorphism of $\Omega^* \rtimes \{1,\tau\}$-spaces. 
\end{proof}

For the group $PSO_{2n}$ there are five different kinds of subsets $\mathsf J$ of 
$\widetilde{D_n}$ which can support cuspidal unipotent representations.\\

\textbf{The case $\mathsf J = D_n$}

Here $\Omega^{\mh P} = 1$ and $\sfa' = 1$. We must have $\omega = 1$, for otherwise
$\mh P$ cannot be $F_\omega$-stable. There are four ways to embed $\mathsf J$ in
$\widetilde D_n$, and they are all associate under $\Omega$. 

By \cite[\S 7.40]{Lusztig-unirep} the geometric diagram has type $D_p D_p$, so
$n$ is even. The element $s = \lambda (\Frob)$ is a lift of the diagonal matrix 
$I_n \oplus -I_n \in SO_{2n}(\CC)$ in $\mathrm{Spin}_{2n}(\CC)$. It follows that 
$(\Omega)^*_\lambda = 1$ and
\begin{equation}\label{eq:8.1}
Z_{\mb G^\vee}(s) = \big( \mr{Spin}_n (\CC) \times \mr{Spin}_n (\CC) \big) \big/ \langle (-1,-1) \rangle .
\end{equation}
By \cite[4.7.iv]{FeOp}, $u_\lambda$ has Jordan blocks of sizes $1,3,\ldots, 2 \sqrt{n} - 1$,
in both almost direct factors $\mr{Spin}_n (\CC)$. The group \eqref{eq:8.1} has (at most) 
one cuspidal pair on which $Z (\mb G^\vee)$ acts as 1, so $\sfa = 1$ and $\sfb = 1$. \\

\textbf{Remark.} 
Let us rename $PSO_{2n}$ as $\mb G_\ad$, and investigate what happens when it is replaced
by an isogenous group $\mb G$, which in particular can be the simply connected cover
$\mb G_\SC = \mr{Spin}_{2n}$. In this remark we will endow objects associated to $\mb G_\ad$
with a subscript ad.

As $\Omega_\SC = 1$, the four elements of $\Omega_\ad \cdot \mathsf J$ define four 
non-conjugate $F_\omega$-stable parahoric subgroups of $G_\SC^{F_\omega}$. Hence the 
pullback of the unique $\pi \in \Irr (G_\ad^{F_\omega})_{[\mh P,\sigma]}$ from above 
to $G_\SC^{F_\omega}$ decomposes as a direct sum of four irreducible representations, 
parametrized by the four elements of $\Omega_\ad \cdot (\mh P,\sigma)$ or, equivalently, 
by the four $\Omega_\ad$-associates of $\mh P$. We note that the diagram automorphism
$\tau$ fixes two of these $(\mh P,\sigma)$ and exchanges the other two.

For $\mb G = SO_{2n}$ we find two direct summands of $\pi$, parametrized by 
$\{\mh P, \rho \mh P\}$ and both $\tau$-stable. For $\mb G$ a half-spin group of rank $n$,
$\pi$ also becomes a direct sum of two irreducible representations upon pulling back to 
$G^{F_\omega}$. Then they are parametrized by $\{\mh P, \eta \mh P\}$. The diagram 
automorphism $\tau$ exchanges these two half-spin groups, so it does not extend to
an automorphism of the (absolute) root datum of such a group.

On the Galois side, the above $(\lambda_\ad,\rho_\ad)$ determines a single L-parameter 
$\lambda_\SC$ for $G_\SC^{F_\omega}$. The centralizer of $\lambda_\SC (\Frob)$ is larger 
than that of $\lambda_\ad (\Frob)$:
\begin{equation}\label{eq:8.2}
Z_{{\mb{G}_\ad}^\vee} (\lambda_\SC (\Frob)) = \langle w \rangle S \big( \mr{Pin}_n (\CC) 
\times \mr{Pin}_n (\CC) \big) \big/ \langle (-1,-1) \rangle ,
\end{equation}
where $w \in \mr{Spin}_{2n}(\CC)$ is a lift of $\left( \begin{smallmatrix}
0 & I_n \\ -I_n & 0  \end{smallmatrix} \right) \in SO_{2n}(\CC)$. Since $G^{F_\omega}$ 
is $K$-split, it suffices to consider enhancements of $\lambda_\SC$ that are trivial on 
$Z(\mb G^\vee)$. The component group of $\lambda_\SC$ for $G^{F_\omega}$ is identified as
\begin{multline*}
Z_{{\mb{G}_\ad}^\vee} (\lambda_\SC) / Z ({\mb{G}_\ad}^\vee) \cong \langle w \rangle \ltimes 
S \big( (\ZZ / 2 \ZZ)^n \times (\ZZ / 2 \ZZ)^n \big) \big/ Z(SO_{2n}(\CC)) \supset  \\ 
Z_{{\mb{G}_\ad}^\vee} (\lambda_\ad) / Z ({\mb G_\ad}^\vee) \cong S \big( (\ZZ / 2 \ZZ)^n \big) 
\times S \big( (\ZZ / 2 \ZZ)^n \big) \big/ Z(SO_{2n}(\CC)) ,
\end{multline*}
where $w$ now has order two and a capital S indicates the subgroup of elements that can be 
realized by an element of a Spin group (not just in a Pin group). The component group
for $\lambda_\ad$ as a L-parameter $\lambda$ for $\mb G = SO_{2n}$ lies inbetween the above two:
\begin{equation}\label{eq:8.5}
Z_{{\mb G_\ad}^\vee} (\lambda) / Z ({\mb G_\ad}^\vee) \cong S \big( (\ZZ / 2 \ZZ)^n\times 
(\ZZ / 2 \ZZ)^n \big) \big/ Z(SO_{2n}(\CC)) .
\end{equation}
It is known that $\rho_\ad$ is the unique alternating character of $\mc A_{\lambda_\ad}$ and of 
$Z_{\mb{G}_\ad^{\, \vee}} (\lambda_\ad) / Z(\mb{G}_\ad^{\,\vee})$. It can be extended in two ways 
to an enhancement $\rho$ of $\lambda$, a representation of \eqref{eq:8.5}. Since $\tau$ fixes
\eqref{eq:8.5} pointwise, it also fixes $\rho$. In particular we can match these two 
$\rho$'s with the set $\{\mh P, \eta \mh P\}$ (from the $p$-adic side for $\mb G = SO_{2n}$) 
in a $\tau$-equivariant way.

Both these extensions $\rho$ are symmetric with respect to the two almost direct factors, 
so they are stabilized by $w$. With Clifford theory it follows that $\rho_\ad$ can be extended in 
precisely four ways to a representation of $Z_{{\mb G_\ad}^\vee} (\lambda_\SC) / Z ({\mb G_\ad}^\vee)$, 
and hence to a $\rho_\SC \in \Irr (\mc A_{\lambda_\SC})$. These four extensions differ only by 
characters of
\begin{equation}\label{eq:8.4}
Z_{{\mb G_\ad}^\vee} (\lambda_\SC) / Z_{{\mb G_\ad}^\vee} (\lambda_\ad) \cong (\ZZ / 2 \ZZ)^2 .
\end{equation}
The group \eqref{eq:8.4} is isomorphic to $\Omega_\ad^*$, by mapping 
$z \in \Omega_\ad^*$ to a $g \in Z_{{\mb G_\ad}^\vee} (\lambda_\SC)$ with 
\[
g \lambda_\ad (\Frob) g^{-1} = z \lambda_\ad (\Frob).
\]
By Lemma \ref{lem:tauOmega*} the set of enhancements $\rho_\SC$ of $\lambda_\SC$ is 
$\Omega^* \rtimes \{ 1,\tau \}$-equivariantly in bijection with the $\Omega$-orbit of $\mh P$.

For $\mb G$ a half-spin group of rank $n$ and $\lambda_\ad$ considered as a L-parameter 
$\lambda$ for $G^{F_\omega}$, $Z_{{\mb G_\ad}^\vee}(\lambda)$ is an index two subgroup of 
$Z_{{\mb G_\ad}^\vee}(\lambda_\ad)$, which contains $Z_{{\mb G_\ad}^\vee}(\lambda_\SC)$ and
differs from \eqref{eq:8.3}. So $\rho_\ad$ can be extended in two ways to an enhancement
$\rho$ of this $\lambda$. We note that $\tau$ maps $(\lambda,\rho)$ to an enhanced
L-parameter for the other half-spin group of rank $n$.\\

\textbf{The case $\mathsf J = D_s \, D_t$ with $s,t > 2$ and $s \neq t$}

Here $\Omega^{\mh P} = \{1,\eta\}$ and the $F_\omega$-stability of $\mh P$ forces
$\omega \in \{1,\eta\}$. In particular $\sfa' = 2$ and $\sfb' = 1$.

Now $[N_{G^{F_\omega}}(\PH^{F_\omega}):\PH^{F_\omega}] = 2$ and there are precisely two 
extensions of $\sigma$ from $\PH^{F_\omega}$ to $N_{G^{F_\omega}}(\PH^{F_\omega})$. They differ 
by a sign on $N_{G^{F_\omega}}(\PH^{F_\omega}) \setminus \PH^{F_\omega}$. Since $\eta$ stabilizes 
$\mathsf J$ and $\PH^{F_\omega}$, $N_{G^{F_\omega}}(\PH^{F_\omega}) \setminus \PH^{F_\omega}$ 
contains elements of the form $\chi (\varpi_K)$, where $\chi \in X_* (\mathbf S)$ represents $\eta$. 
Taking $\chi = e_1$, we see that $N_{G^{F_\omega}}(\PH^{F_\omega}) \setminus \PH^{F_\omega}$ 
has $\tau$-fixed elements. Hence $\tau$ stabilizes both extensions of $\sigma$ to 
$N_{G^{F_\omega}}(\PH^{F_\omega})$.

The two Langlands parameters built from $\mathsf J$ and the unipotent class associated to 
$\sigma$ \cite[\S 5.5, \S 6.9]{Mor} differ by an element of $\Omega^*$. 
From \cite[\S 7.38--7.39]{Lusztig-unirep} we see that the geometric diagram has type $D_p D_q$ 
with $p \neq q$. The element $\lambda (\Frob)$ is a lift of $-I_{2p} \oplus I_{2q}$ to 
$\mr{Spin}_{2n}(\CC)$. It is conjugate to $-\lambda (\Frob) \in \mr{Spin}_{2n}(\CC)$ by a lift of
$-1 \oplus I_{2n-2} \oplus -1$ to $g \in \mb G^\vee$. As $\lambda (\Frob)$ is not conjugate to
$\epsilon \lambda (\Frob)$, we obtain
\[
\Omega^*_\lambda = \{1,-1\} = (\Omega^*)^\tau = (\Omega / \Omega^{\mh P})^*. 
\]
The unipotent class from $\lambda$ \cite[4.7.(iv)]{FeOp}, in the group
\begin{equation}\label{eq:8.3}
Z_{\mb G^\vee}(\lambda (\Frob)) = \big( \mr{Spin}_{2p}(\CC) \times \mr{Spin}_{2q}(\CC) \big) 
\big/ \langle (-1,-1) \rangle ,
\end{equation}
has Jordan blocks $1,3,\ldots,2 N_p - 1$ and $1,3,\ldots,2 N_q - 1$, where $N_p^2 = 2p$ and $N_q^2 = 2q$.
It only supports a (unique) cuspidal local system if $n = p+q$ is even. Then $Z(\mb G^\vee)$ acts 
as 1 if $n$ is divisible by 4 and as $\eta$ if $n \equiv 2$ modulo 4. So $\sfa = 2$ and
$\sfb = 1$. As $\tau$ fixes the above $\lambda (\Frob)$, it stabilizes both the L-parameters,
and then also their enhancements. In particular the LLC is $\tau$-equivariant in this case.\\

\textbf{Remark.} Again we work out what changes if we replace $\mb G$ by $\mb G_\SC =
\mr{Spin}_{2n}$. Any $\pi \in \Irr (G^{F_\omega})_{\cusp,\unip}$ as above decomposes a
direct sum of irreducibles upon pulling back to $G_\SC^{F_\omega}$. These are 
parametrized by $\{ \mh P, \rho \mh P \}$, the set of $G_\SC^{F_\omega}$-conjugacy classes 
of parahoric subgroups of $G$ which are $G^{F_\omega}$-conjugate to $\mh P$. Since $\tau$
stabilizes $\mh P$, it fixes all four elements of $\Irr (G_\SC^{F_\omega})$ under 
consideration.

Regarding $\lambda$ as a L-parameter $\lambda_\SC$ for $G_\SC^{F_\omega}$, we get
\[
Z_{\mb G^\vee}(\lambda_\SC (\Frob)) = S \big( \mr{Pin}_{2p}(\CC) \times \mr{Pin}_{2q}(\CC) \big) 
\big/ \langle (-1,-1) \rangle .
\]
This group admits two cuspidal pairs $(u_\lambda,\rho_\SC)$ on which $Z(\mb G^\vee)$ acts as
1 or $\eta$. Notice that $\tau$ fixes some elements of $Z_{\mb G^\vee}(\lambda_\SC) 
\setminus Z_{\mb G^\vee}(\lambda)$, for example a lift of $I_{2p-1} \oplus -I_2 \oplus I_{2q-1}$
to $\mr{Spin}_{2n}(\CC)$. Hence $\tau$ fixes all enhanced L-parameters for $G_\SC^{F_\omega}$
involved here.\\

\textbf{The case $\mathsf J = D_s \, D_s$ (with $2s = n$)}

Since the finite reductive groups of type $D_1, D_2$ and $D_3$ do not admit cuspidal
unipotent representations, we have $s \geq 4$. Then $\Omega^{\mh P} = \Omega$, so $\sfa' = 4$ 
and $\sfb' = 1$. By \cite[\S 7.41--7.42]{Lusztig-unirep} or \cite[\S 5.5, \S 6.9]{Mor} 
the geometric diagram has type $D_p$ and $\omega \in \{1,\eta\}$. (The tables 
\cite[\S 7.44--7.45]{Lusztig-unirep} cannot appear here, because of parity problems.) 
For $\lambda (\Frob)$ we can take the
unit element so $\Omega^*_\lambda = 1 = (\Omega / \Omega^{\mh P})^*$. Then $Z_{\mb G^\vee}(
\lambda (\Frob)) = \mr{Spin}_{2n}(\CC)$, which has precisely one cuspidal pair on
which $Z(\mb G^\vee)$ acts as 1 or $\eta$.  We find $\sfa = 4$ and $\sfb = 1$. As usual the 
Jordan blocks of $u_\lambda$ \cite[4.7.(iv)]{FeOp} must follow the pattern $1,3,5 \ldots$

The $\Omega^*$-orbit of $\lambda$ forms a set $X$ as in Lemma \ref{lem:tauOmega*}. The four 
extensions of $\sigma$ from $\PH^{F_\omega}$ to $N_{G^{F_\omega}}(\PH^{F_\omega})$ 
also form a set $X'$ as in Lemma \ref{lem:tauOmega*}, and we may identify it with
$\Irr (G^{F_\omega})_{[\mh P,\sigma]}$. Now Lemma  \ref{lem:tauOmega*} yields a 
$\Omega^* \rtimes \{1,\tau\}$-equivariant bijection $X \longleftrightarrow X'$, which 
fulfills all the conditions we impose on the LLC.\\

\textbf{The case $\mathsf J = {}^2 \! A_s$}

The involvement of the diagram automorphism of $A_s$ implies that $\omega = \rho$
or $\omega = \eta \rho$. These two are interchanged by $\tau$. This points to an
easy recipe to make the LLC $\tau$-equivariant in this case: construct it in some
$\Omega^*$-equivariant way for $\omega = \rho$, and then define if for $\omega = \eta 
\rho$ by imposing $\tau$-equivariance.

There are four ways to embed $\mathsf J$ in $\widetilde{D_n}$, two of them are
$F_\omega$-stable and the other two are $F_{\omega \eta}$-stable. We have
$\Omega^{\mh P} = \{1, \omega\}$, so $\sfa' = 2$ and $\sfb' = 1$.

According to \cite[\S 7.46]{Lusztig-unirep} and \cite[\S 6.10]{Mor}, just as in the 
case $\mathsf J = D_n$, $n$ is even and the geometric diagram has type $D_{n/2} D_{n/2}$.
As over there, $\Omega^*_\lambda = \Omega^*$ and $\sfa = 1$.
The group $Z_{\mb G^\vee}(\lambda (\Frob))$ is as in \eqref{eq:8.1}.
It admits two cuspidal pairs on which $Z(\mb G^\vee)$ acts as $\omega$ (so $\sfb = 2$).
Let the unipotent element $u$ be as in \cite[4.7.(vi)]{FeOp}, so with Jordan blocks following
the patterns $1,5,9,\ldots$ or $3,7,11,\ldots$. 

Assume that $\omega = \rho$. 
In terms of $\mr{Spin}_{2n}(\CC)^2$, the cuspidal pairs are of the form 
$(u \times u, \rho_1 \otimes \rho_2)$, where $\rho_1$ and $\rho_2$ differ only by 
the nontrivial diagram automorphism of $\mr{Spin}_{2n}(\CC)$. 

A lift $g \in \mr{Spin}_{2n}(\CC)$ of $I_{n-1} \oplus -I_2 \oplus I_{n-1} 
\in SO_{2n}(\CC)$ satisfies $g \lambda (\Frob) g^{-1} = -\lambda (\Frob)$. Such a $g$ 
acts by outer automorphisms on both almost direct factors of \eqref{eq:8.1}, so it 
exchanges $(u \times u, \rho_1 \otimes \rho_2)$ and $(u \times u, \rho_2 \otimes \rho_1)$. 
Thus $\Omega^*$ acts transitively on the enhancements L-parameters for this case.

The element $w$ from \eqref{eq:8.2} satisfies $w \lambda (\Frob) w^{-1} = \epsilon 
\lambda (\Frob)$. Since $n$ is even, conjugation by $w$ 
exchanges the two almost direct factors of \eqref{eq:8.1}, but nothing more. That
operation exchanges $(u \times u, \rho_1 \otimes \rho_2)$ and 
$(u \times u, \rho_2 \otimes \rho_1)$. Thus $\epsilon$ and $-1 \in \Omega^*$ act
in the same (nontrivial) way on the set of enhancements of $\lambda$. Then $-\epsilon$
fixes stabilizes both these enhancements. As $n$ is even and $\omega = \rho$, 
$(\Omega / \Omega^{\mh P})^* = \{1 , -\epsilon\}$ is precisely the isotropy group
all of enhanced L-parameters under consideration here. 

For $\omega = \eta \rho$ we would get the cuspidal pairs $(u \times u, \rho_i \otimes 
\rho_i)$, and we would find that $\Omega^*$ acts transitively on them, with isotropy 
group $\{1,\epsilon\} = (\Omega / \Omega^{\mh P})^*$.\\

\textbf{The case $\mathsf J = D_t \, {}^2 \! A_s \, D_t$ with $s,t>1$}

Here $2t + s = n -1$. As for $\mathsf J = {}^2 \! A_s$, $\omega \in \{\rho, \eta \rho\}$ 
and $\tau$-equivariance of the LLC is automatic in this case.
We have $\Omega^{\mh P} = \Omega$, so $\sfa' = 4$ and $\sfb' = 1$.

By \cite[\S 6.11]{Mor} and \cite[\S 7.44--7.45]{Lusztig-unirep} the geometric diagram 
has type $D_p D_q$ with $p > q \geq 0$ and $p + q = n$. The unipotent class from $\lambda$ 
is given in \cite[4.7.(vi)]{FeOp} and has the same shape as in the previous case. The image of 
$\lambda (\Frob)$ in $SO_{2n}(\CC)$ is $I_{2p} \oplus -I_{2q}$ or $-I_{2p} \oplus I_{2q}$. 

When $q=0$, the four possibilities for $\lambda (\Frob)$ are non-conjugate and central in 
$\mb G^\vee$, so $\sfa = 4$. The given unipotent class in $\mb G^\vee = \mr{Spin}_{2n}(\CC)$ 
supports just one cuspidal local system on which $Z(\mb G^\vee)$ acts as $\omega$, so $\sfb = 1$.
We also note that $\Omega^*_\lambda = 1 = (\Omega / \Omega^{\mh P})^*$.

When $q>0$, $\lambda (\Frob)$ and $\epsilon \lambda (\Frob) \in \mr{Spin}_{2n}(\CC)$
are not conjugate, but $g \lambda (\Frob) g^{-1} = -\lambda (\Frob)$ is achieved by
taking for $g$ a lift of $-1 \oplus I_{2n-2} \oplus -1$. Hence 
\[
\Omega^*_\lambda = \{1,-1\} \supsetneq (\Omega / \Omega^{\mh P})^* = 1.
\]
The group $Z_{\mb G^\vee}(\lambda (\Frob))$ is given by \eqref{eq:8.3}. The unipotent class
and $\omega$ impose that we only look at cuspidal pairs on which $-1 \in Z(\mb G^\vee)$ acts
nontrivially. Like in the case $\mathsf J = {}^2 \! A_s$ there are four of them, two
relevant for $G^{F_\rho}$ and two relevant for $G^{F_{\eta \rho}}$. Let $\rho_1, \rho_2$
denote cuspidal enhancements for $\mr{Spin}_{2m}(\CC)$ with different central characters,
nontrivial on $-1$, and $m \in \{p,q\}$. Then the enhancements for $\omega = \rho$ are
$\rho_1 \otimes \rho_2$ and $\rho_2 \otimes \rho_1$, and the enhancements for $\omega =
\eta \rho$ are $\rho_i \otimes \rho_i$. The same analysis as in the case $\mathsf J =
{}^2 \! A_s$ shows that $\Omega^*_\lambda$ acts simply transitively on the 
$G^{F_\omega}$-relevant enhancements of $\lambda$.\\

\textbf{The exceptional automorphisms of $D_4$}

All the diagram automorphisms of order 2 are conjugate to $\tau$, so equivariance of the
LLC with respect to those follows in the same way as equivariance with respect to $\tau$.

Let $\tau_1$ and $\tau_2 = \tau_1^2$ be the order 3 diagram automorphisms of $D_4$.
The subset $\mathsf J = D_t {}^2 \! A_s D_t$ with $s>0$ cannot appear here, as $s+1$ needs to
be of the form $b(b+1)/2$ to support a cuspidal unipotent representation. 
Therefore we must have $\mathsf J = D_s D_t$ with $s + t = 4$. The finite reductive groups of 
type $D_1, D_2$ and $D_3$ (these are actually of type $A$) do not admit cuspidal unipotent 
representations, so only the case $\mathsf J = D_4$ remains. 
There $\sfa = \sfb = \sfa' = \sfb' = 1$, so it involves only one representation of 
$G^{F_\omega}$ and only one enhanced L-parameter, 
and these must be fixed by $\tau_1$ and $\tau_2$.

\section{Outer forms of even orthogonal groups}
\label{sec:outDn}

Let us look at $\mb G = PSO^*_{2n}$, the quasi-split adjoint group of type ${}^2 D_n$.
Then $\mb G^\vee = \mr{Spin}_{2n}(\CC)$ and in ${}^L \mb G$ the Frobenius elements act 
nontrivially, by the standard automorphism $\theta = \tau$ of $D_n$ of order 2.
For this $G$ we have $\Omega^\theta=\{1,\eta\}$,
\[
(\Omega^\theta)^* = Z(\mr{Spin}_{2n} (\CC)) / \{1,-1\}= \{1,\overline{\epsilon}\}
\]
and the inner twists are parametrized by $\Omega_\theta = \Omega / \{1,\eta\} = \{1,\bar \rho\}$.\\

\textbf{The case $\mathsf J = D_s \: {}^2 D_t$ with $s > 0, t > 1$}

By \cite[\S 11.4]{Lusztig-unirep2} $\omega$ has to be 1 (which is equivalent to 
$u \in \{ \theta, \theta \eta \}$ in Lusztig's notation). 
Here $\Omega^{\theta,\mh P} = \Omega^\theta$, so $\sfa' = 2$ and $\sfb' = 1$. 

Let $E/K$ be the quadratic unramified field extension over which the quasi-split group
$G^{F_\omega}$ splits, and let Frob be the associated field automorphism. 
From \eqref{eq:defGFu} we see that
\[
G^{F_\omega} = \{ g \in \mb G (E) : \theta \circ \mr{Frob}(g) = g \},
\]
where Frob acts on the coefficients of $g$ as a matrix. In particular the action of
$\tau = \theta$ on $G^{F_\omega}$ reduces to the action of the field automorphism Frob. 

There are precisely two extensions of $\sigma$ from $\PH^{F_\omega}$ to 
$N_{G^{F_\omega}}(\PH^{F_\omega})$. Since $\eta$ stabilizes $\PH$ and commutes with $\tau$, 
one can find $\tau$-fixed elements in $N_{G^{F_\omega}}(\PH^{F_\omega}) \setminus \PH^{F_\omega}$ 
(see the case $\mb G = PSO_{2n}$, $\mathsf J = D_s \, D_t$ and $\sfa = 2$). 
This entails that $\tau$ stabilizes both extensions of $\sigma$ to $N_{G^{F_\omega}}(\PH)$. 

From \cite[\S 11.4]{Lusztig-unirep2} we see that the geometric diagram has type
$B_p \, B_q$ with $p \neq q$ and $p + q + 1 = n$. We can represent the image of
$\lambda (\Frob)$ in $O_{2n}(\CC)$ by the diagonal matrix $-I_{2p+1} \oplus I_{2q+1}$. One finds 
\begin{equation}\label{eq:9.3}
Z_{\mb G^\vee}(\lambda (\Frob)) \cong \big( \text{Spin}_{2p+1}(\CC) \times \text{Spin}_{2q+1}(\CC)
\big) \big/ \langle (-1,-1) \rangle .
\end{equation}
One checks that $\epsilon \lambda (\Frob)$ is not $\mb G^\vee$-conjugate to $\lambda (\Frob)$, 
so $(\Omega^\theta)^*_\lambda = (\Omega^\theta / \Omega^{\theta,\mh P})^* = 1$ and $\sfa = 2$. 
One can obtain $\mc A_\lambda$ from \eqref{eq:9.3} by intersecting with the centralizer
of $\lambda (SL_2 (\CC))$ and adding $Z(\mb G^\vee)$. But since $G^{F_\omega}$ is quasi-split,
we may ignore the addition of the centre and just look at cuspidal pairs for \eqref{eq:9.3}
on which $Z(\mb G^\vee)^\theta$ acts trivially. According to \cite[4.7.(iv)]{FeOp}, $u_\lambda$
has Jordan blocks in the pattern $1,3,5,\ldots$. One sees quickly from the classification in 
\cite{Lus-Intersect} that the class of $u_\lambda$ admits a unique cuspidal local system which 
is equivariant for \eqref{eq:9.3}, so $\sfb = 1$. 

Since $\lambda (\Frob) \theta^{-1}$ can be chosen in $(\mb G^\vee)^\theta$ \cite{Retorsion}, 
$\tau$ fixes both enhanced L-parameters under consideration.
We conclude that the LLC is $\tau$-equivariant in this case.\\

\textbf{The case $\mathsf J = {}^2 D_t$}

Here $\Omega^{\theta,\PH} = \{1\}$, so $\omega = 1$, $\sfa' = 1$ and $\sfb' = 1$. 
The description of $\lambda$ from $s > 0$ remains valid, only now $p = q$.
Let $w \in \mb G^\vee$ be a lift of $\left( \begin{smallmatrix} 0 & I_n \\ -I_n & 0 
\end{smallmatrix} \right) \in SO_{2n}(\CC)$. Picking suitable representatives, 
we can achieve that
\[
w \lambda (\Frob) w^{-1} = \epsilon \lambda (\Frob).
\]
Thus $(\Omega^\theta)^*$ fixes the equivalence class of $\lambda$, $\sfa = 1$ and
$(\Omega^\theta)^*_\lambda = (\Omega^\theta / \Omega^{\theta,\mh P})^*$. In the same
way as above one sees that $\sfb = 1$. This case involves a unique object on both sides of 
the LLC, and the LLC matches them in an obviously $\tau$-equivariant way. \\

\textbf{Remark.} Let $\mb G_\SC = \mr{Spin}_{2n}^*$ be the simply connected cover of 
$\mb G$. When we pull back a $G^{F_\omega}$-representation coming from $(\mh P,\sigma)$ 
as above to $G_\SC^{F_\omega}$, it decomposes as a direct sum of two irreducible 
representations, one associated to $(\mh P,\sigma)$ and one to 
$(\eta \mh P, \eta^* \sigma)$. The diagram automorphism $\tau$ stabilizes $\mh P$ and 
$\eta \mh P$, so it fixes both these representations of $G_\SC^{F_\omega}$.

On the Galois side, we can consider $\lambda$ as a L-parameter $\lambda_\SC$ for
$G_\SC^{F_\omega}$. Its stabilizer is larger than \eqref{eq:9.3}:
\begin{equation}\label{eq:9.2}
Z_{\mb G^\vee}(\lambda_\SC (\Frob)) \cong \langle w \rangle \big( \text{Spin}_{2q+1}(\CC) 
\times \text{Spin}_{2q+1}(\CC) \big) \big/ \langle (-1,-1) \rangle .
\end{equation}
The unipotent class from $\lambda$ supports
two cuspidal local systems which are equivariant under \eqref{eq:9.2}. The diagram
automorphism $\tau$ induces an inner automorphism of \eqref{eq:9.3} and \eqref{eq:9.2}
(namely, conjugation by $\lambda (\Frob) \theta$), so it stabilizes both these cuspidal
enhancements of $\lambda_\SC$.\\

\textbf{The case $\mathsf J = {}^2 ( D_t  A_s D_t )$ with $t > 0$}

By \cite[11.5]{Lusztig-unirep2} $\omega = \bar \rho \in \Omega_\theta$ 
(or $u = \rho \theta$ in Lusztig's notation). Notice that $\tau (\rho) = \eta \rho$,
so $\tau$ does not preserve the group $G^{F_\omega}$. Also, $\tau$ maps enhanced 
L-parameters on which $Z(\mb G^\vee)$ acts according to $\rho$
to enhanced L-parameters on which $Z(\mb G^\vee)$ acts as $\eta \rho$. Consequently 
equivariance with respect to diagram automorphisms is automatic in this case.

We have $\Omega^{\theta,\PH} = \Omega^\theta = \{1,\eta\}$ and
$[N_{G^{F_\omega}}(\PH) : \PH] = \sfa' = 2$. 

The element $\lambda (\Frob)$ and its $\mb G^\vee$-centralizer are as in \eqref{eq:9.3},
only with different conditions on $p$ and $q$. In particular $\sfa = 2$ as above. 
The unipotent class from $\lambda$ is given in \cite[4.7.(vi)]{FeOp}: its Jordan blocks
come in the patterns $1,5,9,\ldots$ or $3,7,11,\ldots$
For this class, only cuspidal $\mc A_\lambda$-representations of dimension $>1$ have
to be considered. The classification of cuspidal local systems for spin groups in 
\cite[\S 14]{Lus-Intersect} shows that \eqref{eq:9.3} admits precisely one on which
$Z(\mb G^\vee)$ acts as $\rho$, so $\sfb = 1$.\\

\textbf{The case $\mathsf J = {}^2 \! A_s$}

As in the previous case we take $\omega = \rho$.
There are four ways to embed this $\mathsf J$ in $\widetilde{{}^2 \! D_n}$,
all conjugate under $\Omega^\theta \rtimes \{1,\theta\}$. When $n$ is even, none of these
is $F_\omega$-stable, so $n$ has to be odd. Then two of these $\mh P$'s are $F_\omega$-stable
and $\Omega^{\theta,\PH} = \{1\}$. Hence $N_{G^{F_\omega}}(\PH^{F_\omega}) = \PH^{F_\omega}$ 
and $\sfa' = 1$. 

The element $\lambda (\Frob)$ and its $\mb G^\vee$-centralizer are still as in \eqref{eq:9.3},
but with $p=q$. Just as above for $\mathsf J = {}^2 \! D_t$, one finds 
$(\Omega^\theta)^*_\lambda = (\Omega^\theta / \Omega^{\theta,\mh P})^*$ and $\sfa = 1$.
The analysis of enhancements of $\lambda$ from the case 
$\mathsf J = {}^2 \! D_t {}^2 \! A_s {}^2 D_t$ remains valid, so $\sfb = 1$.\\

\textbf{Remark.}
Let us consider the pullback of one of the above $G^{F_\omega}$-representations to 
$G_\SC^{F_\omega}$. It decomposes as a sum of two irreducibles, parametrized by $(\mh P,\sigma)$ 
and $(\eta \mh P, \eta^* \sigma)$. Notice that $\tau$ does not stabilize these two parahoric 
subgroups of $G$, rather, it sends them to $F_{\rho \eta}$-stable parahoric subgroups.
Just as in the remark to the case $\mathsf J = D_s {}^2 \! D_t$, one can show that for $G_\SC^{F_\omega}$ 
the L-parameter $\lambda$ admits two relevant enhancements. Both are fixed by $\tau$, except for the
action of $Z(\mb G^\vee)$ on the enhancements, which $\tau$ changes from $\rho$ to $\eta \rho$.\\

\textbf{The exceptional group of type ${}^3 \! D_4$}

Here $\theta$ is a diagram automorphism of $D_4$ of order 3. We have $\mb G^\vee = 
\mr{Spin}_8 (\CC)$, $Z(\mb G^\vee)^\theta = \{1\}$ and $\Omega^\theta = \{1\}$. 
In particular there is a unique inner twist, the quasi-split adjoint group of type 
${}^3 \! D_4$.

According \cite[11.10--11.11]{Lusztig-unirep2} only the subset $\mathsf J = {}^3 \! D_4$ 
of $\widetilde{{}^3 \! D_4}$ supports cuspidal unipotent representations. More precisely,
the associated parahoric subgroup $\PH^{F_\omega}$ has two cuspidal unipotent representations
with different formal degree \cite[\S 4.4.1]{Fen}, so $\sfb' = 1$ for both. 
As $\Omega^\theta = \{1\}$, $N_{G^{F_\omega}}(\PH^{F_\omega}) = \PH^{F_\omega}$ and $\sfa' = 1$. 
As $(\Omega^\theta )^* = \{1\}$, also $\sfa = 1$.

From the geometric diagrams in \cite[11.10--11.11]{Lusztig-unirep2} and 
\cite[\S 4.4]{Retorsion} we see that $Z_{\mb G^\vee}(\lambda (\Frob))$ is either 
$\text{Spin}_4 (\CC)$ or $G_2 (\CC)$, while \cite[\S 4.4.1]{Fen} tells us how these must be
matched with the two relevant supercuspidal representations. Both these complex groups admit 
a unique cuspidal pair, so $\sfb = 1$. Thus, given the formal degree we find exactly one 
cuspidal unipotent representation of $G^{F_\omega}$ and exactly one cuspidal enhanced 
L-parameter. In particular these are fixed by any diagram automorphism of $D_4$, making 
the LLC for these representations equivariant with respect to diagram automorphisms.

\section{Inner forms of $E_6$}
\label{sec:innE6}

Let $\mb G$ be the split adjoint group of type $E_6$. Then $\mb G^\vee$ also has type $E_6$ and 
\[
\Omega^* = Z(\mb G^\vee) \cong \ZZ / 3 \ZZ .
\]
We write $\Omega = \mr{Irr}(\Omega^*) = \{1,\zeta,\zeta^2\}$ and we let $\tau$ be the
nontrivial diagram automorphism of $E_6$. There are two possibilities for 
$\mathsf J \subset \widetilde{E_6}$.\\

\textbf{The case $\mathsf J = E_6$}

Here $\Omega^\PH = \{1\}$ and hence $\omega = 1$. From \cite[7.22]{Lusztig-unirep} we 
deduce that $\sfa = \sfa' = 1$ and $\sfb = \sfb' = 2$ and hence $\Omega^*_\lambda = \Omega^*$. 
Let $\sigma_1$ and $\sigma_2$ be the two cuspidal unipotent representations of $\PH^{F_\omega}$. 
Since $\Omega^*$ has order 3 and $\sfa \sfb = \sfa' \sfb' = 2$, $\Omega^*$ fixes the 
$G^{F_\omega}$-representations induced from $\sigma_1$ and $\sigma_2$, and fixes both 
enhanced L-parameters with the appropriate adjoint $\gamma$-factor.

According to \cite[Theorem 3.23]{Lus-Chevalley} the representation $\sigma_k$ can be 
realized as the eigenspace, for the eigenvalue $e^{2 k \pi i 3} q^3$, of a Frobenius 
element $F$ acting on the top $l$-adic cohomology of a variety $X_w$. Here $w$ is an
element of the Weyl group of $E_6$ which stabilizes the subsystem of type 
$A_2 A_2 A_2$. The action of 
$\theta$ on the $\sigma_k$ comes from its action on $X_w$, the variety of Borel subgroups 
$B$ of $E_6 (\overline{\mathbb F_q})$ such that $B$ and $F(B)$ are in relative 
position $w$. For the particular $w$ used here, 
$X_w$ is $\theta$-stable. Since $E_6$ is split, $F$ acts on it by a field automorphism
applied to the coefficients. The induced action on $X_w$ commutes with the $\theta$-action,
because $F$ and $\theta$ commute as automorphisms of the Dynkin diagram of $E_6$.
In particular $\theta$ stabilizes every eigenspace for $F$, and $\theta$ stabilizes both
$\sigma_1$ and $\sigma_2$. 

We work out the observations from \cite[\S 5.9]{Mor} about the semisimple element
$s = \lambda (\Frob) \in {}^L \mb G$. It corresponds to the central node $v(s)$ of $\widetilde{E_6}$. 
By \cite{Ste} its centralizer in the simply connected group $E_6 (\CC)$ is a complex connected 
group of type $A_2 A_2 A_2$. The root lattice of $A_2 A_2 A_2$ has index 3 in the root 
lattice of $E_6$, so $Z_{\mb G^\vee}(s)$ has centre of order $3 \, |\Omega^*| = 9$.
Hence $Z_{\mb G^\vee}(s)$ is the quotient of the simply connected group $SL_3 (\CC)^3$ by a central
subgroup $C$ of order 3, such that the projection of $C$ on any of the 3 factors 
$SL_3 (\CC)$ is nontrivial. Consequently
\begin{equation}\label{eq:AphiE6}
\mc A_\lambda \cong (\ZZ / 3 \ZZ)^3 / C .
\end{equation}
Since $\omega = 1$, we only have to look at enhancements of 
$\lambda$ which are trivial on $Z(\mb G^\vee)$ and we may replace $Z_{\mb G^\vee}(s)$ by
\[
Z_{\mb G^\vee}(s) / Z(\mb G^\vee) \cong SL_3 (\CC)^3 / C Z(\mb G^\vee).
\]
The centre of the latter group has order 3, and it is generated by the image $v(s)$ of $s$.
The group $SL_3 (\CC)^3$ has $2^3 = 8$ cuspidal pairs, corresponding to the characters of
$Z(SL_3 (\CC)^3) \cong (\ZZ / 3 \ZZ)^3$ which are nontrivial on each of the 3 factors.
Dividing out $C Z(\mb G^\vee)$ leaves only 2 of these characters. 
Since $\tau$ fixes $v(s)$, it stabilizes $s Z(\mb G^\vee)$ and fixes both cuspidal enhancements
of $\lambda$. Thus our LLC for these objects is $\theta$-equivariant.\\

\textbf{Remark}. Let us investigate what happens when $\mb G$ is replaced by its simply
connected cover $\mb G_\SC$ and $\lambda$ is regarded as a L-parameter $\lambda_\SC$ for 
$G_\SC^{F_\omega}$. The centralizer of $\lambda_\SC (\Frob)$ in $\mb G^\vee$ is bigger than
that of $s$. From \cite[Proposition 2.1]{Retorsion} we get a precise description, namely
$Z_{\mb G^\vee}(s) \rtimes \{1,w,w^2\}$, where the Weyl group element $w$ cyclically permutes
the factors of $A_2 A_2 A_2$. In $\mb G^\vee$ we have $w (s) = s z$ with 
$z \in Z(\mb G^\vee) \setminus \{1\}$, so
\[
\mc A_{\lambda_\SC} / Z(\mb G^\vee) = \langle s \rangle \times \langle w \rangle \cong 
(\ZZ / 3 \ZZ)^2 . 
\]
In particular both the cuspidal representations $\rho$ of $\mc A_\lambda$ can be extended 
in 3 ways to characters of $\mc A_{\lambda_\SC}$. As $\rho$ is $\tau$-stable the diagram
automorphism group $\langle \tau \rangle$ acts on the set of extensions of $\rho$
to $\mc A_{\lambda_\SC}$. There are 3 such extensions and $\tau$ has order 2, so it fixes
(at least) one extension, say $\rho_\SC$. From the actions on the root systems we see that
$\tau (w) = w^2$. If $\chi$ is a nontrivial character of $\langle w \rangle$, then
$\rho_\SC \otimes \chi$ is another extension of $\rho$ and 
\[
\tau (\rho_\SC \otimes \chi) = \rho_\SC \otimes \chi^2.
\]
Thus $\tau$ permutes the other two extensions of $\rho$. Notice that this 3-element set
of extensions is, as a $\langle \tau \rangle$-space, isomorphic to the set of standard
parahoric subgroups of $G$ which are $G^{F_\omega}$-conjugate to $\mh P$.\\

\textbf{The case $\mathsf J = {}^3 \! D_4$}

Here $\Omega^\PH = \Omega$ and $\omega \in \Omega$ has order 3. The parahoric subgroup
$\mh P^{F_\omega}$ has two cuspidal unipotent representations, say $\sigma_1$ and 
$\sigma_2$, with different formal degrees. From Lusztig's tables we get $\sfb = \sfb' = 1$ 
and $\sfa = \sfa' = 3$, so $\Omega^*_\lambda = \{1\}$. 

The diagram automorphism $\tau$ stabilizes $\Irr (G^{F_\omega})_{[\mh P,\sigma_i]}$,
because it preserves formal degrees. Using \cite[\S 4.4.1]{Fen} we match the geometric
diagrams \cite[7.20 and 7.21]{Lusztig-unirep} with $[\mh P,\sigma_1]$ and $[\mh P,\sigma_2]$.
As these two diagrams differ, $\tau$ stabilizes the triple of (enhanced) L-parameters
associated to $[\mh P,\sigma_i]$ (for $i = 1,2$). As $\tau$ has order 2, it fixes at least one
element of $\Irr (G^{F_\omega})_{[\mh P,\sigma_i]}$, say $\pi_i$. The group $\Omega^*$ acts 
simply transitively on $\Irr (G^{F_\omega})_{[\mh P,\sigma_i]}$ and $\tau$ acts
nontrivially on $\Omega^*$, so $\pi_i$ is the unique $\tau$-fixed element of 
$\Irr (G^{F_\omega})_{[\mh P,\sigma_i]}$.

By the same argument, $\tau$ fixes exactly one the three enhanced L-parameters associated 
to $[\mh P,\sigma_i]$, say $(\lambda_i,\rho_i)$. Decreeing that $\pi_i$ corresponds to 
$(\lambda_i,\rho_i)$, we obtain a $\Omega^* \rtimes \langle \tau \rangle$-equivariant 
bijection between $\Irr (G^{F_\omega} )_{[\mh P,\sigma_i]}$ and the associated triple in 
$\Phi (G^{F_\omega})_\cusp$.

\section{The outer forms of $E_6$}
\label{sec:outE6}

Now $\tau = \theta$ is the nontrivial diagram automorphism of $E_6$. The groups
$\Omega^\theta, \Omega_\theta, (\Omega^\theta)^*$ and $(\Omega^*)^\theta$ are all
trivial. In particular $G^{F_\omega}$ is necessarily quasi-split.

From \cite{Lusztig-unirep2} we see that only $\mathsf J = {}^2 \! E_6$ supports cuspidal 
unipotent representations. The group $\PH^{F_\omega}$ has one self-dual cuspidal unipotent 
representation $\sigma_0$, for which $\sfa = \sfa' = \sfb = \sfb' = 1$. We see from
\cite[\S 4.4.2]{Fen} that $Z_{\mb G^\vee}(\lambda (\Frob)) \cong F_4 (\CC)$, which 
has just one unipotent class supporting a cuspidal local system. The associated 
$G^{F_\omega}$-representation and its enhanced L-parameter are determined uniquely 
by the geometric diagram \cite[11.7]{Lusztig-unirep2}, so the objects are fixed by $\tau$. 

Also, $\PH^{F_\omega}$ has two other cuspidal unipotent representations $\sigma_1$ and 
$\sigma_2$. For $\sigma_1$ and $\sigma_2$ we have $\sfa = \sfa' = 1$ and $\sfb = \sfb' = 2$ 
\cite[11.6]{Lusztig-unirep2}. The same reasoning as for the inner forms of $E_6$ with 
$J = E_6$, relying on \cite{Lus-Chevalley}, shows that $\tau$ stabilizes both $\sigma_i$.

By \cite[\S 4.4.2]{Fen} $\lambda (\Frob) = s \theta$, where $s \in (\mb G^\vee)^\theta$ 
is associated to the central node of the affine Dynkin diagram of $\mb G^\vee$. 
The orders of $\theta$ and of the image of $s$ in ${\mb G^\vee}_\ad$
(2 and 3, respectively) are coprime, so 
\[
Z_{\mb G^\vee}(\lambda (\Frob)) = (\mb G^\vee)^\theta \cap Z_{\mb G^\vee}(s)
\cong \big( SL_3 (\CC)^3 / C \big)^\theta \cong SL_3 (\CC)^2 / C' 
\]
for suitable central subgroups $C, C'$ of order 3.
Thus the component group of the L-parameter $\lambda$ associated to $\sigma_1,\sigma_2$ 
is obtained from \eqref{eq:AphiE6} by taking $\theta$-invariants. That removes $Z(\mb G^\vee)$
from \eqref{eq:AphiE6}, but then the very definition of $\mc A_\lambda$ says that we
have to include the centre again. It follows that 
\[
\mc A_\lambda = Z(\mb G^\vee) \times \langle s \rangle \cong (\ZZ / 3 \ZZ)^2 .
\]
In \eqref{eq:2.9} we already fixed the $Z(\mb G^\vee)$-character of every relevant 
representation of $\mc A_\lambda$ (namely, the trivial central character), so it suffices to 
consider the representations of the subgroup generated by $s$. Its irreducible cuspidal 
representations are precisely the two nontrivial characters. Since $\langle s \rangle$ 
is fixed entirely by $\theta$, so are these two enhancements of $\lambda$. We conclude 
that also in this case the LLC is $\theta$-equivariant.

\section{Groups of Lie type $E_7$}
\label{sec:E7}

Let $\mb G$ be the split adjoint group of type $E_7$. Then $\mb G^\vee$ also has type 
$E_7$ and $|\Omega| = 2$. From \cite[\S 5.9, \S 6.13]{Mor} and \cite[7.12--7.14]{Lusztig-unirep} 
it is known that two subsets of the affine Dynkin diagram $\widetilde{E_7}$ are relevant 
for our purposes.\\

\textbf{The case $\mathsf J = E_6$}

This $\mathsf J$ only gives rise to supercuspidal unipotent representations of $G^{F_\omega}$ 
if $\omega$ is nontrivial. The associated parahoric subgroup satisfies $\Omega^{\mh P} = \Omega$. 
In \cite[7.12 and 7.13]{Lusztig-unirep} $\sfa = \sfa' = 2$ and $\sfb = \sfb' \in \{1,2\}$.  
In view of Theorem \ref{thm:A}, $\Omega^*$ in each case permutes the two involved L-parameters
$\lambda$. Hence $(\Omega / \Omega^{\mh P})^* = \{1\}$ is precisely the stabilizer of $\lambda$ 
and any of its $G^{F_\omega}$-relevant enhancements.\\

\textbf{The case $\mathsf J = E_7$}

By \cite[7.14]{Lusztig-unirep} $\sfa = \sfa' = 1$, $\sfb = \sfb' = 2$, $\Omega^{\mh P} = 1$ and 
$\omega = 1$, so the group $G^{F_\omega}$ is split. In particular every relevant representation 
of $\mc A_\lambda$ is trivial on $Z(\mb G^\vee)$. The group $\mc A_\lambda / Z(\mb G^\vee)$ is
isomorphic to $\ZZ / 4 \ZZ$ \cite[p. 34]{Ree1} and is generated by the element 
$\lambda (\Frob)$, which has order four in the derived group of $\mb G^\vee$ \cite{Retorsion}.
The nontrivial element of $\Omega^* = Z(\mb G^\vee)$ sends $\lambda (\Frob)$ to a different but 
conjugate element of $\mb G^\vee$. Suppose that $g \in \mb G^\vee$ achieves this conjugation. 
Then conjugation by $g$ stabilizes $\lambda (\Frob) Z(\mb G^\vee)$, so it fixes 
$\mc A_\lambda / Z(\mb G^\vee)$ pointwise. Hence the action of $\Omega^*$ on the enhancements 
of $\lambda$ is trivial, and $(\Omega / \Omega^{\mh P})^*$ stabilizes them all.

\section{Adjoint unramified groups}
\label{sec:adjoint}

First we wrap up our findings for unramified simple adjoint groups, then we prove
Theorem \ref{thm:B} for all unramified adjoint groups.

\begin{prop}\label{prop:11.1}
Theorem \ref{thm:B} holds for all unramified simple adjoint $K$-groups $\mb G$.
\end{prop}
\begin{proof}
In view of Lemma \ref{lem:17.1} and Proposition \ref{prop:17.2}, the objects in
Theorem \ref{thm:B} are unaffected by restriction of scalars for reductive groups. Hence 
we may assume that $\mb G$ is absolutely simple. We start with the Langlands parameters for 
supercuspidal unipotent representations from  \cite{Mor,Lusztig-unirep,Lusztig-unirep2}, 
where all free choices are made compatibly with Theorem \ref{thm:A}. 
In some of the cases a completely canonical $\lambda$ can be found by closer inspection, 
for instance see Section \ref{sec:innAn}.

On the $p$-adic side the subgroup $(\Omega^\theta / \Omega^{\theta,\mh P})^*$ of 
$(\Omega^\theta)^*$ acts trivially, and the quotient group $(\Omega^{\theta,\mh P})^*$ 
acts simply transitively on $\Irr (G^{F_\omega})_{[\mh P,\sigma]}$, see \eqref{eq:2.10}
(based on \cite{Lusztig-unirep}). A bijection from 
$(\Omega^{\theta,\mh P})^*$ to $\Irr (G^{F_\omega})_{[\mh P,\sigma]}$ can be determined
by fixing an extension of $\sigma$ to $N_{G^{F_\omega}}(\mh P^{F_\omega})$.

All possibilities for $(\mh P,\sigma)$ up to conjugacy can be found in \cite[\S 5--6]{Mor} 
and \cite[\S 7]{Lusztig-unirep} (for inner forms of split groups) and 
\cite[\S 11]{Lusztig-unirep2} (for outer forms of split groups). These lists show that the 
$G^{F_\omega}$-conjugacy class of $\mh P$ is uniquely determined by $\lambda$. Hence the 
$(\Omega^{\theta})^*$-orbits on the set of solutions $\pi$ of \eqref{eq:fdegisgammafactor} 
are parametrized by the cuspidal unipotent $\mh P^{F_\omega}$-representations 
with the same formal degree as $\sigma$. In particular there are $\sfb'$ such orbits.

For inner forms of split groups the numbers $\sfa, \sfb, \sfa', \sfb'$ and the equality 
$\sfa \sfb = \sfa' \sfb'$ are known from \cite[\S 7]{Lusztig-unirep}. For outer forms we 
have exhibited these numbers in Sections \ref{sec:outAn}, \ref{sec:outDn} and \ref{sec:outE6}.
That $\sfb' = 1$ for classical groups is known from \cite{Lus-Chevalley}. The equality
$\sfb' = \phi (n_s)$ can be seen from \cite[\S 7]{Lusztig-unirep} and
\cite[\S 11]{Lusztig-unirep2}.

In the adjoint case all parahorics admitting cuspidal unipotent representations with
the same formal degree are conjugate, so $\sfa' = |\Omega^{\theta,\mh P}|$. By \eqref{eq:2.1},
$(\Omega^\theta )^*$ is naturally isomorphic with $Z(\mb G^\vee)_\theta$. By Theorem
\ref{thm:A} $Z(\mb G^\vee)$ acts transitively on the set of $\lambda$'s with the
same adjoint $\gamma$-factor, and that descends to a transitive action of $(\Omega^\theta)^*$.
Therefore $(\Omega^\theta)^* / (\Omega^\theta)^*_\lambda$ acts simply transitively on the set 
of such $\lambda$, and $\sfa$ is as claimed in Theorem \ref{thm:B}.(5).
In the previous sections we checked that $(\Omega^\theta)^*_\lambda$ always contains
$(\Omega^\theta / \Omega^{\theta,\mh P})^*$.

This entails that we can find a bijection as in part (1), which is 
$(\Omega^\theta)^*$-equivariant as far as $\pi$ and $\lambda$ are concerned, but maybe
not on the relevant enhancements of $\lambda$. Notice that by Theorem \ref{thm:A} our 
method determines $\lambda$ uniquely up to $(\Omega^\theta)^*$ (given $\pi$). 

A priori it is possible that $(\Omega^\theta / \Omega^{\theta,\mh P})^*$ acts nontrivially 
on some enhancements. To rule that out we need another case-by-case check.
There are only few cases with $\sfb > \sfb'$, or equivalently 
$(\Omega^\theta)^*_\lambda \supsetneq (\Omega^\theta / \Omega^{\theta,\mh P})^*$, namely 
\cite[7.44, 7.46, 7.51 and 7.52]{Lusztig-unirep}. In those cases $\sfb' = 1$, and we 
checked in Sections \ref{sec:innCn} and \ref{sec:innDn} that $(\Omega^\theta)^*_\lambda$ 
acts transitively on the set of $G^{F_\omega}$-relevant cuspidal enhancements of 
$\lambda$, with isotropy group $(\Omega^\theta / \Omega^{\theta,\mh P})^*$.

In the other cases $\sfb = \sfb'$ and $(\Omega^\theta)^*_\lambda = 
(\Omega^\theta / \Omega^{\theta,\mh P})^*$. Usually $\sfb = \sfb' = 1$, then 
$\mc A_\lambda$ has only one relevant cuspidal representation $\rho$ and 
$(\Omega^\theta / \Omega^{\theta,\mh P})^*$ is the stabilizer of $(\lambda,\rho)$.
When $\sfb = \sfb' > 1$, $\mb G$ must be an exceptional group. For Lie types $G_2, F_4, E_8$ 
and ${}^3 D_4$ the group $(\Omega^\theta)^*$ is trivial, so there is nothing left to prove.
For Lie types $E_6$ and $E_7$ see Sections \ref{sec:innE6}, \ref{sec:outE6} and \ref{sec:E7}.

This shows part (3) and part (1) except the equivariance with respect to diagram automorphisms. 
But the latter was already verified in the previous sections (notice that in Sections 
\ref{sec:innBn} and \ref{sec:innCn} the Dynkin diagrams only admit the trivial automorphism).

Now only part (6) on the Galois side remains. By the earlier parts, there are precisely
$\sfb' = \phi (n_s)$ orbits under $(\Omega^\theta)^*$. Since all solutions $\lambda$
for \eqref{eq:fdegisgammafactor} are in the same $(\Omega^\theta)^*$-orbit, the orbits
on $\Phi_\nr (G^{F_\omega})_\cusp$ can be parametrized by enhancements of one $\lambda$.
More precisely, such orbits are parametrized by any set of representatives for the
action of $(\Omega^\theta)^*_\lambda$ on the $G^{F_\omega}$-relevant enhancements of
$\lambda$.
\end{proof}

\begin{prop}\label{prop:11.2}
Theorem \ref{thm:B} holds for all unramified adjoint $K$-groups $\mb G$.
\end{prop}
\begin{proof}
Every adjoint linear algebraic group is a direct product of simple adjoint groups. 
It is clear that everything in Theorem \ref{thm:B} (apart from diagram automorphisms)
factors naturally over direct products of groups. Here the required compatibility with
(almost) direct products, as in \eqref{eq:2.20}, says that the enhanced L-parameter of
a $\pi \in \Irr (G^{F_\omega})_{\cusp,\unip}$ is completely determined by what happens
for the simple factors of $\mb G$. In particular, Proposition \ref{prop:11.1}
establishes Theorem \ref{thm:B}, except equivariance with respect to diagram automorphisms, 
for all unramified adjoint groups which are inner forms of a $K$-split group.

Consider an unramified adjoint $K$-group with simple factors $\mb G_i$:
\begin{equation}\label{eq:11.2}
\mb G = \mb G_1 \times \cdots \times \mb G_d .
\end{equation}
Suppose that a diagram automorphism $\tau$ maps $\mb G_1$ to $\mb G_2$. Then $\mb G_1$ and 
$\mb G_2$ are isomorphic, and $\tau$ also maps $\mb G_1^\vee$ to $\mb G_2^\vee$ inside $\mb G^\vee$. 
Let $\Gamma_1$ be the stabilizer of $\mb G_1$ in the group of diagram automorphisms of $\mb G$.

Assume that the $\theta$-action on the set $\{ \mb G_1,\ldots,\mb G_d\}$ is trivial. Then
Theorem \ref{thm:B} for $\mb G$ follows directly from Proposition \ref{prop:11.1} for the 
$\mb G_i$, except possibly for part (2). But we can make the LLC from Theorem \ref{thm:B} 
$\tau$-equivariant by first constructing it for $\mb G_1$ in a $\Gamma_1$-equivariant way, 
and then defining it for $\mb G_2$ by imposing $\tau$-equivariance.

When $\theta$ acts nontrivially on the set of direct factors of $\mb G$, the above enables
us to reduce to the case where the $\mb G_i$ form one orbit under the group 
$\langle \theta \rangle$. Then clearly $G^{F_\omega} \cong G_1^{F^d_\omega}$. For a more
precise formulation, let $K_{(d)}$ be the unramified extension of $K$ of degree $d$. 
Then the $K$-group $\mb G^\omega$ is the restriction of scalars, from $K_{(d)}$ to $K$, of the 
$K_{(d)}$-group $\mb G_1^\omega$. Lemma \ref{lem:17.1} says that there is a natural bijection
\begin{equation}\label{eq:11.1}
\Phi (G^{F_\omega})_\cusp \longrightarrow \Phi (G_1^{F_\omega^d})_\cusp .
\end{equation}
Proposition \ref{prop:17.2} says that Theorem \ref{thm:B} for the $K$-group $\mb G$ is equivalent 
to Theorem \ref{thm:B} for the simple, adjoint $K_{(d)}$-group $\mb G_1$, via \eqref{eq:11.1}.
Now $\theta$ has been replaced by $\theta^d$, which stabilizes $\mb G_1$, so we can apply the 
method from the case with trivial $\theta$-action on the set of simple factors of $\mb G$.
\end{proof}

\section{Semisimple unramified groups}
\label{sec:semisimple}

Let $\mb G$ be a semisimple unramified $K$-group, and let $\mb G_\ad$ be its adjoint 
quotient. We will compare the numbers $\sfa,\sfb,\sfa'$ and $\sfb'$ for $\mb G$ with
those for $\mb G_\ad$, which we denote by a subscript ad.

Let $\pi_\ad, \lambda_\ad, \mh P_\ad, \sigma_\ad$ be as in \eqref{eq:1.4}, for $\mb G_\ad$. 
From \cite{Lusztig-unirep,Lusztig-unirep2}, Theorem \ref{thm:B} and Lemma \ref{lem:3.1}
we know that $(\Omega_\ad^{\theta,\mh P} )^*$ acts simply transitively on the set of
$\pi_\ad \in \mr{Irr} (G_\ad^{F_\omega})$ containing $(\mh P_\ad, \sigma_\ad)$. In other words,
$(\Omega_\ad^{\theta} )^*$ acts transitively, and $(\Omega_\ad^\theta / 
\Omega_\ad^{\theta,\mh P} )^*$ is the stabilizer of $\pi_\ad$.

Let $\pi \in \Irr (G^{F_\omega})_{\cusp,\unip}$ be contained in the pullback of $\pi_\ad$
to $G^{F_\omega}$. It is known \cite[\S 3]{Lus-Chevalley} that unipotent cuspidal 
representations of a finite reductive group depend on the Lie type of the group. 
So every $(\mh P_\ad, \sigma_\ad)$ lifts uniquely to $(\mh P, \sigma)$ and
$\pi \in \Irr (G^{F_\omega})_{[\mh P,\sigma]}$. The packets of cuspidal 
unipotent representations of these parahoric subgroups satisfy
\begin{equation}\label{eq:11.7}
\sfb' = \sfb'_\ad .
\end{equation}
When $\mb G_\ad$ is adjoint and simple, Theorem \ref{thm:A} and Lusztig's classification show 
that the formal degree (of $\pi_\ad$) determines a unique conjugacy class of $F_\omega$-stable 
parahoric subgroups of $G_\ad^{F_\omega}$ which gives rise to one or more supercuspidal 
unipotent representations with that formal degree. Via a factorization as in \eqref{eq:11.2} 
this extends to all unramified adjoint $\mb G_\ad$.

This need not be true when $\mb G$ is not adjoint, but then still all such parahoric 
subgroups are associate by elements of $\Omega^\theta_\ad$. It follows that the number of 
$G^{F_\omega}$-conjugacy classes of such parahoric subgroups is precisely
$[ \Omega_\ad^\theta / \Omega_\ad^{\PH,\theta} : \Omega^\theta/\Omega^{\PH,\theta} ] = g'$.
By \eqref{eq:2.10} the group $(\Omega^{\theta,\mh P})^*$ acts simply transitively on 
the set of irreducible $G^{F_\omega}$-representations containing $\sigma$. It follows that
\begin{equation}\label{eq:11.8}
\sfa' = |\Omega^{\theta,\mh P}| g' = 
[\Omega^{\theta,\mh P}_\ad : \Omega^{\theta,\mh P}]^{-1} g' \, \sfa'_\ad ,
\end{equation}
and that $\sfa' \sfb'$ equals the number of supercuspidal unipotent
$G^{F_\omega}$-representations which on each $K$-simple factor $\mb G_i$ give the same formal 
degree as $\pi$ (see page \pageref{lem:3.1}).

By Theorem \ref{thm:B}.(3) for $\mb G_\ad$:
\begin{equation}\label{eq:11.3}
(\Omega_\ad^\theta )_{\lambda_\ad}^* = (\Omega_\ad^\theta / N_{\lambda_\ad} )^*
\text{ for a subgroup } N_{\lambda_\ad} \subset \Omega^{\theta,\mh P}_\ad .
\end{equation}
By Theorem \ref{thm:B}.(5) $N_{\lambda_\ad}^*$ is naturally in bijection with the 
$(\Omega_\ad^\theta )^*$-orbit of $\lambda_\ad$. In particular 
\begin{equation}\label{eq:11.4}
\sfa_\ad = |(\Omega_\ad^\theta )^* \lambda_\ad| = |N_{\lambda_\ad}| .
\end{equation}
Also, $(\Omega_\ad^\theta )_{\lambda_\ad}^*$ has $\sfb_\ad$ elements and acts simply 
transitively on the set of $G^{F_\omega}_\ad$-relevant cuspidal enhancements of $\lambda_\ad$.

\begin{lemma}\label{lem:11.3}
Let $\lambda \in \Phi_\nr^2 (G^{F_\omega})$ be the projection of $\lambda_\ad$ via
$\mb G^\vee_\SC \to \mb G^\vee$.
\begin{enumerate}
\item $(\Omega^\theta)^*$ acts transitively on the collection of $\lambda' \in \Phi_\nr^2 (G^{F_\omega})$ 
which, for every $K$-simple factor $\mb G_i$ of $\mb G$, have the same $\gamma$-factor 
$\gamma (0, \mr{Ad}_{\mb G_i^\vee} \circ \lambda',\psi)$ as $\lambda$.
\item The stabilizer of (the equivalence class of) $\lambda$ equals 
$(\Omega^\theta / \Omega^\theta \cap N_{\lambda_\ad})^*$, and it contains
$(\Omega^\theta / \Omega^{\theta,\mh P})^*$.
\item $\sfa = |N_{\lambda_\ad} \cap \Omega^\theta|$.
\end{enumerate}
\end{lemma}
\begin{proof}
(1) By Theorem \ref{thm:B} for $\mb G_\ad$, $(\Omega^\theta_\ad )^* \lambda_\ad$ is 
precisely the collection of L-parameters for $G_\ad^{F_\omega}$ with a given adjoint 
$\gamma$-factor. Consequently every lift of a $\mb G^\vee$-conjugate of $\lambda$ is 
$\mb G^\vee_\ad$-conjugate to an element of $(\Omega^\theta_\ad )^* \lambda_\ad$, and 
$(\Omega^\theta )^* \lambda$ is the collection of L-parameters for $G^{F_\omega}$ 
with the same adjoint $\gamma$-factor as $\lambda$.\\
(2) Since $(\Omega_\ad^\theta / \Omega_\ad^{\theta,\mh P})^*$ stabilizes $\lambda_\ad$ (by 
Theorem \ref{thm:B}.(3) for $\mb G_\ad$), its image $(\Omega^\theta / \Omega^{\theta,\mh P})^*$ 
under $(\Omega_\ad^\theta)^* \to (\Omega^\theta)^*$ stabilizes $\lambda$.

In $\mb G^\vee$ some different elements of $\mb G^\vee_\SC$ become equal, namely 
$\ker (\mb G^\vee_\SC \to \mb G^\vee) = (\Omega_\ad / \Omega )^*$. The image of 
$\ker (\mb G^\vee_\SC \to \mb G^\vee)$ in $(\Omega_\ad^\theta)^*$ is 
\[
(\Omega_\ad^\theta / \Omega^\theta)^* = \ker \big( (\Omega_\ad^\theta)^* \to (\Omega^\theta)^* \big).
\]
Hence the stabilizer of $\lambda$ in $(\Omega^\theta)^*$ is precisely the image in 
$(\Omega^\theta)^*$ of the $(\Omega_\ad^\theta)^*$-stabilizer of the orbit 
$(\Omega_\ad^\theta / \Omega^\theta)^* \lambda_\ad$. That works out as
\begin{multline*}
(\Omega_\ad^\theta / \Omega^\theta)^* (\Omega^\theta_\ad)^*_{\lambda_\ad} \big/ 
(\Omega_\ad^\theta / \Omega^\theta)^* =
(\Omega^\theta_\ad)^*_{\lambda_\ad} \big/ 
\big( (\Omega_\ad^\theta / \Omega^\theta)^* \cap (\Omega^\theta_\ad)^*_{\lambda_\ad} \big) \\
= (\Omega^\theta_\ad / N_{\lambda_\ad})^* \big/ (\Omega^\theta_\ad / 
\Omega^\theta N_{\lambda_\ad})^* = (\Omega^\theta N_{\lambda_\ad} / N_{\lambda_\ad})^*
= (\Omega^\theta / \Omega^\theta \cap N_{\lambda_\ad})^* .
\end{multline*}
(3) We saw in \eqref{eq:11.3} that $(N_{\lambda_\ad})^*$ acts simply transitively on
$(\Omega^\theta_\ad )^* \lambda_\ad$, so it also acts transitively on 
$(\Omega^\theta )^* \lambda$. The stabilizer of $\lambda$ in this group is
$(N_{\lambda_\ad} / N_{\lambda_\ad} \cap \Omega )^*$, the image of $(\Omega_\ad / \Omega )^*$ 
in $(N_{\lambda_\ad})^*$. Then the quotient group 
\[
N_{\lambda_\ad}^* / (N_{\lambda_\ad} / N_{\lambda_\ad} \cap \Omega )^* =
(N_{\lambda_\ad} \cap \Omega)^* = (N_{\lambda_\ad} \cap \Omega^\theta)^*
\]
acts simply transitively on $(\Omega^\theta )^* \lambda$. We deduce that $\sfa = 
|N_{\lambda_\ad} \cap \Omega^\theta|$. 
\end{proof}

In the setting of Lemma \ref{lem:11.3}, $\mc A_\lambda$ contains $\mc A_{\lambda_\ad}$ 
as a normal subgroup. We want to compare these subgroups of $\mb G^\vee_\SC$, and the 
cuspidal local systems which they support.

\begin{lemma}\label{lem:11.4}
The group $\mc A_\lambda / \mc A_{\lambda_\ad}$ is isomorphic to 
$(\Omega^\theta_\ad / \Omega^\theta N_{\lambda_\ad})^*$. 
\end{lemma}
\begin{proof}
First we determine which lifts of $\lambda$ to a L-parameter for $G_\ad^{F_\omega}$ are 
$\mb G^\vee_\SC$-conjugate. To be conjugate, they have to be related by elements of 
$(\Omega_\ad^\theta)^*_{\lambda_\ad} = (\Omega^\theta_\ad / N_{\lambda_\ad})^*$. 
To be lifts of the one and the same $\lambda$, they may differ only by elements of 
\[
\ker (\mb G^\vee_\SC \to \mb G^\vee) = (\Omega_\ad / \Omega )^* .
\] 
Therefore two lifts of $\lambda$ are conjugate if and only if they differ by an element of the intersection
\[
(\Omega^\theta_\ad / N_{\lambda_\ad})^* \cap (\Omega_\ad^\theta / \Omega^\theta )^* =
(\Omega^\theta_\ad / \Omega^\theta N_{\lambda_\ad})^* .
\]
We write $S_\lambda = Z_{\mb G^\vee_\SC}(\mr{im} \, \lambda) = 
\pi_0 \big( Z_{\mb G^\vee_\SC}(\mr{im} \, \lambda) \big)$, where $\mb G^\vee_\SC$ acts by
conjugation, via the natural map to the derived group of $\mb G^\vee$. The above implies 
\begin{equation}\label{eq:11.5}
S_\lambda / S_{\lambda_\ad} \cong (\Omega_\ad^\theta / \Omega^\theta N_{\lambda_\ad})^* .
\end{equation}
The more subtle component group $\mc A_{\lambda_\ad}$ contains $S_{\lambda_\ad}$ with index 
$[Z(\mb G^\vee_\SC) : Z(\mb G^\vee_\SC)^{\mb W_F}] = | (1 - \theta) Z(\mb G^\vee_\SC) |$. Similarly 
$[\mc A_{\lambda} : S_{\lambda}] = [Z(\mb G^\vee_\SC) : Z(\mb G^\vee_\SC)^{\mb W_F}]$ and hence
\begin{equation}\label{eq:11.6}
\mc A_\lambda / \mc A_{\lambda_\ad} \cong S_\lambda / S_{\lambda_\ad} \qedhere
\end{equation}
\end{proof}

Next we compare the cuspidal enhancements of $\lambda$ and $\lambda_\ad$. Since $\mc A_\lambda$
contains $\mc A_{\lambda_\ad}$ as normal subgroup, it acts on $\mc A_{\lambda_\ad}$ (by conjugation)
and it acts on $\Irr (\mc A_{\lambda_\ad})$. For $\rho_\ad \in \Irr (\mc A_{\lambda_\ad})$, we let  
$(\mc A_\lambda )_{\rho_\ad}$ be its stabilizer in $\mc A_\lambda$.

\begin{lemma}\label{lem:11.6}
Every irreducible cuspidal representation $\rho_\ad$ of $\mc A_{\lambda_\ad}$ extends to a
representation of $(\mc A_\lambda )_{\rho_\ad}$.
\end{lemma}
\begin{proof}
Recall from Theorem \ref{thm:B}.(3) for $\mb G_\ad$ that the
stabilizer of $(\lambda_\ad,\rho_\ad)$ in $(\Omega_\ad^\theta)^*$ is $(\Omega_\ad^\theta / 
\Omega_\ad^{\theta,\mh P})^*$. Under the isomorphism from Lemma \ref{lem:11.4} or \eqref{eq:11.5}, 
the stabilizer of $\rho_\ad$ in $\mc A_\lambda / \mc A_{\lambda_\ad}$ corresponds to 
\begin{equation}\label{eq:11.12}
(\Omega^\theta_\ad / \Omega^\theta N_{\lambda_\ad})^* \cap (\Omega_\ad^\theta / 
\Omega_\ad^{\theta,\mh P})^* = (\Omega_\ad^\theta / \Omega^\theta \Omega_\ad^{\theta,\mh P})^* .
\end{equation}
In the proof of Proposition \ref{prop:11.2} we checked that everything for $\mb G_\ad$ factors
as a direct product of objects associated to simple adjoint groups. In particular $\rho_\ad$
is a tensor product of cuspidal representations $\rho_i$ of groups $\mc A_i$ associated to 
L-parameters $\lambda_i$ for adjoint simple groups $G_i^{F_\omega}$. Thus it suffices to show 
that every such $\rho_i$ extends to a representation of $(\mc A_\lambda )_{\rho_\ad}$. 
The action of $\mc A_\lambda$ on $\rho_i$ factors through the almost direct factor of 
$G^{F_\omega}$ which corresponds to $G_i^{F_\omega}$. This enables us to reduce to the case 
where $\mb G$ is almost simple, which we assume for the remainder of this proof.

Now we can proceed by classification, using \cite[\S 7]{Lusztig-unirep}. By \eqref{eq:11.12}
and \cite[Theorem 1.2]{AMS1} we have to consider projective representations of 
$(\Omega_\ad^\theta / \Omega^\theta \Omega_\ad^{\theta,\mh P})^*$. In almost all cases this 
group is cyclic, because $\Omega_\ad^\theta$ is cyclic. Every 2-cocycle (with values in
$\CC^\times$) of a cyclic group is trivial, so then by \cite[Proposition 1.1.a]{AMS1} $\rho_\ad$ 
extends to a representation of $(\mc A_\lambda )_{\rho_\ad}$. 

The only exceptions are the inner twists of split groups of type $D_{2n}$, for those
$\Omega^\theta_\ad \cong (\ZZ / 2 \ZZ)^2$. The group $(\Omega^\theta_\ad / 
\Omega^\theta \Omega_\ad^{\theta,\mh P})^*$ can only be non-cyclic if $\mb G$ is simply 
connected and $\Omega_\ad^{\theta,\mh P} = 1$, which forces $\overline{\mh P}$ to be of 
type $D_{2n}$. For this case, see the remark to $\mathsf J = D_n$ in Section \ref{sec:innDn}.
\end{proof}

It requires more work to relate the numbers of $G^{F_\omega}$-relevant cuspidal
enhancements of $\lambda$ (i.e. $\sfb$) and of $\lambda_\ad$ (i.e. $\sfb_\ad$), 
in general their ratio is less than $[\mc A_{\lambda} : \mc A_{\lambda_\ad}]$.

\begin{lemma}\label{lem:11.5}
$\sfb = g' [\Omega^{\theta,\mh P}_\ad : \Omega^{\theta,\mh P} N_{\lambda_\ad}]^{-1} \sfb_\ad$.
\end{lemma}
\begin{proof}
It follows directly from \cite[Definition 6.9]{AMS1} that an irreducible 
$\mc A_\lambda$-representation is cuspidal if and only if its restriction to
$\mc A_{\lambda_\ad}$ is a direct sum of cuspidal representations. Such a situation
can be analysed with a version of Clifford theory \cite[Theorem 1.2]{AMS1}. 
Briefly, this method entails that first we
exhibit the $\mc A_\lambda$-orbits of cuspidal representations in $\Irr (\mc A_{\lambda_\ad})$. 
In every such orbit we pick one representation $\rho_\ad$ and we determine its
stabilizer $(\mc A_\lambda)_{\rho_\ad}$. By a choice of intertwining operators, $\rho_\ad$ can 
be extended to a projective representation $\widetilde{\rho_\ad}$ of $(\mc A_\lambda)_{\rho_\ad}$.
Then the set of the irreducible $\mc A_\lambda$-representations that contain $\rho_\ad$ 
is in bijection with the set of irreducible representations (say $\tau$) of a twisted group 
algebra of $(\mc A_\lambda)_{\rho_\ad} / \mc A_{\lambda_\ad}$. The bijection sends $\tau$ to
\begin{equation}\label{eq:11.16}
\mr{ind}_{(\mc A_\lambda)_{\rho_\ad}}^{\mc A_\lambda} (\tau \otimes \widetilde{\rho_\ad}) \in
\Irr (\mc A_\lambda) .
\end{equation}
If $\rho_\ad$ can be extended to a (linear) representation of $(\mc A_\lambda)_{\rho_\ad}$, 
the aforementioned twisted group algebra becomes simply the group algebra of 
$(\mc A_\lambda)_{\rho_\ad} / \mc A_{\lambda_\ad}$. In that simpler case, the desired number of 
cuspidal irreducible $\mc A_\lambda$-representations is the sum, over the $\mc A_\lambda$-orbits 
of the appropriate $\mc A_{\lambda_\ad}$-representations, of the numbers 
\begin{equation}\label{eq:11.9}
|\Irr ((\mc A_\lambda)_{\rho_\ad} / \mc A_{\lambda_\ad})| = 
[(\mc A_\lambda)_{\rho_\ad} : \mc A_{\lambda_\ad}] .
\end{equation}
For this equality we use that $\mc A_\lambda / \mc A_{\lambda_\ad}$ is abelian, which
is immediate from Lemma \ref{lem:11.4}.

Let us make this explicit. By \eqref{eq:11.12} $(\mc A_\lambda)_{\rho_\ad}$ is the inverse image 
of $(\Omega_\ad^\theta / \Omega^\theta \Omega_\ad^{\theta,\mh P})^*$ in $\mc A_\lambda$ under 
\eqref{eq:11.5}. Notice that this does not depend on $\rho_\ad$, every other relevant enhancement 
gives the same stabilizer. It follows that the quotient group
\begin{equation}\label{eq:11.13}
\begin{aligned}
\mc A_\lambda / (\mc A_\lambda)_{\rho_\ad} & \cong (\Omega^\theta_\ad / 
\Omega^\theta N_{\lambda_\ad})^* / (\Omega_\ad^\theta / \Omega^\theta \Omega_\ad^{\theta,\mh P})^* \\
& \cong (\Omega^\theta \Omega_\ad^{\theta,\mh P} / \Omega^\theta N_{\lambda_\ad})^* 
\cong (\Omega_\ad^{\theta,\mh P} / \Omega^{\theta,\mh P} N_{\lambda_\ad})^*
\end{aligned}
\end{equation}
acts freely on the collection of $G^{F_\omega}$-relevant cuspidal enhancements of $\lambda_\ad$.
Now we can compute the number of $\mc A_\lambda$-orbits of such enhancements:
\begin{equation}\label{eq:11.10}
\sfb_\ad \, |(\Omega_\ad^{\theta,\mh P} / \Omega^{\theta,\mh P} N_{\lambda_\ad})^* |^{-1} =
\sfb_\ad \, |\Omega^{\theta,\mh P} N_{\lambda_\ad}| \, |\Omega_\ad^{\theta,\mh P}|^{-1} .
\end{equation}
By Lemma \ref{lem:11.6} $\rho_\ad$ can be extended to a representation $\widetilde{\rho_\ad}$ of 
$(\mc A_\lambda)_{\rho_\ad}$. It follows from \eqref{eq:11.9} that every $\mc A_\lambda$-orbit 
of $G^{F_\omega}$-relevant cuspidal enhancements of $\lambda_\ad$ accounts for the same number
of $G^{F_\omega}$-relevant cuspidal enhancements of $\lambda$, namely
\begin{equation}\label{eq:11.11}
|(\Omega_\ad^\theta / \Omega^\theta \Omega_\ad^{\theta,\mh P})^*| =
\frac{|\Omega_\ad^\theta |}{|\Omega^\theta \Omega_\ad^{\theta,\mh P}|} =
\frac{|\Omega_\ad^\theta |\, |\Omega^{\theta,\mh P}|}{|\Omega^\theta| \, |\Omega_\ad^{\theta,\mh P}|}
= g' .
\end{equation}
By \cite[Theorem 1.2]{AMS1} $\sfb$ is the product of \eqref{eq:11.10} and \eqref{eq:11.11}.
\end{proof}

\begin{lemma}\label{lem:11.7}
Fix a $G^{F_\omega}$-relevant cuspidal $\rho_\ad \in \Irr (\mc A_{\lambda_\ad})$. There exists
a bijection between:
\begin{itemize}
\item the set of $\rho \in \Irr (\mc A_\lambda)$ that contain $\rho_\ad$,
\item the set of $G^{F_\omega}$-conjugacy classes of parahoric subgroups of $G$ that are 
$G_\ad^{F_\omega}$-conjugate to $\mh P$,
\end{itemize}
which is equivariant for $\Omega^\theta_\ad / \Omega^{\theta,\mh P}_\ad \Omega^\theta$ and with
respect to diagram automorphisms.
\end{lemma}
\begin{proof}
By \eqref{eq:11.16} and Lemma \ref{lem:11.6} every $\rho \in \Irr (\mc A_\lambda)$ 
which contains $\rho_\ad$ is of the form
\begin{equation}\label{eq:11.18}
\rho = \mr{ind}_{(\mc A_\lambda)_{\rho_\ad}}^{\mc A_\lambda} (\omega \otimes \widetilde{\rho_\ad})
\end{equation}
for a unique
\begin{equation}\label{eq:11.17}
\omega \Omega_\ad^{\theta,\mh P} \Omega^\theta \in \Omega^\theta_\ad / \Omega_\ad^{\theta,\mh P} 
\Omega^\theta = \Irr \big( (\Omega^\theta_\ad / \Omega_\ad^{\theta,\mh P} 
\Omega^\theta)^* \big).
\end{equation}
On the other hand, the group in \eqref{eq:11.17} parametrizes the $G^{F_\omega}$-conjugacy classes
of $G^{F_\omega}_\ad$-conjugates of $\mh P$. Decreeing that \eqref{eq:11.18} corresponds to
$\omega \mh P \omega^{-1}$, we obtain the required bijection and the 
$\Omega^\theta_\ad / \Omega^{\theta,\mh P}_\ad \Omega^\theta$-equivariance.

Notice that the set of (standard) parahoric subgroups of $G$ is the direct product of the analogous
sets for the almost direct simple factors of $\mb G$. Together with the explanation at the start of
the proof of Lemma \ref{lem:11.6}, this entails that for equivariance with respect to diagram
automorphisms it suffices to check the cases where $\mb G$ is almost simple. 

We only have to consider the Lie types $A_n, D_n$ and $E_6$, for the others do not admit nontrivial 
diagram automorphisms. Among these, we only have to look at the parahoric subgroups $\mh P$ with
$\Omega^{\theta,\mh P} \neq \Omega^\theta$, or equivalently at the $\mathsf J$ that are not
$\Omega^\theta$-stable. That takes care of the inner forms of type $A_n$ and of the outer forms of 
type $E_6$. For the outer forms of type $A_n$ ($\mathsf J = {}^2 \! A_s {}^2 \! A_t$), the inner
forms of $D_n$ ($\mathsf J = D_n$ and $\mathsf J = D_s D_t$), the outer forms of $D_n$ 
($\mathsf J = {}^2 D_t$ and $\mathsf J = {}^2 A_s$) and the inner forms of $E_6$ 
($\mathsf J = E_6$) see the remarks in Sections \ref{sec:outAn}, \ref{sec:innDn}, 
\ref{sec:outDn} and \ref{sec:innE6}.
\end{proof}

\section{Proof of main theorem for semisimple groups}
\label{sec:proofss}

Proposition \ref{prop:11.2} proves Theorem \ref{thm:B} for unramified adjoint groups. 
When we replace an adjoint group by a group in the same isogeny class, several unipotent 
cuspidal representations of $G_\ad^{F_\omega}$ coalesce and then decompose as a sum of 
$g'$ irreducible unipotent cuspidal representations of $G^{F_\omega}$. Similarly, several 
enhanced L-parameters for $G_\ad^{F_\omega}$ coincide, and they can be further enhanced 
in $g'$ ways to elements of $\Phi_\nr (G^{F_\omega})_\cusp$. 

From \eqref{eq:2.10} we see that the $G_\ad^{F_\omega}$-representations which contain 
$\pi \in \Irr (G^{F_\omega})_{[\mh P,\sigma]}$ form precisely one orbit for 
$(\Omega^\theta_\ad / \Omega^\theta)^*$. The action of $(\Omega^\theta_\ad)^*$ 
on $\Irr (G_\ad^{F_\omega})_{\unip,\cusp}$ reduces to an action of $(\Omega^\theta)^*$ on
$\Irr (G^{F_\omega})_{\unip,\cusp}$, and the stabilizers become 
\[
(\Omega^\theta / \Omega_\ad^{\theta,\mh P} \cap \Omega^\theta)^* = 
(\Omega^\theta / \Omega^{\theta,\mh P})^*.
\]
A bijection 
\begin{equation}\label{eq:11.19}
(\Omega^{\theta,\mh P})^* \cong (\Omega^\theta)^* / (\Omega^\theta / \Omega^{\theta,\mh P})^*
\longrightarrow \Irr (G^{F_\omega})_{[\mh P,\sigma]}
\end{equation}
can be specified by fixing an extension of $\sigma$ from $\mh P^{F_\omega}$ to 
$N_{G^{F_\omega}}(\mh P^{F_\omega})$ \cite[\S 2]{Opd2}.
In particular $\Irr (G^{F_\omega})_{[\mh P,\sigma]}$ forms exactly one 
$(\Omega^\theta)^*$-orbit. Consequently the $(\Omega^\theta)^*$-orbits on the set of 
solutions $\pi$ of \eqref{eq:fdegisgammafactor} are parametrized by the 
$G^{F_\omega}$-conjugacy classes of $(\mh P', \sigma')$ with fdeg$(\sigma') = \mr{fdeg}
(\sigma)$. There are $g' \sfb' = g' \phi (n_s)$ of those.

Recall that $\lambda$ is the image of $\lambda_\ad$ under ${\mb G^\vee}_\SC \to \mb G^\vee$.
When $\mb G$ is simple, Theorem \ref{thm:A} says that $(\Omega^\theta)^* \lambda$ is 
the unique $(\Omega^\theta)^*$-orbit of L-parameters for $G^{F_\omega}$ with for each 
$K$-simple factor of $\mb G$ the same 
adjoint $\gamma$-factor as $\lambda_\ad$ (up to a rational number). It follows quickly 
from the definitions that adjoint $\gamma$-factors are multiplicative for almost direct 
products of reductive groups, cf. \cite[\S 3]{GrRe}. From \eqref{eq:2.21} we see that 
the formal degrees of supercuspidal unipotent representations are also multiplicative for
almost direct products, up to some rational numbers $C_\pi$ (which can be made 
explicit, see \cite[\S 1]{HII} and \cite[\S 4.6]{Opd2}). Hence the uniqueness of 
$(\Omega^\theta)^* \lambda$ in the above sense also holds for semisimple $\mb G$, provided 
we impose the compatibility with almost direct products from \eqref{eq:2.20}.
Together with \eqref{eq:11.19} this proves Theorem \ref{thm:B}.(2).

By Lemma \ref{lem:11.3} the $\lambda_\ad$ which coalesce to $\lambda$ form precisely 
one orbit under\\ $(\Omega^\theta_\ad / \Omega^\theta \cap N_{\lambda_\ad})^*$. 
From the proof of Lemma \ref{lem:11.4} we see that the restriction of a relevant 
cuspidal representation $\rho$ of $\mc A_\lambda$ to $\mc A_{\lambda_\ad}$ contains 
precisely the enhancements of $\lambda_\ad$ in one $(\Omega^\theta_\ad / 
\Omega^\theta N_{\lambda_\ad})^*$-orbit. 

From Lemma \ref{lem:11.3}.(2) we know that the $(\Omega^\theta)^*$-stabilizer of $(\lambda,\rho)$
is contained in $(\Omega^\theta / \Omega^\theta \cap N_{\lambda_\ad})^*$, and from the proof
of Lemma \ref{lem:11.5} we see that it must stabilize the $\mc A_\lambda$-orbit of a 
$\rho_\ad \in \Irr (\mc A_{\lambda_\ad})$. By Theorem \ref{thm:B}.(3) for $\mb G_\ad$ and by
Lemma \ref{lem:11.4}, the $(\Omega^\theta_\ad)^*$-stabilizer of that orbit is
\begin{equation}\label{eq:11.14}
(\Omega^\theta_\ad / \Omega^{\theta,\mh P}_\ad)^* 
(\Omega^\theta_\ad / \Omega^\theta N_{\lambda_\ad})^* .
\end{equation}
Hence the $(\Omega^\theta)^*$-stabilizer of that orbit is the image of \eqref{eq:11.14} in
$(\Omega^\theta)^*$, that is, $(\Omega^\theta / \Omega^{\theta,\mh P})^*$.
From \eqref{eq:11.11} we know that the different representations $\rho$ of $\mc A_\lambda$ 
associated to the orbit of $\rho_\ad$ are parametrized by $\Omega_\ad^\theta / \Omega^\theta 
\Omega_\ad^{\theta,\mh P}$. Elements of $(\Omega^\theta)^*$ exert no influence on the last
group, so  $(\Omega^\theta / \Omega^{\theta,\mh P})^*$ is precisely the stabilizer of
$(\lambda,\rho)$ in $(\Omega^\theta)^*$. This proves Theorem \ref{thm:B}.(3).

Part (4) can be observed from the adjoint case and \eqref{eq:11.7}. For Part (5), we note
by \eqref{eq:11.7} and \eqref{eq:11.8}
\[
\frac{\sfa' \sfb'}{\sfa'_\ad \sfb'_\ad} =
\frac{g' \, |\Omega^{\theta,\mh P}|}{|\Omega^{\theta,\mh P}_\ad|} .
\]
On the other hand, by Lemmas \ref{lem:11.3} and \ref{lem:11.5}
\begin{multline*}
\frac{\sfa \sfb}{\sfa_\ad \sfb_\ad} = \frac{|N_{\lambda_\ad} \cap \Omega^\theta| \, g'}{|N_{\lambda_\ad}| 
\, [\Omega^{\theta,\mh P}_\ad : \Omega^{\theta,\mh P} N_{\lambda_\ad}]}
= \frac{|N_{\lambda_\ad} \cap \Omega^\theta| \, g'  |\Omega^{\theta,\mh P}| \, |N_{\lambda_\ad}|}{
|N_{\lambda_\ad}| \, |\Omega^{\theta,\mh P}_\ad| \, |\Omega^{\theta,\mh P} \cap N_{\lambda_\ad}|}
=  \frac{g' \, |\Omega^{\theta,\mh P}|}{|\Omega^{\theta,\mh P}_\ad|} .
\end{multline*}
Thus Theorem \ref{thm:B}.(5) for $\mb G_\ad$ implies that $\sfa' \sfb' = \sfa \sfb$.
 
Now we can construct a LLC for $\Irr (G^{F_\omega})_{\unip,\cusp}$. Every $\pi$ 
in there corresponds to a unique $(\Omega^\theta_\ad / \Omega^\theta)^*$-orbit in 
$\Irr (G_\ad^{F_\omega})_{\unip,\cusp}$. Then Proposition \ref{prop:11.2} gives an orbit
\begin{equation}\label{eq:11.15}
(\Omega^\theta_\ad / \Omega^\theta)^* (\lambda_\ad,\rho_\ad) \subset 
\Phi_\nr (G_\ad^{F_\omega})_\cusp .
\end{equation}
By Lemma \ref{lem:11.3}.(2) that determines a single $\lambda \in \Phi_\nr (G^{F_\omega})$
and from Lemma \ref{lem:11.4} we get one $\mc A_\lambda$-orbit  
\begin{equation}\label{eq:11.20}
(\Omega^\theta_\ad / \Omega^\theta N_{\lambda_\ad})^* \rho_\ad \subset \Irr (\mc A_{\lambda_\ad}) .
\end{equation}
But \eqref{eq:11.15} does not yet determine a unique representation of $\mc A_\lambda$, in general
several extensions of $\rho_\ad$ to $\rho \in \Irr (\mc A_\lambda)$ are possible. 
By Lemma \ref{lem:11.7} we can match these $\rho$'s with the $G^{F_\omega}$-conjugacy classes
of parahoric subgroups of $G$ that are $G_\ad^{F_\omega}$-conjugate to $\mh P$, in a way which
is equivariant for $\Omega^\theta_\ad$ and for diagram automorphisms. For $\pi \in 
\Irr (G^{F_\omega})_{[\mh P,\sigma]}$ we now choose the $\rho$
which corresponds to the $G^{F_\omega}$-conjugacy class of the $\mh P$.
Above we saw that $\pi$ and $(\lambda,\rho)$ have the same isotropy group in
$(\Omega^\theta)^*$, so we get a well-defined map from $(\Omega^\theta)^* \pi$ to
$(\Omega^\theta)^* (\lambda,\rho)$. This map is equivariant for $(\Omega^\theta)^*$ and for
all diagram automorphisms that stabilize the domain.

For all $G^{F_\omega}$-representations in the Out$(G^{F_\omega})$-orbit of 
$(\Omega^\theta)^* \pi$, we define the LLC by imposing equivariance with respect to diagram 
automorphisms. If $\tau$ is a diagram automorphism of $\mb G$ with $\tau (\omega) \neq \omega$, 
then for $z \in (\Omega^\theta)^* \cong X_\Wr (G^{F_\omega})$ we define the enhanced L-parameter 
of $\tau^* (z \otimes \pi) \in \Irr (G^{F_{\tau (\omega)}})_{\cusp,\unip}$ to be 
$(\tau (z \lambda), \tau^* \rho)$.

For another $\pi' \in \Irr (G^{F_\omega})_{\unip,\cusp}$ we construct $(\lambda',\rho') \in
\Phi_\nr (G^{F_\omega})_\cusp$ in the same way. We only must take care that, if $\lambda' = \lambda$,
we select a $\rho'$ that we did not use already. Since $\sfa' \sfb' = \sfa \sfb$, this procedure
yields a bijection $\Irr (G^{F_\omega})_{\unip,\cusp} \to \Phi_\nr (G^{F_\omega})_\cusp$.

As explained above, at the same time this determines bijections 
\[
\Irr (G^{F_{\tau (\omega)}})_{\unip,\cusp} \to \Phi_\nr (G^{F_{\tau (\omega)}})_\cusp 
\]
for all diagram automorphisms $\tau$ of $\mb G$. The union of all these bijections is the LLC for
all the involved representations, and then it is equivariant with respect to diagram automorphisms.
 
From parts (1) and (3) we get the number of $(\Omega^\theta)^*$-orbits on the Galois side of the LLC, 
namely $g' \phi (n_s)$, just as on the $p$-adic side. Since the L-parameters with the same adjoint 
$\gamma$-factor form just one $(\Omega^\theta)^*$-orbit, the orbits of enhanced L-parameters can 
be parametrized by enhancements of $\lambda$, just as in the adjoint case.

\section{Proof of main theorem for reductive groups}
\label{sec:mainred}

First we check that Theorem \ref{thm:C} is valid for any $K$-torus, ramified or unramified.
Of course, the local Langlands correspondence for $K$-tori is well-known, due to Langlands.

\begin{prop}\label{prop:15.1}
Let $\mb T$ be a $K$-torus and write $T = \mb T (K_\nr)$. 
\begin{enumerate}
\item The unipotent representations of $\mb T (K)$ are precisely its weakly unramified characters.
\item The LLC for $\Irr (\mb T (K))_\unip$ is injective, and has as image the collection of
L-parameters 
\[
\lambda : \mb W_K \times SL_2 (\CC) \to \mb T^\vee \rtimes \mb W_K
\]
such that $\lambda (w,x) = (1,w)$ for all $w \in \mb I_K, x \in SL_2 (\CC)$. 
\item The map from (2) is equivariant for $({\Omega_{\mb I_K}}^*)_\Frob$ and with respect to\\ 
$\mb W_K$-automorphisms of the root datum.
\end{enumerate}
\end{prop}
The target in part (2) is the analogue of $\Phi_\nr (G^{F_\omega})$ for tori. 
As $\mc A_\lambda = 1$, we can ignore enhancements here.
\begin{proof}
(1) The kernel $T_1$ of the Kottwitz homomorphism \cite[\S 7]{Ko}
\[
T \to X^* ((\mb T^\vee)^{\mb I_K}) 
\]
has finite index in the maximal bounded subgroup of $T$. By \cite[Appendix, Lemma 5]{HR}
$T_1$ equals the unique parahoric subgroup of $T$. Then $\mb T (K)_1 = T_1^\Frob$ is the unique
parahoric subgroup of $\mb T (K) = T^\Frob$. The finite reductive quotient $\overline{T}^\Frob$
is again a torus, so its only cuspidal unipotent representation is the trivial representation.
Hence the unipotent $\mb T(K)$-representations are precisely the characters of $\mb T(K)$ that
are trivial on $\mb T(K)_1$, that is, the weakly unramified characters. \\
(2) It is known (e.g. from \cite[\S 3.3.1]{Hai}) that the LLC for tori puts $X_\Wr (\mb T(K))$
in bijection with $(Z(\mb T^\vee)^{\mb I_K})_\Frob$.\\
(3) From $Z(\mb T^\vee) \cong \Omega^*$ we see that 
\begin{equation}\label{eq:15.1}
X_\Wr (\mb T (K)) \cong (Z(\mb T^\vee)^{\mb I_K})_\Frob \cong (\Omega^*)^{\mb I_K}_\Frob \cong
({\Omega_{\mb I_K}}^*)_\Frob .
\end{equation}
Now it is clear that the LLC for $\Irr (\mb T (K))_\unip$ is equivariant under \eqref{eq:15.1}.

Since the LLC for tori is natural, it is also equivariant with respect to all automorphisms of
$X^* (\mb T)$ that define automorphisms of $\mb T (K)$. These are precisely the automorphisms
of the root datum $(X^* (\mb T), \emptyset, X_* (\mb T),\emptyset,\emptyset)$ that commute
with $\mb W_K$.
\end{proof}

Next we consider the case that $\mb G$ is an unramified reductive $K$-group such that
$Z (\mb G)^\circ (K)$ is anisotropic. By \cite[Theorem BTR]{Pra}, that happens if and only if
$Z(G^{F_\omega})$ is compact. Equivalently, every unramified character of $G^{F_\omega}$ is 
trivial. We will divide the proof of Theorem \ref{thm:C} for such groups over a sequence of lemmas.

\begin{lemma}\label{lem:15.2}
Suppose that $Z(\mb G)^\circ$ is $K$-anisotropic, and let $\Omega_\der$ be the 
$\Omega$-group for $\mb G_\der$. Then $\Omega^\theta_\der = \Omega^\theta$.
\end{lemma}
\begin{proof}
By \eqref{eq:2.1} $\Omega_\der^\theta \cong X^* (Z(\mb G^\vee)_\theta / Z(\mb G^\vee)^\circ)$. 
Since $Z(\mb G)^\circ$ is $K$-anisotropic, so is $\mb G / \mb G_\der$, and
\[
0 = X^* (\mb G / \mb G_\der)^\Frob = X_* (Z(\mb G^\vee)^\circ )^\theta .
\]
This implies
\begin{equation}\label{eq:15.2}
(1 - \theta) Z(\mb G^\vee)^\circ = Z(\mb G^\vee)^\circ ,
\end{equation}
so $\Omega^\theta = X^* (Z(\mb G^\vee)_\theta) = X^* (Z(\mb G^\vee)_\theta / Z(\mb G^\vee)^\circ )$.
\end{proof}

\begin{lemma}\label{lem:15.3} 
Suppose that $Z(\mb G)^\circ$ is $K$-anisotropic.
The inclusion $G_\der^{F_\omega} \to G^{F_\omega}$ induces a bijection
\[
\Irr (G^{F_\omega})_{\cusp,\unip} \longrightarrow \Irr (G_\der^{F_\omega})_{\cusp,\unip} .
\]
\end{lemma}
\begin{proof}
These two groups have the same affine Dynkin diagram $\mathsf I$. For any proper subset of 
$\mathsf I$, the two associated parahoric subgroups, of $G$ and of $G_\der$, give rise to
connected reductive $\mf F$-groups of the same Lie type. The collection of (cuspidal) unipotent 
representations of a connected reductive group over a (fixed) finite field depends only
on the Lie type of the group \cite[Proposition 3.15]{Lus-Chevalley}. Hence any cuspidal 
unipotent $\sigma \in \Irr (\mh P_\der^{F_\omega})$ extends in a unique way to a 
representation $\sigma'$ of $\mh P^{F_\omega}$. More precisely, both $\sigma$ and $\sigma'$ 
factor via the canonical map to $\mh P_\ad^{F_\omega}$.

By Lemma \ref{lem:15.2} $G^{F_\omega}$ and $G_\der^{F_\omega}$ also have the same 
$\Omega^\theta$-group. From \eqref{eq:2.6} we see that 
\begin{equation}\label{eq:15.12}
N_{G^{F_\omega}}(\mh P^{F_\omega}) = 
N_{G_\der^{F_\omega}}(\mh P_\der^{F_\omega}) \mh P^{F_\omega} .
\end{equation}
After \eqref{eq:2.6} we checked that there exists an extension $\sigma^N$ of $\sigma$ to
a representation of $N_{G_\der^{F_\omega}}(\mh P^{F_\omega})$. Since $\sigma$ and $\sigma'$
factor via $\mh P_\ad^{F_\omega}$, $\sigma^N$ extends uniquely to a representation of 
\eqref{eq:15.12}.

Now the classification of supercuspidal unipotent representations, as in \eqref{eq:2.8} and 
further, is the same for $G^{F_\omega}$ and for $G_\der^{F_\omega}$. The explicit form 
\eqref{eq:2.10} shows that the ensuing bijection is induced by $G_\der^{F_\omega} \to G^{F_\omega}$.
\end{proof}

\begin{lemma}\label{lem:15.4}
Suppose that $Z(\mb G)^\circ$ is $K$-anisotropic. The canonical map\\
$\mb G^\vee \to \mb G^\vee / Z (\mb G^\vee)^\circ$ induces a bijection 
$\Phi_\nr (G^{F_\omega}) \to \Phi_\nr (G_\der^{F_\omega})$.
\end{lemma}
\begin{proof}
Suppose that $\lambda, \lambda' \in \Phi_\nr (G^{F_\omega})$ become equal in
$\Phi_\nr (G_\der^{F_\omega})$. Then there exists a $g \in \mb G^\vee$ such that
$g \lambda' g^{-1} = \lambda$ as maps $\mb W_K \times SL_2 (\CC) \to \mb G^\vee / Z(\mb G^\vee)^\circ
\rtimes \mb W_K$. In particular $g \lambda' (\Frob) g^{-1} = z_1 \lambda (\Frob)$
for some $z_1 \in Z(\mb G^\vee)^\circ$. By \eqref{eq:15.2}
we can find $z_2 \in Z(\mb G^\vee)^\circ$ with $z_2^{-1} \theta (z_2) = z_1$. Then
\[
z_2 g \lambda' (\Frob) g^{-1} z_2^{-1} = z_2 z_1 \lambda (\Frob) z_2^{-1} = 
z_2 z_1 \theta (z_2^{-1}) \lambda (\Frob) = \lambda (\Frob) .
\]
Replace $\lambda'$ by the equivalent parameter $\lambda'' = z_2 g \lambda' g^{-1} z_2^{-1}$.
These parameters are unramified, so $\lambda'' |_{\mb W_K} = \lambda |_{\mb W_K}$. But
$\lambda |_{SL_2 (\CC)} = \lambda'' |_{SL_2 (\CC)}$ still only holds as maps
$SL_2 (\CC) \to \mb G^\vee / Z(\mb G^\vee)^\circ$. In any case, $\lambda \big(1, \left(
\begin{smallmatrix} 1 & 1 \\ 0 & 1 \end{smallmatrix} \right) \big)$ and
$\lambda'' \big(1, \left( \begin{smallmatrix} 1 & 1 \\ 0 & 1 \end{smallmatrix} \right) \big)$ 
determine the same unipotent class (in $\mb G^\vee$ and in $\mb G^\vee / Z(\mb G^\vee)^\circ$).
Consequenty $\lambda$ and $\lambda''$ are $\mb G^\vee$-conjugate.

Conversely, consider a $\tilde \lambda \in \Phi_\nr (G_\der^{F_\omega})$. We may assume that
$\tilde \lambda (\Frob) = s \theta Z(\mb G^\vee)^\circ$ for some $s \in S^\vee$. Then $s \theta$
and $\tilde \lambda (\Frob)$ centralize the same subalgebra of 
\[
\mr{Lie}(\mb G^\vee) = \mr{Lie}({\mb G^\vee}_\der) \oplus \mr{Lie}(Z(\mb G^\vee)^\circ) .
\] 
As d$\tilde \lambda (\mathfrak{sl}_2 (\CC))$ is contained in 
\[
\mr{Lie} \big( (\mb G^\vee)^{s \theta} \big) = \mr{Lie}\big( Z_{\mb G^\vee}(s \theta) \big) = 
\mr{Lie}\big( Z_{\mb G^\vee}(\tilde \lambda (\Frob)) \big) ,
\] 
we can lift d$\tilde \lambda (\mathfrak{sl}_2 (\CC))$ to a homomorphism 
$\lambda : SL_2 (\CC) \to Z_{\mb G^\vee}(s \theta)^\circ$. Together with $\lambda (\Frob) :=
s \theta$ this defines a preimage of $\tilde \lambda$ in $\Phi_\nr (G^{F_\omega})$. 
\end{proof}

\begin{lemma}\label{lem:15.5}
Suppose that $Z(\mb G)^\circ$ is $K$-anisotropic. Let $\lambda \in \Phi_\nr (G^{F_\omega})$ 
and let $\lambda_\der$ be its image in $\Phi_\nr (G_\der^{F_\omega})$. Then
$\mc A_\lambda = \mc A_{\lambda_\der}$.
\end{lemma}
\begin{proof}
Recall the construction of $\mc A_\lambda$ from \eqref{eq:2.11} and \eqref{eq:2.12}. It says
that $\mc A_{\lambda_\der}$ is the component group of 
\[
Z^1_{(\mb G_\der)^\vee_\SC} (\lambda_\der) = \big\{ g \in (\mb G_\der)^\vee_\SC : g \lambda_\der 
g^{-1} = \lambda_\der b \text{ for some } b \in B^1 (\mb W_K, Z({\mb G_\der}^\vee)) \big\} .
\]
From ${\mb G_\der}^\vee = \mb G^\vee / Z(\mb G^\vee)^\circ$ we see that 
$(\mb G_\der)^\vee_\SC = {\mb G^\vee}_\SC$.
Since $\lambda$ is unramified, the difference with $\lambda_\der$ resides only in the image
of the Frobenius element (see the second half of the proof of Lemma \ref{lem:15.4}).
To centralize $\lambda_\der$ means to centralize $\lambda$, up to adjusting $\lambda (\Frob)$
by an element of $Z(\mb G^\vee)^\circ$. Together with \eqref{eq:15.2} this implies that
\[
Z^1_{(\mb G_\der)^\vee_\SC} (\lambda_\der) = \big\{ g \in {\mb G^\vee}_\SC : g \lambda 
g^{-1} = \lambda b \text{ for some } b \in B^1 (\mb W_K, Z(\mb G^\vee)) \big\} = 
Z^1_{{\mb G^\vee}_\SC}(\lambda) .
\]
In particular 
\[
\mc A_{\lambda_\der} = \pi_0 \big( Z^1_{(G_\der)^\vee_\SC} (\lambda_\der) \big)
= \pi_0 \big( Z^1_{{\mb G^\vee}_\SC}(\lambda) \big) = \mc A_\lambda . \qedhere
\]
\end{proof}

\begin{prop}\label{prop:15.8}
Theorem \ref{thm:C} holds whenever $Z(\mb G)^\circ$ is $K$-anisotropic. 
\end{prop}
\begin{proof}
Lemmas \ref{lem:15.2}, \ref{lem:15.3}, \ref{lem:15.4}, \ref{lem:15.5} and Theorem 
\ref{thm:B} prove parts (5) and (6) of Theorem \ref{thm:C}, as well as the unicity up to
weakly unramified characters.
The equivariance properties (2) and (3) in Theorem \ref{thm:C} follow from the semisimple 
case, because the isomorphisms in the aforementioned lemmas are natural. 

Assume that $\mb G$ is the almost direct product of $K$-groups $\mb G_1$ and $\mb G_2$. 
Then 
\begin{equation}\label{eq:15.10}
\mb G_\der = \mb G_{1,\der} \mb G_{2,\der} \text{ and }
Z(\mb G)^\circ = Z(\mb G_1)^\circ Z (\mb G_2)^\circ
\end{equation}
are also almost direct products, and there are epimorphisms of $K$-groups
\begin{equation}\label{eq:15.11}
\mb G_{1,\der} \times Z(\mb G_1)^\circ \times \mb G_{2,\der} \times Z (\mb G_2)^\circ 
\longrightarrow \mb G_\der \times Z(\mb G)^\circ \longrightarrow \mb G .
\end{equation}
Notice that the connected centres of $\mb G_1$ and $\mb G_2$ are $K$-anisotropic. 
By Lemma \ref{lem:15.3} $Z(\mb G)^\circ (K)$, and the $Z(\mb G_i)^\circ (K)$ have unique 
irreducible unipotent representations, namely the trivial representation of each of these
groups. That and Theorem \ref{thm:B}.(1) show 
that our instances of the LLC for $\mb G (K)$ and the $\mb G_i (K)$ are compatible 
with the almost direct products \eqref{eq:15.10}. For the same reason they are compatible 
with the second map in \eqref{eq:15.11}. The composition of the maps in \eqref{eq:15.11} 
factors via $\mb G_1 \times \mb G_2 \to \mb G$, so our LLC is also compatible with that 
almost direct product.

Conversely, Lemma \ref{lem:15.3} and Theorem \ref{thm:B}.(4) leave no choice for the LLC 
in this case, it just has to be the same as for $\Irr (\mb G_\der^\omega (K))_{\cusp,\unip}$. 
We already know from Theorem \ref{thm:B} that for the latter the L-parameters are uniquely 
determined modulo $(\Omega^\theta)^*$ by properties (1), (2) and (4) in Theorem \ref{thm:C}. 
Hence the same goes for the LCC for $\Irr (\mb G^\omega (K))_{\cusp,\unip}$. 
\end{proof}

Now $\mb G$ may be any unramified reductive $K$-group. Let $Z(\mb G^\omega)_s$ be the 
maximal $K$-split central torus of $\mb G^\omega$. Recall that the $K$-torus 
$Z(\mb G^\omega)^\circ$ is the almost direct product of $Z(\mb G^\omega)_s$ and a 
$K$-anisotropic torus $Z(\mb G^\omega)_a$ \cite[Proposition 13.2.4]{Spr}.
These central subgroups do not depend on $\omega$, so may denote them simply by 
$Z(\mb G)_s$ and $Z(\mb G)_a$.

\begin{lemma}\label{lem:15.7}
Any cuspidal unipotent $\sigma \in \Irr (\mh P^{F_\omega})$ can be extended to 
$N_{G^{F_\omega}}(\mh P^{F_\omega})$.
\end{lemma}
\begin{proof}
Recall that by Hilbert 90 the continuous Galois cohomology group $H^1_c (K,Z (\mb G^\omega)_s)$ 
is trivial. The long exact sequence in Galois cohomology yields a short exact sequence
\begin{equation}\label{eq:15.9}
1 \to Z(\mb G^\omega)_s (K) \to \mb G^\omega (K) \to 
\big( \mb G^\omega / Z(\mb G^\omega)_s \big) (K) \to 1 .
\end{equation}
The restriction of $\sigma$ to $\mh P^{F_\omega} \cap Z(\mb G^\omega)_s (K) = Z(\mb G)_s (\mf o_K)$ 
is inflated from a unipotent representation of $Z(\mb G)_s (\mf F)$, so it is a multiple of 
the trivial representation. Thus we can extend $\sigma$ trivially across $Z(\mb G)_s (K)$,
making it a representation of 
\[
\mh P^{F_\omega} Z(\mb G)_s (K) / Z(\mb G)_s (K) \cong 
\mh P_{\mb G^\omega / Z(\mb G^\omega)_s}^{F_\omega} .
\]
By \cite[Proposition 13.2.2]{Spr} the connected centre of $\mb G^\omega / Z(\mb G^\omega)_s$ 
is $K$-anisotropic. From Lemma \ref{lem:15.3} we know that $\sigma$ extends canonically to
a representation of 
\[
N_{\big( \mb G^\omega / Z(\mb G^\omega)_s \big) (K)} \big( \mh P_{\mb G^\omega / Z(\mb G^\omega)_s
}^{F_\omega} \big) \cong N_{G^{F_\omega}}(\mh P^{F_\omega}) / Z(\mb G^\omega)_s (K) .
\]
This can be regarded as the required extension of $\sigma$. 
\end{proof}

With a similar argument we can prove a part of Theorem \ref{thm:C} for reductive groups.

Every (cuspidal) unipotent representation of $\mb G^\omega (K)$ restricts to a
unipotent character of $Z(\mb G^\omega)_s (K)$. From Proposition \ref{prop:15.1}.(1)
we know that those are precisely the weakly unramified characters of $Z(\mb G^\omega)_s (K)$. 
This torus is $K$-split, so all its weakly unramified characters are in fact unramified.
Since $\CC^\times$ is divisible, every $\chi \in X_\nr (Z(\mb G^\omega)_s (K))$ can be
extended to an unramified character of $\mb G^\omega (K)$. Thus every $\pi \in 
\Irr (\mb G^\omega (K) )_{\cusp,\unip}$ can be made trivial on $Z(\mb G^\omega)_s (K)$ by
an unramified twist:
\begin{equation}\label{eq:15.13}
\pi = \chi \otimes \pi' \text{ with } \chi \in X_\nr (\mb G^\omega (K)) \text{ and }
\pi' \in \Irr (\mb G^\omega / Z(\mb G^\omega)_s)(K) .
\end{equation}
By the functoriality of the Kottwitz homomorphism, \eqref{eq:15.9} induces a short exact sequence
\[
1 \to X_\Wr  \big( \big( \mb G^\omega / Z(\mb G^\omega)_s \big) (K) \big) \to 
X_\Wr  (\mb G^\omega (K)) \to X_\nr (Z (\mb G^\omega)_s (K)) \to 1 .
\]
Thus we can reformulate the above as a bijection
\begin{equation}\label{eq:15.3}
\begin{split}
\Irr \big( \big(\mb G^\omega / Z(\mb G^\omega)_s \big) (K) \big)_{\cusp,\unip} 
\underset{X_\Wr  ( ( \mb G^\omega / Z(\mb G^\omega)_s ) (K) )}{\times}
X_\Wr  \big( \mb G^\omega (K) \big) \\
\longrightarrow \Irr \big( \mb G^\omega (K) \big)_{\cusp,\unip} . 
\end{split}
\end{equation}
On the Galois side of the LLC there is a short exact sequence
\begin{equation}\label{eq:15.5}
1 \to {}^L (\mb G / Z(\mb G)_s) \to {}^L \mb G \to  
{}^L Z(\mb G)_s = {Z(\mb G)_s}^\vee \times \mb W_K \to 1.
\end{equation}
This induces maps between L-parameters for these groups. 
It also induces a short exact sequence
\begin{equation}\label{eq:15.6}
1 \to Z \big( (\mb G / Z(\mb G)_s)^\vee \big)_\theta \to Z(\mb G^\vee)_\theta 
\to \big( {Z(\mb G)_s}^\vee \big)_\theta = {Z(\mb G)_s}^\vee \to 1 ,
\end{equation}
whose terms can be interpreted as the sets of weakly unramified characters of the 
associated $K$-groups (or of their inner forms). As 
$\Phi_\nr^2 (Z(\mb G^\omega)_s) \cong {Z(\mb G^\omega)_s}^\vee$, 
\eqref{eq:15.6} and \eqref{eq:15.5} show that the map
\[
\Phi_\nr^2 (\mb G^\omega (K)) \to \Phi_\nr^2 (Z(\mb G^\omega)_s)
\]
is surjective with fibres $\Phi_\nr^2 ( \mb G^\omega / Z(\mb G^\omega)_s )$. 
With \eqref{eq:15.6} we obtain a bijection
\begin{equation}\label{eq:15.4}
\Phi_\nr^2 \big( \big(\mb G^\omega / Z(\mb G^\omega)_s \big) (K) \big) 
\underset{Z ( (\mb G / Z(\mb G)_s)^\vee )_\theta}{\times}
Z(\mb G^\vee)_\theta \longrightarrow \Phi_\nr^2 (\mb G^\omega (K)) .
\end{equation}

\begin{lemma}\label{lem:15.6}
$B^1 \big( \mb W_K, Z \big( (\mb G / Z(\mb G)_s)^\vee \big) \big) =
B^1 \big( \mb W_K, Z(\mb G^\vee) \big)$
\end{lemma}
\begin{proof}
Consider the short exact sequence of $K$-groups
\[
1 \to \mb G^\omega_\der Z(\mb G^\omega)_a \to \mb G^\omega \to
\mb T := \mb G^\omega \big/ \big(\mb G^\omega_\der Z(\mb G^\omega)_a \big) \to 1 .
\]
By \cite[Proposition 13.2.2]{Spr} $\mb T$ is a $K$-split torus. In the short exact
sequences of complex groups
\[
\begin{array}{ccccccccc}
1 & \to & \big( \mb G / Z(\mb G)_s \big)^\vee & \to & \mb G^\vee & \to & 
{Z(\mb G)_s}^\vee & \to & 1, \\
1 & \to & \mb T^\vee & \to & \mb G^\vee & \to & 
\big( \mb G_\der Z(\mb G)_a \big)^\vee & \to & 1 ,
\end{array}
\]
the Lie algebra of $\mb T^\vee$ maps isomorphically to the Lie algebra of 
${Z(\mb G)_s}^\vee$. Hence
\[
\mb G^\vee = \mb T^\vee \big( \mb G / Z(\mb G)_s \big)^\vee \quad \text{and} \quad
Z(\mb G^\vee) = \mb T^\vee Z \big( \big( \mb G / Z(\mb G)_s \big)^\vee \big) .
\]
The last equation entails that every element $b \in B^1 \big( \mb W_K, Z(\mb G^\vee) \big)$ 
is of the form 
\[
b(w) = t z w z^{-1} t^{-1} w^{-1} \quad \text{for some } t \in \mb T^\vee, z \in 
Z \big( (\mb G / Z(\mb G)_s )^\vee \big).
\]
As $\mb T^\vee$ is fixed by $\mb W_F$ and central in $\mb G^\vee$, $b(w) = z w z^{-1} w^{-1}$.
This says that $b \in B^1 \big( \mb W_K, Z \big( (\mb G / Z(\mb G)_s)^\vee \big) \big)$.
\end{proof}

\emph{Proof of Theorem \ref{thm:C}.}\\
Both for $\mb G^\omega (K)$ and for $\big(\mb G^\omega / Z(\mb G^\omega)_s \big) (K)$
the component groups of L-parameters are computed in the simply connected cover of
${\mb G^\vee}_\der = {\big(\mb G / Z(\mb G)_s \big)^\vee}_\der$,
see \eqref{eq:2.11}. By Lemma \ref{lem:15.6} and \eqref{eq:2.12} the group $\mc A_\lambda$
for $\lambda \in \Phi \big( \big(\mb G^\omega / Z(\mb G^\omega)_s \big) (K) \big)$ is the same
as the component group for $\lambda$ as L-parameter for $\mb G^\omega (K)$. Any
$z \in Z(\mb G^\vee)_\theta $ is made from central elements of $\mb G^\vee$, 
so $\mc A_{z \lambda}$ for $\mb G^\omega (K)$ equals $\mc A_\lambda$ for
$\mb G^\omega (K)$, and then also for $\big(\mb G^\omega / Z(\mb G^\omega)_s \big) (K)$.
This says that \eqref{eq:15.4} extends to a bijection between the spaces of enhanced 
L-parameters. Recall from \eqref{eq:2.13} that cuspidality of the enhancements is 
defined via the group $Z^1_{{(\mb G^\omega)^\vee}_\SC}(\lambda (\mb W_K))$, which is the
same for $\mb G^\omega (K)$ as for  $\big(\mb G^\omega / Z(\mb G^\omega)_s \big) (K)$.
Hence \eqref{eq:15.4} extends to a bijection
\begin{equation}\label{eq:15.8}
\Phi_\nr \big( \big(\mb G^\omega / Z(\mb G^\omega)_s \big) (K) \big)_\cusp 
\underset{Z ( (\mb G / Z(\mb G)_s)^\vee )_\theta}{\times}
Z(\mb G^\vee)_\theta \longrightarrow \Phi_\nr (\mb G^\omega (K))_\cusp .
\end{equation}
As $\mb G^\omega / Z(\mb G^\omega)_s$ has $K$-anisotropic centre, we already know 
Theorem \ref{thm:C} for that group from Proposition \ref{prop:15.8}. Using that and
comparing \eqref{eq:15.8} with \eqref{eq:15.3}, we obtain a bijection
\begin{equation}\label{eq:15.7}
\Irr \big( \mb G^\omega (K) \big)_{\cusp,\unip} \longleftrightarrow 
\Phi_\nr (\mb G^\omega (K))_\cusp .
\end{equation}
By construction \eqref{eq:15.7} satisfies parts (2), (5) and (6) of Theorem \ref{thm:C},
while part (1) does not apply. What happens for $Z(\mb G^\omega)_s (K)$ in \eqref{eq:15.3} 
and \eqref{eq:15.8} is completely determined by the LLC for tori, so any non-canonical 
choices left in \eqref{eq:15.7} come from $\big(\mb G^\omega / Z(\mb G^\omega)_s \big) (K)$. 
By Theorem \ref{thm:C} for the latter group, the only free choices are twists by weakly 
unramified characters of that group.

Concerning part (3), let $\tau$ be a $\mb W_K$-automorphism 
of the absolute root datum of $(\mb G,\mb S)$. From Theorem \ref{thm:C} for 
$\big( \big(\mb G^\omega / Z(\mb G^\omega)_s \big) (K)$ we know that \eqref{eq:15.7} is
$\tau$-equivariant on the subset $\Irr \big( \big(\mb G^\omega / Z(\mb G^\omega)_s \big) 
(K) \big)_{\cusp,\unip}$. We also know, from \eqref{eq:2.2}, that the LLC for 
$X_\Wr (\mb G^\omega (K))$ is $\tau$-equivariant. In view of \eqref{eq:15.3} and
\eqref{eq:15.8}, this implies that \eqref{eq:15.7} is also $\tau$-equivariant. 

We note that the LLC for unipotent characters of tori is compatible with almost 
direct products, that follows readily from Proposition \ref{prop:15.1}. Consider $\mb G$
as the almost direct product of $Z(\mb G^\omega)_s$ and $\mb G_\der^\omega Z(\mb G^\omega)_a$,
where $Z(\mb G^\omega)_a$ denotes the maximal $K$-anisotropic subtorus of 
$Z(\mb G^\omega)^\circ$. Let $\pi \in \Irr (\mb G^\omega (K))_{\cusp,\unip}$, with
enhanced L-parameter $(\lambda_\pi, \rho_\pi)$. In terms of \eqref{eq:15.3} we write
$\pi = \pi_\der \otimes \chi$ and in terms of \eqref{eq:15.4} we write 
$\lambda_\pi = \lambda_{\pi_\der} \lambda_\chi$ and $\rho_\pi = \rho_{\pi_\der}$. Then
\[
\pi |_{Z(\mb G^\omega)_s (K)} = \chi |_{Z(\mb G^\omega)_s (K)} \quad \text{and} \quad
\pi |_{(\mb G_\der^\omega Z(\mb G^\omega)_a) (K)} = 
\pi_\der |_{(\mb G_\der^\omega Z(\mb G^\omega)_a) (K)} \otimes \chi |_{\mb G_\der^\omega (K)} .
\]
The naturality of the LLC for weakly unramified characters entails that the L-parameter of
$\chi |_{Z(\mb G^\omega)_s (K)}$ (resp. of $\chi |_{\mb G_\der^\omega Z(\mb G^\omega)_a) (K)}$)
is the image of $\lambda_\chi$ in $Z(\mb G)_s^\vee$ (resp. in
$Z (\mb G^\vee ) / Z(\mb G^\vee)^{\circ,\mb W_F}$). Lemmas \ref{lem:15.3}, 
\ref{lem:15.4} and \ref{lem:15.5} show that, to analyse the enhanced L-parameter of
$\pi_\der |_{(\mb G_\der^\omega Z(\mb G^\omega)_a) (K)}$, it suffices to consider the
restriction to $\mb G_\der^\omega (K)$. Then we are back in the case of semisimple groups,
and the constructions in the proof of Theorem \ref{thm:B}, see especially \eqref{eq:11.15}, 
were designed such that the L-parameter of $\pi_\der |_{\mb G_\der^\omega (K)}$ is the
image of $\lambda_{\pi_\der}$ in $\Phi_\nr (\mb G_\der^\omega (K))$. Similarly, the
constructions in Section \ref{sec:semisimple} and their wrap-up after \eqref{eq:11.20}
show that the enhancement for $\pi_\der |_{\mb G_\der^\omega (K)}$ contains 
$\rho_{\lambda_{\pi_\der}}$. 

The above says that \eqref{eq:15.7} is compatible with the almost direct product 
$\mb G = \big( \mb G_\der^\omega Z(\mb G^\omega)_a \big) Z(\mb G^\omega)_s$.
Now the same argument as in and directly after \eqref{eq:15.10}
and \eqref{eq:15.11} shows that Theorem \ref{thm:C}.(4) holds. $\qquad \Box$

\section{The Hiraga--Ichino--Ikeda conjecture}
\label{sec:HII}

We fix an additive character $\psi : K \to \CC^\times$ which is trivial
on the ring of integers $\mf o_K$ but nontrivial on any larger fractional ideal.
We endow $K$ with the Haar measure that gives $\mf o_K$ volume 1 and we normalize the 
Haar measure on $(\mb G^\omega / Z(\mb G^\omega)_s)(K) = \mb G^\omega (K) / 
Z(\mb G^\omega)_s (K)$ as in \cite{HII}. As $\psi$ has 
order 0, the Haar measure agrees with that in \cite[\S 5]{GrGa} and \cite[\S 4]{Gro}.
The formal degree of a square-integrable modulo centre representation of $\mb G^\omega (K)$ 
(e.g. a unitary supercuspidal representation) can be defined as in \cite[p. 285]{HII}. 

For a L-parameter $\lambda \in \Phi (\mb G^\omega (K))$ we write
\begin{equation}\label{eq:S}
S_\lambda^\sharp = \pi_0 \big( Z_{(\mb G / Z(\mb G)_s)^\vee} (\lambda) \big) . 
\end{equation}
When $\lambda$ is discrete, as it will be most of the time in this paper, we do not have
to take the group of components in \eqref{eq:S}, for the centralizer group is already finite.

Let $\pi \in \Irr (\mb G^\omega (K))_{\cusp,\unip}$ and let $(\lambda_\pi, \rho_\pi)$
be its enhanced L-parameter from Theorem \ref{thm:C}. 
It was conjectured in \cite[Conjecture 1.4]{HII} that
\begin{equation}\label{eq:HII}
\mr{fdeg}(\pi) = \dim (\rho_\pi) \, |S_{\lambda_\pi}^\sharp |^{-1} \, 
|\gamma (0,\mr{Ad} \circ \lambda_\pi, \psi)| .
\end{equation}
We will prove \eqref{eq:HII} with a series of lemmas, of increasing generality.
Thanks to Proposition \ref{prop:17.2} we do not have to worry about restriction of scalars.

\begin{lemma}\label{lem:16.3}
The equality \eqref{eq:HII} holds for $\Irr (\mb G^\omega (K))_{\cusp,\unip}$ 
when $\mb G$ is semisimple and adjoint. 
\end{lemma}
\begin{proof}
By \cite[Proposition 6.1.4]{GrGa} the normalization of the Haar measure on 
$\mb G^\omega (K)$ is respected by direct products of reductive 
$K$-groups. It follows that all the terms in \eqref{eq:HII} behave multiplicatively
with respect to direct products. Using that and Proposition \ref{prop:17.2}, we can follow
the strategy from the proof of Proposition \ref{prop:11.2} to reduce to the case of simple 
adjoint groups. For such groups \eqref{eq:HII} was proven in \cite[Theorem 4.11]{Opd2}.
\end{proof}

\begin{lemma}\label{lem:16.4}
The equality \eqref{eq:HII} holds for $\Irr (\mb G^\omega (K))_{\cusp,\unip}$ 
when $\mb G$ is semisimple.  
\end{lemma}
\begin{proof}
As in Sections \ref{sec:semisimple} and \ref{sec:proofss}, we consider the covering map
$\mb G \to \mb G_\ad$. We will show that it adjusts both sides of \eqref{eq:HII} by
the same factor.

From \eqref{eq:2.7} and \eqref{eq:2.8} we see that
\begin{equation}\label{eq:16.18}
\mr{fdeg}(\pi) = \frac{\dim (\sigma^N)}{\mr{vol}( N_{G^{F_\omega}}(\mh P^{F_\omega}))}
= \frac{\dim (\sigma)}{\mr{vol}(\mh P^{F_\omega}) \, |\Omega^{\theta,\mh P}|} .
\end{equation}
The volume of the Iwahori subgroup of $\mb G^\omega (K)$ with respect to our normalized
Haar measure was computed in \cite[(4.11)]{Gro}. In our notation, it says
\begin{equation}\label{eq:16.28}
\mr{vol}(\mh P^{F_\omega}) = |\overline{\mh P}^{F_\omega}| 
q_K^{- (\dim \overline{\mh P} + \dim \mb G) / 2} .
\end{equation}
With \cite[\S 5.1]{DeRe} one sees that this formula holds for all parahoric 
subgroups of $\mb G^\omega (K)$.

By \cite[Proposition 1.4.12.c]{MaGe} the cardinality of the group of $\mf F$-points of a 
connected reductive group does not change when we replace it by an isogenous $\mf F$-group. 
In particular $|\overline{\mh P_\ad}^{F_\omega}| = |\overline{\mh P}^{F_\omega}|$. 
From \cite[\S 3]{Lus-Chevalley} we know that $\sigma$ can also be regarded as a cuspidal
unipotent representation $\sigma_\ad$ of $\overline{\mh P_\ad}^{F_\omega}$, and then of
course $\dim (\sigma_\ad) = \dim (\sigma)$.

Choose a $\pi_\ad \in \Irr (\mb G_\ad^\omega (K))_{[\mh P_\ad,\sigma_\ad]}$ whose pullback
to $\mb G^\omega (K)$ contains $\pi$. The above allows us to simplify
\begin{equation}\label{eq:16.19}
\frac{\mr{fdeg}(\pi)}{\mr{fdeg}(\pi_\ad)} = 
\frac{\dim (\sigma) q_K^{- (\dim \overline{\mh P_\ad} + \dim \mb G_\ad) / 2} 
|\overline{\mh P_\ad}^{F_\omega}| \, |\Omega_\ad^{\theta,\mh P}|}{\dim (\sigma_\ad) 
q_K^{-( \dim \overline{\mh P} + \dim \mb G) / 2} |\overline{\mh P}^{F_\omega}| \, 
|\Omega^{\theta,\mh P}|} = \frac{|\Omega_\ad^{\theta,\mh P}|}{|\Omega^{\theta,\mh P}|} .
\end{equation}
Recall from Section \ref{sec:proofss} that the image of $\lambda_{\pi_\ad}$ under
${\mb G_\ad}^\vee \to \mb G^\vee$ is $\lambda_\pi$. Adjoint $\gamma$-factors are defined via
the action on the Lie algebra of $\mb G^\vee$, so
\begin{equation}\label{eq:16.20}
\gamma (s,\mr{Ad} \circ \lambda_\pi,\psi) = \gamma (s,\mr{Ad} \circ \lambda_{\pi_\ad},\psi) .
\end{equation}
From \eqref{eq:11.16} and Lemmas \ref{lem:11.4} and \ref{lem:11.6} we see that 
\begin{equation}\label{eq:16.21}
\dim (\rho_\pi) / \dim (\rho_{\pi_\ad}) = 
[\mc A_{\lambda_\pi} : (\mc A_{\lambda_\pi} )_{\rho_{\pi_\ad}}] . 
\end{equation}
With \eqref{eq:11.13} we can express \eqref{eq:16.21} as
\begin{equation}\label{eq:16.22}
[\Omega_\ad^{\theta,\mh P} : \Omega^{\theta,\mh P} N_{\lambda_{\pi_\ad}}] 
= |\Omega_\ad^{\theta,\mh P}| \, |\Omega^{\theta,\mh P} \cap N_{\lambda_{\pi_\ad}}| \, 
|\Omega^{\theta,\mh P}|^{-1} |N_{\lambda_{\pi_\ad}}|^{-1} .
\end{equation}
Write $S_{\lambda_\pi} = \pi_0 \big( Z_{{\mb G^\vee}_\SC}(\lambda) \big)$, as in the proof of 
Lemma \ref{lem:11.4}. Then $S_{\lambda_{\pi_\ad}} = S_{\lambda_{\pi_\ad}}^\sharp$ and
\begin{equation}\label{eq:16.23}
| S_{\lambda_\pi} |\, |S_{\lambda_\pi}^\sharp|^{-1} =
|Z ({\mb G^\vee}_\SC )^\theta | \, |Z(\mb G^\vee)^\theta |^{-1} .
\end{equation}
Like for any finite abelian group with a $\Z$-action, there are as many invariants as 
co-invariants. Also taking \eqref{eq:2.1} account, \eqref{eq:16.23} equals
\begin{equation}\label{eq:16.24}
|Z ({\mb G^\vee}_\SC )_\theta | \, |Z(\mb G^\vee)_\theta |^{-1} =
| (\Omega_\ad^\theta)^* |\, |(\Omega^\theta)^*|^{-1} = [\Omega_\ad^\theta : \Omega^\theta] .
\end{equation}
With \eqref{eq:11.5} we obtain
\begin{equation}\label{eq:16.25}
\begin{aligned}
| S_{\lambda_\pi}^\sharp|\, |S_{\lambda_{\pi_\ad}}^\sharp|^{-1} & =
| S_{\lambda_\pi}^\sharp | \,| S_{\lambda_\pi} |^{-1} [ S_{\lambda_\pi} : S_{\lambda_{\pi_\ad}}] \\
& = [\Omega_\ad^\theta : \Omega^\theta]^{-1} [\Omega_\ad^\theta : \Omega^\theta N_{\lambda_{\pi_\ad}}] \\
& = |\Omega^\theta| \, |\Omega^\theta N_{\lambda_{\pi_\ad}} |^{-1} = 
|\Omega^\theta \cap N_{\lambda_{\pi_\ad}} | \, |N_{\lambda_{\pi_\ad}} |^{-1} .
\end{aligned}
\end{equation}
Recall from \eqref{eq:11.3} that $N_{\lambda_{\pi_\ad}} \subset \Omega_\ad^{\theta,\mh P}$,
which implies $\Omega^\theta \cap N_{\lambda_{\pi_\ad}} = 
\Omega^{\theta,\mh P} \cap N_{\lambda_{\pi_\ad}}$. From \eqref{eq:16.20}, \eqref{eq:16.22}, 
\eqref{eq:16.23} and \eqref{eq:16.25} we deduce
\begin{equation}\label{eq:16.26}
\frac{ \dim (\rho_\pi) \, |S_{\lambda_{\pi_\ad}}^\sharp | \, 
|\gamma (0,\mr{Ad} \circ \lambda_\pi, \psi)| }{\dim (\rho_{\pi_\ad}) \, |S_{\lambda_\pi}^\sharp | \, 
|\gamma (0,\mr{Ad} \circ \lambda_{\pi_\ad}, \psi)| } =
\frac{ |\Omega_\ad^{\theta,\mh P_\ad}| \, |\Omega^{\theta,\mh P} \cap N_{\lambda_{\pi_\ad}}| \,
|N_{\lambda_{\pi_\ad}} |}{|\Omega^{\theta,\mh P}| \, |N_{\lambda_{\pi_\ad}}| \, 
|\Omega^\theta \cap N_{\lambda_{\pi_\ad}}| } =
\frac{  |\Omega_\ad^{\theta,\mh P_\ad}| }{|\Omega^{\theta,\mh P}|} .
\end{equation}
Now \eqref{eq:16.19}, \eqref{eq:16.26} and Lemma \ref{lem:16.3} imply \eqref{eq:HII}
for all $\pi \in \Irr (\mb G^\omega (K))_{\cusp,\unip}$.
\end{proof} 

\begin{lemma}\label{lem:16.5}
The equality \eqref{eq:HII} holds for $\Irr (\mb G^\omega (K))_{\cusp,\unip}$ 
when $\mb G$ is reductive and $Z(\mb G)^\circ$ is $K$-anisotropic.   
\end{lemma}
\begin{proof}
Recall that in Section \ref{sec:mainred} we established Theorem \ref{thm:C} for 
$\mb G$ via restriction to $\mb G_\der$. By \cite[\S 3]{Lus-Chevalley}, the cuspidal
unipotent representations of $\mh P^{F_\omega}$ can be identified with those of
$\mh P_\der^{F_\omega}$. (Recall that by definitions all these representations are
inflated from finite reductive groups.) We denote $\sigma$ as 
$\mh P_\der^{F_\omega}$-representation by $\sigma_\der$. In Lemma \ref{lem:15.2} 
we checked that $\Omega^{\theta,\mh P} = \Omega_\der^{\theta,\mh P}$.  

Let $\overline{\mh P_a}$ be the image of $Z(\mb G)_a (K_\nr) = Z(\mb G)^\circ (K_\nr)$
in $\overline{\mh P}$, an $\mf F$-anisotropic torus of the same dimension as $Z(\mb G)_a$. 
Then $\overline{\mh P_a} \times \overline{\mh P_\der}$ is isogenous to $\overline{\mh P}$, 
and \cite[Proposition 1.4.12.c]{MaGe} tells us that
\[
|\overline{\mh P_a}^{F_\omega}| \, |\overline{\mh P_\der}^{F_\omega}| =
|\overline{\mh P}^{F_\omega}| .
\]
Since \eqref{eq:16.18} and \eqref{eq:16.28} are also valid for $\mb G^\omega (K)$, we can 
compare the formal degrees of $\pi$ and its restriction $\pi_\der$ to $\mb G_\der^\omega (K)$:
\begin{equation}\label{eq:16.27}
\!\! \frac{\mr{fdeg}(\pi)}{\mr{fdeg}(\pi_\der)} = \frac{\dim (\sigma) 
q_K^{- (\dim \overline{\mh P_\der} + \dim \mb G_\der) / 2} |\overline{\mh P_\der}^{F_\omega}| \, 
|\Omega_\der^{\theta,\mh P}|}{\dim (\sigma_\der) q_K^{- (\dim \overline{\mh P} + \dim \mb G) / 2} 
|\overline{\mh P}^{F_\omega}| \, |\Omega^{\theta,\mh P}|} = 
\frac{q_K^{(\dim \overline{\mh P_a} + \dim Z(\mb G)_a) / 2}}{|\overline{\mh P_a}^{F_\omega}|} . 
\hspace{-4mm}
\end{equation}
Recall from Lemma \ref{lem:15.4} and the proof of Proposition \ref{prop:15.8} that
$\lambda_{\pi_\der}$ is the canonical image of $\lambda_\pi$ under $\mb G^\vee \to 
{\mb G_\der}^\vee$. Using the decomposition
\[
\mf g^\vee = \mf g_\der^\vee \oplus Z(\mf g^\vee) = 
\mr{Lie}({\mb G_\der}^\vee) \oplus \mr{Lie}(Z(\mb G)_a^\vee) ,
\]
we can write 
\[
\mr{Ad}_{\mb G^\vee} \circ \lambda_\pi = \mr{Ad}_{{\mb G_\der}^\vee} \circ \lambda_{\pi_\der} 
\oplus \text{(action of } \mb W_K / \mb I_K \text{ on } Z(\mf g^\vee) ) .
\]
The action of $\mb W_K$ on $Z(\mf g^\vee)$ can be considered as the composition of 
the adjoint representation of ${}^L Z(\mb G)_a$ and $\mr{id}_{\mb W_K}$ (as 
L-parameter for $Z (\mb G)_a (K)$). From the definition \eqref{eq:16.1} we see that 
\begin{equation}\label{eq:16.29}
\gamma (s, \mr{Ad}_{\mb G^\vee} \circ \lambda_\pi, \psi) = 
\gamma (s, \mr{Ad}_{{\mb G_\der}^\vee} \circ \lambda_{\pi_\der}, \psi) 
\gamma (s, \mr{Ad}_{Z(\mb G)_a^\vee} \circ \mr{id}_{\mb W_K}, \psi) .
\end{equation}
Recall from Lemma \ref{lem:15.5} and the proof of Proposition \ref{prop:15.8} that $\rho_\pi$
can be identified with $\rho_{\pi_\der}$. Since $Z(\mb G)_a (K)$ is a torus, L-parameters
for that group do not need enhancements. Formally, we can say that the enhancement of
$\mr{id}_{\mb W_K}$ is the trivial one-dimensional representation of 
$\mc A_{\mr{id}_{\mb W_K}} = \pi_0 (Z(\mb G)_a^\vee) = 1$.

As in the proof of Lemma \ref{lem:15.5} we see that 
\[
S^\sharp_{\lambda_{\pi_\der}} = \big\{ g Z(\mb G^\vee)^\circ \in {\mb G_\der}^\vee :
g \lambda_\pi g^{-1} = z \lambda_\pi \text{ for some } z \in Z(\mb G^\vee)^\circ \big\} .
\]
By \eqref{eq:15.2} this equals
\begin{multline*}
\big\{ g Z(\mb G^\vee)^\circ \in {\mb G_\der}^\vee : g' \lambda_\pi g'^{-1} = \lambda_\pi 
\text{ for some } g' \in g Z(\mb G^\vee)^\circ \big\} \\
= Z_{\mb G^\vee}(\lambda_\pi) \big/ Z_{Z (\mb G^\vee)^\circ}(\lambda_\pi) =
S^\sharp_{\lambda_\pi} \big/ Z (\mb G^\vee)^{\circ,\theta} .
\end{multline*}
Now we compare the right hand sides of \eqref{eq:HII} for $\pi$ and $\pi_\der$:
\begin{equation}\label{eq:16.30}
\frac{ \dim (\rho_\pi) \, |S_{\lambda_{\pi_\der}}^\sharp | \, 
|\gamma (0,\mr{Ad}_{\mb G^\vee} \circ \lambda_\pi, \psi)| }{\dim (\rho_{\pi_\der}) \, 
|S_{\lambda_\pi}^\sharp | \, |\gamma (0,\mr{Ad}_{{\mb G_\der}^\vee} \circ \lambda_{\pi_\der}, \psi)| } = 
\frac{ |\gamma (0, \mr{Ad}_{Z(\mb G)_a^\vee} \circ \mr{id}_{\mb W_K}, \psi)| 
}{| Z(\mb G^\vee )^{\circ,\theta} |} .
\end{equation}
It was shown in \cite[Lemma 3.5]{HII} that 
\[
|\gamma (0, \mr{Ad}_{Z(\mb G)_a^\vee} \circ \mr{id}_{\mb W_K}, \psi)| = q_K^{\dim (Z(\mb G)_a)} 
| (Z(\mb G^\vee )^{\circ,\theta} | \, |\overline{\mh P_a}^{F_\omega}|^{-1} . 
\]
Then \eqref{eq:16.30} becomes 
\[
q_K^{\dim Z(\mb G)_a} |\overline{\mh P_a}^{F_\omega}|^{-1} =
q_K^{(\dim \overline{\mh P_a} + \dim Z(\mb G)_a) / 2}  |\overline{\mh P_a}^{F_\omega}|^{-1} ,
\]
which equals \eqref{eq:16.27}. In combination with Lemma \ref{lem:16.4} for $\mb G_\der$ 
that gives \eqref{eq:HII} for $\pi \in \Irr (\mb G^\omega (K))_{\cusp,\unip}$.
\end{proof}

We are ready to extend \eqref{eq:HII} to $K_\nr$-split reductive $K$-groups.\\

\emph{Proof of Theorem \ref{thm:HII}}\\
Let $\pi \in \Irr (\mb G^\omega (K))_{\cusp,\unip}$ be unitary. As observed in
\eqref{eq:15.13}, there exists an unramified character $\chi \in X_\nr (\mb G^\omega (K))$
such that $\pi' := \pi \otimes \chi^{-1}$ is trivial on $Z(\mb G^\omega)_s (K)$.
By definition \cite[p. 285]{HII}
\begin{equation}\label{eq:16.31}
\mr{fdeg} (\pi) = \mr{fdeg} (\pi') ,
\end{equation}
where $\pi'$ is regarded as a representation of $(\mb G^\omega / Z(\mb G^\omega)_s)(K)$.
By construction \eqref{eq:15.7} 
\[
(\lambda_\pi, \rho_\pi) = (\lambda_{\pi'} \chi^\vee, \rho_{\pi'}) ,
\]
where $\chi^\vee \in Z(\mb G^\vee )^\circ_\theta$ is the image of $\chi$ under \eqref{eq:2.2}.
Recall that our adjoint representation of ${}^L \mb G$ does not act on 
Lie$(\mb G^\vee)$ but on Lie$((\mb G / Z(\mb G)_s)^\vee )$.
From \eqref{eq:16.1} and the definitions of L-functions and $\epsilon$-factors in 
\cite{Tat} we see that 
\[
\gamma (s,\mr{Ad} \circ \lambda_\pi,\psi) = 
\gamma (s,\mr{Ad} \circ \lambda_{\pi'},\psi) .
\]
The group $S^\sharp_{\lambda_\pi}$ is already defined via $(\mb G / Z(\mb G)_s)^\vee$, 
so it equals $S^\sharp_{\lambda_{\pi'}}$. A part of the construction of \eqref{eq:15.7}
is that $\rho_\pi$ can be identified with $\rho_{\pi'}$. Thus the entire expression
\[
\dim (\rho_\pi) \, |S_{\lambda_\pi}^\sharp |^{-1} \, 
\gamma (s,\mr{Ad} \circ \lambda_\pi, \psi) 
\]
remains unchanged when we replace $\pi$ by $\pi'$ and $\mb G^\omega$ by $\mb G^\omega / 
Z(\mb G^\omega)_s$. In view of \eqref{eq:16.31} this means that \eqref{eq:HII} for
$\mb G^\omega (K)$ is equivalent to \eqref{eq:HII} for $(\mb G^\omega / Z(\mb G^\omega)_s)(K)$.
The group $\mb G^\omega / Z(\mb G^\omega)_s$ has $K$-anisotropic connected centre, so
for $\Irr \big( (\mb G^\omega / Z(\mb G^\omega)_s)(K) \big)_{\cusp,\unip}$ we have already
established \eqref{eq:HII} in Lemma \ref{lem:16.5}. $\qquad \Box$

\appendix

\section{Restriction of scalars and adjoint $\gamma$-factors}
\label{sec:restriction}

Let $L/K$ be a finite separable extension of non-archimedean local fields.
In this appendix we will first investigate to what extent local factors are inductive
for $\mb W_L \subset \mb W_K$. This question is well-known for Weil group representations,
but more subtle for representations of Weil--Deligne groups. We could not find in the 
literature, although it probably is known to some experts. After establishing the inductivity
result for general Weil--Deligne representations, we will check that it
applies to the Langlands parameters obtained from restriction of scalars of reductive 
groups. Then we will show that the HII conjectures are stable under Weil restriction.

We follow the conventions of \cite{Tat} for local factors. Let $\psi : K \to \CC^\times$ 
be a nontrivial additive character . We endow $K$ with the Haar measure that gives the ring 
of integers $\mf o_K$ volume 1, and similarly for $L$. For $s \in \CC$ let $\omega_s :
\mb W_K \to \CC^\times$ be the character $w \mapsto \| w \|^s$. For any (finite
dimensional) $\mb W_K$-representation $V$, by definition
\[
L(s, V) = L(\omega_s \otimes V) \quad \text{and} \quad 
\epsilon (s,V,\psi) = \epsilon (\omega_s \otimes V , \psi).
\]
We endow objects associated to $L$ with a subscript $L$, to distinguish 
them from objects for $K$ (without subscript). The restriction of $\omega_s$ from
$\mb W_K$ to $\mb W_L$ equals $\omega_s$ (as defined purely in terms of $L$), so for
any $\mb W_L$-representation $V_L$:
\begin{equation}\label{eq:17.5}
\mr{ind}_{\mb W_L}^{\mb W_K} \big( \omega_s \otimes V_L \big) =  
\omega_s \otimes \mr{ind}_{\mb W_L}^{\mb W_K} V_L .
\end{equation}
As concerns representations of the Weil--Deligne group $\mb W_K \times SL_2 (\CC)$, we 
only consider those which are admissible, that is, finite dimensional and the image of
$\mb W_K$ consists of semisimple automorphisms. In view of \cite[\S 4.1.6]{Tat} that
is hardly a restriction for local factors. It has the advantage that the category of
such representations is semisimple, so all the local factors are additive and make sense
for virtual admissible representations of $\mb W_K \times SL_2 (\CC)$. (The definitions 
of these local factors will be recalled in the course of the next proofs.)

\begin{lemma}\label{lem:A.1}
L-functions of Weil--Deligne representations are inductive. That is, for any
admissible virtual representation $(\tau_L,V_L)$ of $\mb W_L \times SL_2 (\CC)$:
\[
L(s,\mr{ind}_{\mb W_L \times SL_2 (\CC)}^{\mb W_K \times SL_2 (\CC)} \tau_L) =
L(s,\tau_L) \quad \text{for all } s \in \CC.
\] 
\end{lemma}
\begin{proof}
Since these local factors are additive, we may assume that $(\rho_L,V_L)$ is an actual 
representation.
We write $N_L = \mr{d}(\tau_L |_{SL_2 (\CC)}) \left( \begin{smallmatrix}0 & 1 \\ 
0 & 0 \end{smallmatrix} \right) \in \mr{End}_\CC (V_L)$ and
\[
V = \mr{ind}_{\mb W_L \times SL_2 (\CC)}^{\mb W_K \times SL_2 (\CC)} V_L ,\quad
\tau =\mr{ind}_{\mb W_L \times SL_2 (\CC)}^{\mb W_K \times SL_2 (\CC)} \tau_L ,\quad
N = \mr{d}(\tau|_{SL_2 (\CC)}) \left( \begin{smallmatrix}0 & 1 \\ 0 & 0 \end{smallmatrix} \right)
\in \mr{End}_\CC (V) .
\]
Then the kernel of $N$ in $V$ is stable under $\tau (\mb W_K)$ 
and the kernel of $N_L$ is a $\tau_L (\mb W_L)$-stable subspace of $V_L$. One checks directly 
that $\ker N = \mr{ind}_{\mb W_{L}}^{\mb W_K} (\ker N_L)$ and 
\begin{equation}\label{eq:16.7}
\tau |_{\ker N} = \mr{ind}_{\mb W_{L}}^{\mb W_K} \big( \tau_L |_{\ker N_L} \big) .
\end{equation}
By definition \cite[\S 4.1.6]{Tat}
\begin{equation}\label{eq:16.8}
L(s,\tau) = L(\omega_s \otimes \ker N^{\mb I_K}) = 
\det \big( 1 - q_K^{-s} \tau (\Frob) |_{\ker N^{\mb I_K}} \big)^{-1} .
\end{equation}
The function L, from Weil group representations 
to $\CC^\times$, is additive and inductive \cite[\S 3.3.2]{Tat}. 
The latter means that 
\begin{equation}\label{eq:16.9}
L \big( \omega_s \otimes \ker N_L^{\mb I_L} \big) = 
L \big( \mr{ind}_{\mb W_{L}}^{\mb W_K} (\omega_s \otimes \ker N_L^{\mb I_L}) \big) =
L \big( \omega_s \otimes \mr{ind}_{\mb W_{L}}^{\mb W_K} \ker N_L^{\mb I_L} \big) .
\end{equation}
Let $E$ be the maximal unramified subextension of $L/K$, and define $N_E, \tau_E$ etcetera
in the same way as for $K$. Since $\mb I_E = \mb I_K$, 
\begin{equation}\label{eq:17.1}
\mr{ind}_{\mb W_E}^{\mb W_K} \big( \ker N_E^{\mb I_E} \big) = 
\big( \mr{ind}_{\mb W_E}^{\mb W_K} \ker N_E \big)^{\mb I_K} = \ker N^{\mb I_K} .  
\end{equation}
From \eqref{eq:16.7}, \eqref{eq:16.8}, \eqref{eq:16.9} and \eqref{eq:17.1} we deduce that
\[
L(s,\tau) = L(\omega_s \otimes \ker N^{\mb I_K}) =
L \big( \omega_s \otimes \mr{ind}_{\mb W_E}^{\mb W_K} 
(\ker N_E^{\mb I_E}) \big) = L(s,\tau_E) .
\] 
The extension $L/E$ is totally ramified, so $q_L = q_E$. We can take $\Frob_E$ in $\mb W_L$, 
then it is also a Frobenius element of $\mb W_L$. From
\[
\omega_s \otimes \ker N_E^{\mb I_E} = \big( \omega_s \otimes 
\mr{ind}_{\mb W_{L}}^{\mb W_E} (\ker N_L^{\mb I_L}) \big)^{\mb I_E} = 
\big( \mr{ind}_{\mb W_{L}}^{\mb W_E} (\omega_s \otimes \ker N_L^{\mb I_L}) \big)^{\mb I_E} 
= \omega_s \otimes \ker N_L^{\mb I_L} 
\]
and \eqref{eq:16.8} we obtain $L(s,\tau_E) = L(s,\tau_L)$. 
\end{proof}

Let $e_{L/K}$ and $f_{L/K}$ denote the ramification index and the residue degree of $L/K$,
respectively. We endow $L$ with the discrete valuation whose image is $\ZZ \cup \{ \infty \}$. 
The restriction of this valuation to $K$ equals $e_{L/K}$ times the valuation of $K$.

Recall that the order of $\psi$ is the largest $n \in \Z$ such that $\psi (k) = 1$ for
all $k \in K$ of valuation $\geq -n$. Let $\psi_L : L \to \CC^\times$ be the composition
of $\psi$ with the trace map for $L/K$. We recall from \cite[Proposition III.3.7]{Ser}
that the order of $\psi_L$ is determined by the order of $\psi$ and the different $\mc D_{L/K}$
of $L/K$. For $l \in L^\times$ we define another additive character $\psi_{L,l}$ of $L$ by 
\[
\psi_{L,l}(y) = \psi_L (ly) = \psi (\mathrm{tr}_{L/K}(ly)) .
\]

\begin{thm}\label{thm:A.2}
Let $(\tau_L,V_L)$ be an admissible virtual representation of $\mb W_L \times SL_2 (\CC)$.
\begin{enumerate}
\item For every $l \in L^\times$ and every $s \in \CC$:
\[
\frac{\gamma (s,\mr{ind}_{\mb W_L \times SL_2 (\CC)}^{\mb W_K \times SL_2 (\CC)} \tau_L, \psi)}{
\gamma (s,\tau_L,\psi_{L,l})} = \frac{\epsilon (s,\mr{ind}_{\mb W_L \times 
SL_2 (\CC)}^{\mb W_K \times SL_2 (\CC)} \tau_L, \psi)}{\epsilon (s,\tau_L,\psi_{L,l})} .
\]
\item For all $s \in \CC$:
\[
\frac{\epsilon (s,\mr{ind}_{\mb W_L \times SL_2 (\CC)}^{\mb W_K \times SL_2 (\CC)} \tau_L, \psi)}{
\epsilon (s,\tau_L,\psi_L)} = \left( \frac{\epsilon (\CC [\mb W_K / \mb W_{L}], 
\psi)}{\epsilon (\mr{triv}_{\mb W_L}, \psi_L) } \right)^{\dim (V_L)} .
\]
In particular, $\epsilon$-factors of Weil--Deligne representations are inductive for 
virtual representations of dimension zero.
\item When $L/K$ is unramified:
\[
\epsilon (s,\mr{ind}_{\mb W_L \times SL_2 (\CC)}^{\mb W_K \times SL_2 (\CC)} \tau_L, \psi)
\epsilon (s,\tau_L,\psi_L)^{-1} = (-1)^{([L:K] - 1) \mr{ord}(\psi) \dim (V_L)} .
\]
\item Suppose that $(\omega_s \otimes \tau_L,V_L)$ is self-dual or unitary. For any $l \in L^\times$:
\[
\left| \frac{\epsilon (s,\mr{ind}_{\mb W_L \times SL_2 (\CC)}^{\mb W_K \times SL_2 (\CC)} \tau_L, \psi)}{
\epsilon (s,\tau_L,\psi_{L,l})} \right| = \left| \frac{\epsilon (\CC [\mb W_K / \mb W_{L}], 
\psi)}{\epsilon (\mr{triv}_{\mb W_L}, \psi_{L,l}) } \right|^{\dim (V_L)} .
\]
\item We denote the Artin conductor of a $\mb W_K$-representation $V$ by $\mb a (V)$. When 
$\mr{ord}(\psi) = \mr{ord}(\psi_{L,l}) = 0$ and $(\omega_s \otimes \tau_L,V_L)$ is self-dual or unitary:
\[
\left| \frac{\epsilon (s,\mr{ind}_{\mb W_L \times SL_2 (\CC)}^{\mb W_K \times SL_2 (\CC)} \tau_L, \psi)}{
\epsilon (s,\tau_L,\psi_{L,l})} \right| = q_K^{\mb a (\CC [\mb W_K / \mb W_L]) \dim (V_L) / 2} =
[\mf o_L : \mc D_{L/K}]^{\dim (V_L) / 2} .
\]
\end{enumerate}
\end{thm}
\begin{proof}
We use the conventions and notations from the proof of Lemma \ref{lem:A.1}. \\
(1) Recall from \cite[\S 3.2]{GrRe} that for any admissible representation $V$ of 
$\mb W_K \times SL_2 (\CC)$:
\begin{equation}\label{eq:16.1}
\gamma (s,V,\psi) = \epsilon (s,V, \psi) L(1-s,V^*) L(s,V)^{-1}.
\end{equation}
With this definition (1) is an obvious consequence of Lemma \ref{lem:A.1}.\\
(2) We write
\[
\mr{coker} \, N = V / \ker N \text{ and coker } N_L = V_L / \ker N_L. 
\]
These are representations of $\mb W_K$ and $\mb W_L$, respectively. 
From \eqref{eq:17.5} and \eqref{eq:16.7} we see that 
\begin{equation}\label{eq:16.12}
\mr{ind}_{\mb W_L}^{\mb W_K} (\omega_s \otimes \mr{coker}\, N_L) = 
\omega_s \otimes \mr{coker}\, N . 
\end{equation}
Since $\mb I_K$ is compact, 
\[
(\mr{coker}\, N)^{\mb I_K} = V^{\mb I_K} \big/ \ker N^{\mb I_K} . 
\]
For a $\mb W_K \times SL_2 (\CC)$-representation $(\tau,V_\tau)$ we define a 
$\mb W_K$-representation $(\tau_0,V_\tau)$ by 
\[
\tau_0 (w) = \tau \Big( w, \matje{\| w \|^{1/2}}{0}{0}{\| w \|^{-1/2}} \Big) .
\]
By definition \cite[\S 4.1.6]{Tat}
\begin{equation}\label{eq:16.10}
\epsilon (s,\tau,\psi) = 
\epsilon (\omega_s \otimes \tau_0,\psi)  
\det \big(- \omega_s \otimes \tau (\Frob) |_{\mr{coker}\, N^{\mb I_K}} \big) ,
\end{equation}
As $L/E$ is totally ramified, $q_L = q_E$ and $\Frob_L = \Frob_E$. 
From \eqref{eq:17.5} we see that 
\[
\omega_s \otimes \mr{coker} N^{\mb I_E} = 
\big( \omega_s \otimes \mr{ind}_{\mb W_{L}}^{\mb W_E} \mr{coker} N_L \big)^{\mb I_E} 
= \omega_s \otimes  \mr{coker} N_L^{\mb I_L}.
\]
Hence the rightmost term in \eqref{eq:16.10} is the same for $\tau_E$ and for $\tau_L$.

As in \eqref{eq:17.1}, we find that
\[
\omega_s \otimes \mr{coker} N^{\mb I_K} =
\big( \omega_s \otimes \mr{ind}_{\mb W_E}^{\mb W_K} \mr{coker} N_E \big)^{\mb I_K} = 
\mr{ind}_{\mb W_E}^{\mb W_K} \big( \omega_s \otimes \mr{coker} N_E^{\mb I_E} \big) .
\]
With elementary linear algebra one checks that 
\begin{equation}\label{eq:16.13}
\det \big(- \omega_s \otimes \tau (\Frob) |_{\mr{coker}\, N^{\mb I_K}} \big) =
\det \big(- \omega_s \otimes \tau_E (\Frob^{[E:K]}) |_{\mr{coker}\, N_E^{\mb I_E}} \big) .
\end{equation}
Since $\Frob^{[E:K]}$ is a Frobenius element of $\mb W_E$, we see that here the rightmost 
term in \eqref{eq:16.10} is the same for $\tau$ and for $\tau_E$, which we already
know is the same as for $\tau_L$.  

By \cite[\S 3.4]{Tat}, $\epsilon (V,\psi)$ is additive and inductive in degree 0 
(i.e. for virtual $\mb W_K$-representations $V$ of dimension 0). Consider the virtual 
$\mb W_L$-representation 
\begin{equation}\label{eq:17.14}
V_L^\circ := ( \omega_s \otimes \tau_{L,0}, V_L ) \,-\, 
\dim ( V_L ) (\mr{triv}, \CC) .
\end{equation}
The inductivity in degree 0 says
\begin{equation}\label{eq:16.14}
\begin{aligned}
& \epsilon (\omega_s \otimes \tau_{L,0},\psi_L) \epsilon (\mr{triv}_{\mb W_L}, 
\psi_L)^{-\dim (V_L)} = \epsilon (V_L^\circ,\psi_L) = \\ 
& \epsilon (\mr{ind}_{\mb W_L}^{\mb W_K} V_L^\circ,\psi) = 
\epsilon (\omega_s \otimes \tau_0 ,\psi) 
\epsilon (\CC [\mb W_K / \mb W_{L}], \psi)^{-\dim (V_L )} .
\end{aligned}
\end{equation}
We rewrite \eqref{eq:16.14} as
\begin{equation}\label{eq:17.3}
\frac{\epsilon (\omega_s \otimes \tau_0, \psi)}{\epsilon 
(\omega_s \otimes \tau_{L,0}, \psi_L)} = 
\left( \frac{\epsilon (\CC [\mb W_K / \mb W_{L}], 
\psi)}{\epsilon (\mr{triv}_{\mb W_L}, 
\psi_L)} \right)^{\dim (V_L )}. 
\end{equation}
In view of \eqref{eq:16.10} and the above analysis of the rightmost term in that formula,
\eqref{eq:17.3} equals $\epsilon (s,\tau ,\psi) \epsilon (s,\tau_L, \psi_L)^{-1}$, as 
asserted. In particular we see that
\[
\epsilon (s,\tau ,\psi) = \epsilon (s,\tau_L, \psi_L)  \text{ if } \dim (V_L) = 0 ,
\]
proving the inductivity in degree zero.\\
(3) When $L/K$ is unramified, we can simplify \eqref{eq:17.3}. In that case \cite[Proposition III.3.7
and Theorem III.5.1]{Ser} show that ord$(\psi_L) = \mr{ord}(\psi)$. Furthermore
the $\mb W_K$-representation $\CC [\mb W_K / \mb W_{L}]$ is a direct sum of unramified characters, 
which makes it easy to calculate its $\epsilon$-factor. Pick $a \in K^\times$ with valuation
$-\mr{ord}(\psi)$, so that the additive character $\psi_a : k \mapsto \psi (a k)$ has order zero.
From \cite[(3.4.4) and \S 3.2.6]{Tat} we obtain
\begin{equation}\label{eq:17.4}
1 = \epsilon (\CC [\mb W_K / \mb W_{L}], \psi_a) = 
|a|_K^{-[L:K]} \det (a, \CC[\mb W_K / \mb W_L]) \epsilon (\CC [\mb W_K / \mb W_{L}], \psi) .
\end{equation}
By the assumptions on $a$ and $L/K$,
\[
|a|_K^{[L:K]} = q_K^{\mr{ord}(\psi) [L:K]} = q_L^{\mr{ord}(\psi)} = q_L^{\mr{ord}(\psi_L)} .
\]
The group $\mb W_K / \mb W_L$ can be identified with 
$\langle \Frob \rangle / \langle \Frob^{[L:K]} \rangle$.
Elements of valuation one act on that through a cycle of length $[L:K]$, and $a$ acts by the
$\mr{ord}(\psi)$-th power of that cycle. Consequently
\[
\det (a, \CC[\mb W_K / \mb W_L]) = (-1)^{([L:K] - 1) \mr{ord}(\psi) } .
\]
On the other hand, by \cite[(3.2.6.1)]{Tat} $\epsilon (\mr{triv}_{\mb W_L}, \psi_L) = 
q_L^{\mr{ord}(\psi_L)}$. Now \eqref{eq:17.4} becomes
\[
\epsilon (\CC [\mb W_K / \mb W_{L}], \psi) = 
q_L^{\mr{ord}(\psi)} (-1)^{([L:K] - 1) \mr{ord}(\psi) } =
\epsilon (\mr{triv}_{\mb W_L}, \psi_L) (-1)^{([L:K] - 1) \mr{ord}(\psi) } .
\]
Combine this with part (2).\\
(4) By \cite[(3.4.4)]{Tat} applied to \eqref{eq:17.14}
\begin{equation}\label{eq:17.13}
\epsilon (V_L^\circ ,\psi_{L,l}) \epsilon (V_L^\circ,\psi_L)^{-1} = \det (l,V_L^\circ) = 
\det (\omega_s \otimes \tau_{L,0}(l), V_L) .
\end{equation}
By the assumed self-duality or unitarity of $\omega_s \otimes \tau$, 
$|\det (\omega_s \otimes \tau_{L,0}(l), V_L)| = 1$. Then \eqref{eq:17.13} says 
\[
\left| \frac{\epsilon (s,\tau_L,\psi_{L,l})}{\epsilon 
(\mr{triv}_{\mb W_L},\psi_{L,l})^{\dim (V_L)}} \right| = \left| \frac{\epsilon 
(s,\tau_L,\psi_L)}{\epsilon (\mr{triv}_{\mb W_L},\psi_L)^{\dim (V_L)}} \right| .
\]
Combine that with part (2).\\
(5) We work out the right hand side of part (4). Since ord$(\psi_{L,l}) = 0$ 
and $\mf o_L$ has volume 1, $\epsilon (\mr{triv}_{\mb W_L}, \psi_{L,l}) = 1$
\cite[\S 3.2.6]{Tat}. The Haar measure on $K$ which gives $\mf o_K$ volume 1 is
self-dual with respect to the additive character $\psi$ of order 0. Hence
\cite[(3.4.7)]{Tat} applies. It tells us that 
\begin{equation}\label{eq:17.16}
| \epsilon (\CC [\mb W_K / \mb W_L], \psi) |^2 = q_K^{\mb a (\CC [\mb W_K / \mb W_L])} . 
\end{equation}
Let $\mf d_{\mf o_L / \mf o_K}$ be the discriminant of $\mf o_L / \mf o_K$, an
ideal of $\mf o_K$.
By \cite[Corollary VI.2.4]{Ser} applied to the trivial representation of $\mb W_L$:
\begin{equation}\label{eq:17.11}
\mb a (\CC [\mb W_K / \mb W_L]) = f_{L/K} \mb a (\mr{triv}_{\mb W_L}) + 
\mr{val}_K (\mf d_{\mf o_L / \mf o_K}) = \mr{val}_K (\mf d_{\mf o_L / \mf o_K}) .
\end{equation}
By \cite[Proposition III.3.6]{Ser} the image of the different $\mc D_{L/K}$ (an ideal in
$\mf o_L$) under the norm map for $L/K$ is precisely $\mf d_{\mf o_L / \mf o_K}$, so
\begin{equation}\label{eq:17.15}
\mr{val}_K (\mf d_{\mf o_L / \mf o_K}) = \mr{val}_L (\mc D_{L/K}) [L:K] e_{L/K}^{-1} = 
\mr{val}_L (\mc D_{L/K}) f_{L/K} .
\end{equation}
It follows from \eqref{eq:17.16}--\eqref{eq:17.15} that 
\begin{equation}\label{eq:17.17}
\left| \frac{\epsilon (\CC [\mb W_K / \mb W_{L}], \psi)}{\epsilon (\mr{triv}_{\mb W_L}, 
\psi_{L,l}) } \right| = q_K^{\mb a (\CC [\mb W_K / \mb W_L]) / 2} = q_K^{\mr{val}_L 
(\mc D_{L/K}) f_{L/K} / 2} = q_L^{\mr{val}_L (\mc D_{L/K}) /2} . 
\end{equation}
Finally we use that $q_L^{\mr{val}_L (\mc D_{L/K})} = [\mf o_L : \mc D_{L/K}]$ and we
plug \eqref{eq:17.17} into part (4).
\end{proof}

Supppose that $L/K$ is a finite separable field extension and that $\mb H$ is any connected
reductive $L$-group. Let $\mb G = \mr{Res}_{L/K}(\mb H)$ be the restriction of scalars of 
$\mb H$, from $L$ to $K$. Then $\mb G (K) = \mb H (L)$ and, according to 
\cite[Proposition 8.4]{Bor}, Shapiro's lemma yields a natural bijection
\begin{equation}\label{eq:16.2}
\Phi (\mb G (K)) \longrightarrow \Phi (\mb H (L)) .
\end{equation}
It is desirable that \eqref{eq:16.2} preserves L-functions, $\epsilon$-factors and 
$\gamma$-factors -- basically that is an aspect of the well-definedness of these local factors. 
Recall from \cite[\S 12.4.5]{Spr} that
\[
\mb G = \mr{ind}_{\mb W_K}^{\mb W_L} \mb H  \cong \mb H^{[L:K]} 
\quad \text{as } L \text{-groups.}
\]
This yields isomorphisms of $\mb W_K$-groups
\begin{equation}\label{eq:16.3}
\mb G^\vee \cong \mr{ind}_{\mb W_L}^{\mb W_K} (\mb H^\vee) \cong (\mb H^\vee)^{[L:K]} .
\end{equation}
We regard $\mb H^\vee$ as a subgroup of $\mb G^\vee$, embedded as the factor associated to the 
identity element in $\mb W_K / \mb W_L$. From the proof of Shapiro's lemma one gets an explicit 
description of \eqref{eq:16.2}: it sends 
\[
\lambda \in \Phi ( \mb G (K) ) \quad \text{to} \quad \big( \mb H^\vee\text{-component of }
\lambda \big|_{\mb W_L \times SL_2 (\CC)} \big) \in \Phi \big( \mb H (L) \big) .
\]
Let Ad or $\mr{Ad}_{\mb G^\vee}$ denote the adjoint representation of ${}^L \mb G$ on 
\[
\mf g^\vee / Z(\mf g^\vee)^{\mb W_K} = 
\mr{Lie}(\mb G^\vee) \big/ \mr{Lie} \big( (Z(\mb G)_s)^\vee \big) . 
\]
For any $\lambda \in \Phi (\mb G (K))$, $\mr{Ad} \circ \lambda$ is an admissible 
representation of $\mb W_K \times SL_2 (\CC)$ on $\mf g^\vee / Z(\mf g^\vee)^{\mb W_K}$.
We refer to the local factors of $\mr{Ad} \circ \lambda$ as the adjoint local factors of 
$\lambda$ (all of them tacitly with respect to the Haar measure on $K$ that gives $\mf o_K$
volume 1).

Beware that the split component of the centre $Z(\mb H)_s$ is not compatible with
restriction of scalars: $\mr{Res}_{L/K} Z(\mb H)_s$ is not $K$-split and contains
$Z (\mr{Res}_{L/K} \mb H)_s$ as a proper subgroup.

\begin{lemma}\label{lem:17.4}
Let $L$ be a finite separable extension of $K$ and denote the bijection 
$\Phi (\mb G (K)) \to \Phi \big( \mb H (L) \big)$ from \eqref{eq:16.2} by 
$\lambda \mapsto \lambda_{\mb H}$. Suppose that $Z(\mb H)^\circ$ is $L$-anisotropic.

Then $\mr{Ad}_{\mb G^\vee} \circ \lambda$ can be regarded as 
$\mr{ind}_{\mb W_{L} \times SL_2 (\CC)}^{\mb W_K \times SL_2 (\CC)} (\mr{Ad}_{\mb H^\vee} \circ 
\lambda_{\mb H})$. In particular the adjoint local factors of $\lambda$ and $\lambda_{\mb H}$ are 
related as in Lemma \ref{lem:A.1} and Theorem \ref{thm:A.2} (with $V_L = \mf h^\vee$ and
$V = \mf g^\vee$).
\end{lemma}
\begin{proof}
The proof of Shapiro's lemma and \eqref{eq:16.2} entail that every 
$\lambda \in \Phi (\mb G (K))$ is of the form 
\begin{equation}\label{eq:17.31}
\lambda = \mr{ind}_{\mb W_L \times SL_2 (\CC)}^{\mb W_K \times SL_2 (\CC)} \lambda_{\mb H} 
\quad \text{for a} \quad \lambda_{\mb H} \in \Phi \big( \mb H (L) \big) .
\end{equation}
To make sense of this induction, we regard $\lambda$ (resp. $\lambda_{\mb H}$) as a representation 
on $\mb G^\vee$ (resp. $\mb H^\vee$) and we apply \eqref{eq:16.3}. Then
$\mr{Ad}_{\mb G^\vee} \circ \lambda : \mb W_K \times SL_2 (\CC) \to \mr{Aut} ( \mf g^\vee)$ 
equals $\mr{ind}_{\mb W_{L} \times SL_2 (\CC)}^{\mb W_K \times SL_2 (\CC)} (\mr{Ad}_{\mb H^\vee} 
\circ \lambda_{\mb H})$. Knowing that, Lemma \ref{lem:A.1} and Theorem \ref{thm:A.2} apply.
\end{proof}

Lemma \ref{lem:17.4} shows that adjoint L-functions of groups with anisotropic centre are always 
preserved under restriction of scalars (most likely that was known already). Surprisingly, it 
also shows that adjoint $\epsilon$-factors and adjoint $\gamma$-factors are usually not 
preserved under Weil restriction, only if $L/K$ is unramified (and up to a sign).

We will deduce from Theorem \ref{thm:A.2} that the HII 
conjectures \cite{HII} are stable under Weil restriction: they hold for $\mb G (K)$ if and 
only if they hold for $\mb H (L)$. For that statement to make sense, we need a way
to transfer enhancements of L-parameters from $\mb G (K)$ to $\mb H (L)$:

\begin{lemma}\label{lem:17.1}
The map \eqref{eq:16.2} extends naturally to a bijection 
$\Phi_e (\mb G (K)) \to \Phi_e \big( \mb H (L) \big)$, which preserves cuspidality.
\end{lemma}
\begin{proof}
For $\Phi ( \mb H (L))$ the equivalence relation on L-parameters and the component 
groups come from the conjugation action of $\mb H^\vee$ and for $\Phi (\mb G (K))$ they come 
from the conjugation action of \eqref{eq:16.3}. But \eqref{eq:16.2} 
means that a L-parameter for $\mb G (K)$ depends (up to equivalence) only on its coordinates in
one factor $\mb H^\vee$ of $\mb G^\vee$, so the conjugation action of the remaining factors
of $\mb G^\vee$ can be ignored. Further, \eqref{eq:17.31} induces a group isomorphism
\begin{equation}\label{eq:17.32}
\begin{array}{ccc}
\mc A_{\lambda_{\mb H}} & \cong & \mc A_\lambda \\
a & \mapsto & [\gamma \mapsto a]
\end{array},
\end{equation}
where the right hand side is a subgroup of $\mr{ind}_{\mb W_L}^{\mb W_K} ({\mb H^\vee}_\SC)$.
Consequently \eqref{eq:16.2} extends to a bijection
\begin{equation}\label{eq:16.4}
\Phi_e (\mb G (K)) \to \Phi_e \big( \mb H (L) \big) .
\end{equation}
For $\mb G (K)$ cuspidality of enhancements of $\lambda$ is formulated via 
\begin{equation}\label{eq:16.5}
Z_{{\mb G^\vee}_\SC}(\lambda (\mb W_K)) = Z_{ \mr{ind}_{\mb W_L}^{\mb W_K} ({\mb H^\vee}_\SC ) }
(\lambda (\mb W_K)) \cong Z_{{\mb H^\vee}_\SC} (\lambda (\mb W_L)) ,
\end{equation}
where the isomorphism is a restriction of $\mc A_\lambda \cong \mc A_{\lambda_{\mb H}}$.
The right hand side of \eqref{eq:16.5} is just the group in which we detect cuspidality of 
enhancements of $\lambda \big|_{\mb W_L \times SL_2 (\CC)}$. 
Therefore \eqref{eq:16.4} respects cuspidality.
\end{proof}

In \cite{GrGa} a canonical Haar measure $|\omega_{\mb G}|$ on $\mb G (K)$ was constructed.
It involves the motive $M_{\mb G}$ of a reductive group \cite{Gro}. By \cite[(18) and \S 3.4]{GrRe} 
the Artin conductor of $M_{\mb G}$ equals the Artin conductor $\mb a (\mf g^\vee)$, where the 
Lie algebra $\mf g^\vee$ is endowed with the $\mb W_F$-action from conjugation in ${}^L \mb G$.

In \cite{HII} a measure $\mu_{\mb G,\psi}$ on $\mb G (K)$ 
is defined in terms of $|\omega_{\mb G}|$ and the Haar measure on $K$ that is self-dual with 
respect to $\psi$. When $\psi$ has order 0, that self-dual measure gives $\mf o_K$ volume 1 and 
$\mu_{\mb G,\psi}$ equals $|\omega_{\mb G}|$. That works well for $K_\nr$-split groups, but for 
ramified groups we need \cite[Correction]{HII}. Unfortunately the formula
\[
\mu'_{\mb G,\psi} = q^{-\mb a (\mf g^\vee )/2} \mu_{\mb G,\psi} 
\]
from \cite[Correction]{HII} is incompatible with the equality $\mu'_{\mb G,\psi} = \mu_{\mb G,\psi}$
for many groups \cite[Corollary 7.3]{GrGa}. On the other hand, such a $\mu'_{\mb G,\psi}$ is
definitely needed in \cite[Correction]{HII}. To make it work in all cases, we redefine
\begin{equation}\label{eq:17.18}
\mu_{\mb G,\psi} := q_K^{-(\mb a (\mf g^\vee) + \mr{ord}(\psi) \dim \mb G) /2} |\omega_{\mb G} | . 
\end{equation}
For $K_\nr$-split reductive groups this $\mu_{\mb G,\psi}$ agrees with \cite{HII}, because for 
those $\mb a (\mf g^\vee) = 0$ \cite[\S 4]{GrGa} and \eqref{eq:17.18} exhibits the same 
transformation behaviour with respect to $\psi$ as in \cite[(1.1)]{HII}.

\begin{lemma}\label{lem:A.3}
Let $\mb N$ be a normal connected $K$-subgroup of $\mb G$
such that the sequence of $K$-rational points
\[
1 \to \mb N (K) \to \mb G (K) \to (\mb G / \mb N)(K) \to 1
\]
is exact. Then $\mu_{\mb G,\psi} = \mu_{\mb N,\psi} \mu_{\mb G / \mb N,\psi}$, 
in the sense that for all $f \in C_c (\mb G(K))$:
\[
\int_{\mb G (K)} f(g) \, d \mu_{\mb G,\psi}(g) = \int_{(\mb G / \mb N)(K)}  \Big(
\int_{\mb N (K)} f(hn) \, d \mu_{\mb N,\psi}(n) \Big) d \mu_{\mb G / \mb N,\psi}(h) .
\]
\end{lemma}
\begin{proof}
We revisit the construction of $|\omega_{\mb G}|$ in \cite[\S 5]{GrGa}. Let $\mb G_0$ be the split
form of $\mb G$ and let $\varphi : \mb G \to \mb G_0$ be a isomorphism of $K_s$-groups.
Write $\mb N_0 = \varphi (\mb N)$ and $(\mb G / \mb N)_0 = \mb G_0 / \mb N_0$. We choose a Chevalley
model for $\mb G_0$ over $\mf o_K$. This also provides $\mb N_0$ and $(\mb G / \mb N)_0$ with 
Chevalley models over $\mf o_K$. Let $\omega_{\mb N_0}$ be an invariant differential form of
top degree which has good reduction modulo $\varpi_K$ (with respect to the Chevalley model). 
Choosing $\omega_{(\mb G / \mb N)_0}$ in the same way, the product $\omega_{\mb N_0} \omega_{(\mb G / 
\mb N)_0}$ defines an analogous invariant differential form $\omega_{\mb G_0}$ for $\mb G_0$.

The invariant differential forms $\omega_{\mb N}, \omega_{\mb G}$ and $\omega_{\mb G / \mb N}$ 
are obtained from their "split versions" by pullback along $\varphi$, so they also satisfy
$\omega_{\mb G} = \omega_{\mb N} \omega_{\mb G / \mb N}$. In view of the exactness of the 
sequence in the statement, the associated measures on the $K$-rational points of the involved 
groups are related as
\begin{equation}\label{eq:17.20}
|\omega_{\mb G}| = |\omega_{\mb N}| \, |\omega_{\mb G / \mb N}| . 
\end{equation}
As representation of Gal$(K_s/K)$, $\mf g^\vee$ is the direct sum of 
$\mf n^\vee := \mr{Lie}(\mb N^\vee)$ and Lie$(\mb G / \mb N)^\vee$. 
With the additivity of Artin conductors we deduce that
\begin{equation}\label{eq:17.21}
\hspace{-2mm} q_K^{-(\mb a (\mf g^\vee) + \mr{ord}(\psi) \dim \mb G) /2} =
q_K^{-(\mb a (\mf n^\vee) + \mr{ord}(\psi) \dim \mb N) /2} q_K^{-\big( \mb a (\mr{Lie}(\mb G / 
\mb N)^\vee) + \mr{ord}(\psi) \dim (\mb G / \mb N) \big) /2} . \hspace{-3mm}
\end{equation}
Comparing \eqref{eq:17.20} and \eqref{eq:17.21} with \eqref{eq:17.18}, we obtain
$\mu_{\mb G,\psi} = \mu_{\mb N,\psi} \mu_{\mb G / \mb N,\psi}$.
\end{proof}

Replacing $\mb H (L)$ by $\mb G (K)$ effects the formal degrees of square-integrable 
representations with respect to \eqref{eq:17.18}, but in a transparent way:

\begin{lemma}\label{lem:16.2}
Suppose that $Z(\mb H)^\circ$ is $L$-anisotropic and that $\pi \in \mr{Rep} (\mb G (K))$ is
square-integrable. Then
\[
\mr{fdeg}(\pi, \mu_{\mb G,\psi}) q_K^{-(\mb a (\mf g^\vee) + \mr{ord}(\psi) \dim \mb G)/2} =
\mr{fdeg}(\pi, \mu_{\mb H,\psi_L}) q_L^{-(\mb a (\mf h^\vee) + \mr{ord}(\psi_L) \dim \mb H)/2} .
\]
\end{lemma}
\begin{proof}
By \cite[Proposition 6.1.4]{GrGa} the measure $|\omega_{\mb G}|$ is respected by restriction 
of scalars, that is, $|\omega_{\mb H}| = |\omega_{\mb G}|$. Note that
\[
\mr{fdeg}(\pi, \mu_{\mb G,\psi}) q_K^{-(\mb a (\mf g^\vee) + \mr{ord}(\psi) \dim \mb G)/2} =
\mr{fdeg}(\pi, |\omega_{\mb G}|) 
\]
and similarly for $\mb H$.
\end{proof}

Our adjoint local factors coincide with those \cite{HII} if the additive characters
on $K$ and $L$ have order zero. But that is not always tenable. Namely, if 
$\psi : K \to \CC^\times$ has order zero, then $\psi_L$ (the composition of $\psi$ with the 
trace map for $L/K$) need not have order zero. More precisely, when ord$(\psi) = 0$, 
\cite[Theorem III.5.1]{Ser} says that $\psi_L$ has order zero if and only if $L/K$ is unramified. 

So far we used the Haar measure $dx$ on $K$ that gives $\mf o_K$ volume 1. When we compute 
$\epsilon$-factors or $\gamma$-factors with respect to an arbitrary additive character $\psi$,
the conventions in \cite{HII} impose that employ the Haar measure on $K$ which is self-dual
with respect to $\psi$. Thus we take $q^{-\mr{ord}(\psi) / 2} dx$ and we include it in the
notations of $\epsilon$-factors and $\gamma$-factors. 

For $a \in K^\times$ the additive character $\psi_a : x \mapsto \psi (xa)$ has order
$\mr{ord}(\psi) + \mr{val}_K (a)$. We recall from \cite[(3.4.3) and (3.4.4)]{Tat} that
\begin{align*}
\epsilon (s,V,\psi_a, q^{-(\mr{ord}(\psi) + \mr{val}_K (a))/2} dx) & =
\epsilon (s,V, \psi, q^{-(\mr{ord}(\psi) + \mr{val}_K (a))/ 2} dx) q_K^{\mr{val}_K (a) \dim (V)} 
\! \det (V,a) \\
& = \epsilon (s,V, \psi, q^{-\mr{ord}(\psi)/2} dx) q_K^{\mr{val}_K (a) \dim (V) / 2} \det (V,a) \\
& = \epsilon (s,V, \psi, q^{-\mr{ord}(\psi)/2} dx) | a |_K^{- \dim (V) / 2} \det (V,a) .
\end{align*}
With these notations, the HII conjecture can be formulated more precisely as
\begin{equation}\label{eq:17.19}
\mr{fdeg}(\pi, \mu_{\mb G,\psi}) = \dim (\rho_\pi) \, |S_{\lambda_\pi}^\sharp |^{-1} \, 
|\gamma (0,\mr{Ad}_{\mb G^\vee} \circ \lambda_\pi, \psi, q^{-\mr{ord}(\psi)/2} dx)| .
\end{equation}
Let us recall what \cite{HII} says about this for a split torus $GL_1^n$. Let $\chi$ be a unitary 
smooth character of $GL_1^n (L)$. By definition fdeg$(\chi, \mu_{GL_1^n,\psi_L})$ equals the formal 
degree of the trivial representation of the trivial group, which is just 1. Similarly all objects 
on the right hand of \eqref{eq:17.19} are trivial, for they come from the zero-dimensional 
representation of $\mb W_K \times SL_2 (\CC)$. Hence \eqref{eq:17.19} reduces to the equality $1 = 1$.

After restriction of scalars and dividing out the split component of the centre, we end
up with the anisotropic $K$-torus $\mb T := \mr{Res}_{L/K}(GL_1^n) / GL_1^n$. From the proof of
Theorem \ref{thm:HII} on page \pageref{eq:16.31}, we know that both sides of \eqref{eq:17.19}
for $\mr{Res}_{L/K}(GL_1)(K)$ reduce to the same expressions for $\mb T (K)$. Hence we may assume 
that $\chi$ is trivial on $GL_1^n (K)$. The conjecture of Hiraga, Ichino and Ikeda was never 
conjectural for tori, they proved it immediately. By \cite[Correction]{HII}, applied to the smooth 
character $\chi$ of $\mb T(K)$:
\begin{equation}\label{eq:17.2}
\mr{fdeg}(\chi, \mu_{\mb T,\psi}) = \frac{q^{\mb a (\mf t^\vee) / 2}}{|\mb T^{\vee,\mb W_K}|}
\frac{L(1,\mr{Ad}_{\mb T^\vee} \circ \lambda_\chi)}{L(0,\mr{Ad}_{\mb T^\vee} \circ \lambda_\chi)} 
= \frac{\dim (\rho_\chi)}{|S_{\lambda_\chi}^\sharp |} 
|\gamma (0,\mr{Ad}_{\mb T^\vee} \circ \lambda_\chi, \psi)| .
\end{equation}
Moreover, none of the terms in \eqref{eq:17.2} depends on $\chi$ and $\rho_\chi$ is just the 
trivial representation of the group $\mc A_{\lambda_\chi} = 1$.

\begin{prop}\label{prop:17.2}
Let $\mb H$ be any connected reductive $L$-group and let $\pi \in \Irr (\mb H (L))$ 
be square-integrable modulo centre. Let $(\lambda,\rho) \in
\Phi_e (\mb G (K))$ be an enhanced L-parameter associated to $\pi$, and let
$(\lambda_{\pi_\mb H}, \rho_{\pi_\mb H}) \in \Phi_e (\mb H (L))$ be its image under the map 
from Lemma \ref{lem:17.1}. The following are equivalent:
\begin{itemize}
\item The HII conjecture \eqref{eq:HII} holds for $\pi$ as $\mb G (K)$-representation,
with respect to any nontrivial additive character of $K$.
\item The HII conjecture \eqref{eq:HII} holds for $\pi$ as $\mb H (L)$-representation,
with respect to any nontrivial additive character of $L$.
\end{itemize} 
\end{prop}
\begin{proof}
If the HII conjecture holds for $\pi \in \Irr (\mb G (K))$ with respect to one nontrivial
additive character of $K$, then by \cite[Lemma 1.1 and 1.3]{HII} it holds with respect to
all nontrivial additive characters of $K$. The same applies to $\pi$ as representation of
$\mb H (L)$. Therefore we may assume that ord$(\psi) = 0$, and it suffices to consider
one additive character of $L$.

First we assume that $Z(\mb H)^\circ$ is $L$-anisotropic.
Choose $l \in L^\times$ of valuation $-\mr{ord}(\psi_L)$. Then the character $\psi_{L,l} :
x \mapsto \psi_L (l x)$ of $L$ has order zero. By Lemma \ref{lem:16.2} 
\begin{equation}\label{eq:17.8}
\mr{fdeg}(\pi, \mu_{\mb H, \psi_{L,l}}) = \mr{fdeg}(\pi, 
\mu_{\mb G, \psi}) \, q_K^{(f_{L/K} \mb a (\mf h^\vee) - \mb a (\mf g^\vee)) / 2} .
\end{equation}
By the self-duality of the adjoint representation of ${}^L \mb G$, the right hand side of the 
HII conjecture \eqref{eq:17.19} is
\begin{equation}\label{eq:17.9}
\frac{\dim (\rho_\pi)}{ |S_{\lambda_\pi}^\sharp |} \, \left| \epsilon (0,\mr{Ad}_{\mb G^\vee} 
\circ \lambda_\pi, \psi) \frac{L(1, \mr{Ad}_{\mb G^\vee} \circ 
\lambda_\pi)}{L(0,\mr{Ad}_{\mb G^\vee} \circ \lambda_\pi)} \right| .
\end{equation}
It follows quickly from \eqref{eq:16.5} that \eqref{eq:16.2} induces an isomorphism
\begin{equation}\label{eq:17.24}
S_{\lambda_\pi}^\sharp = \pi_0 \big( Z_{\mb G^\vee}(\lambda_\pi) \big) \longrightarrow
\pi_0 \big( Z_{\mb H^\vee}(\lambda_{\pi_{\mb H}}) \big) = S_{\lambda_{\pi_{\mb H}}}^\sharp .
\end{equation}
In \eqref{eq:17.32} we checked that $\mc A_{\lambda_\pi} \cong \mc A_{\lambda_{\pi_{\mb H}}}$,
which implies $\dim (\rho_\pi) = \dim (\rho_{\pi_{\mb H}})$. By Lemmas \ref{lem:A.1} and 
\ref{lem:17.4} the L-functions in \eqref{eq:17.9} do not change if we replace 
$\lambda_\pi$ by $\lambda_{\pi_\mb H}$. So all terms in \eqref{eq:17.9}, 
except possibly the $\epsilon$-factor, are inert under 
\[
(\lambda_\pi, \rho_\pi) \mapsto (\lambda_{\pi_\mb H}, \rho_{\pi_\mb H}).
\]
By Lemma \ref{lem:17.4} and Theorem \ref{thm:A.2}.(5) 
\begin{equation}\label{eq:17.10}
| \epsilon (0,\mr{Ad}_{\mb H^\vee} \circ \lambda_{\pi_{\mb H}}, \psi_{L,l})| =
| \epsilon (0,\mr{Ad}_{\mb G^\vee} \circ \lambda_\pi, \psi )| 
q_K^{-\mb a (\CC [\mb W_K / \mb W_L]) \dim (\mb H^\vee) / 2}
\end{equation}
By \cite[Corollary VI.2.4]{Ser} 
\[
\mb a (\mf g^\vee) = f_{L/K} \mb a (\mf h^\vee) + \dim (\mb H^\vee) \mb a (\CC [\mb W_K / \mb W_L]) .
\]
Then \eqref{eq:17.10} becomes
\begin{equation}\label{eq:17.12}
| \epsilon (0,\mr{Ad}_{\mb H^\vee} \circ \lambda_{\pi_{\mb H}}, \psi_{L,l})| =
| \epsilon (0,\mr{Ad}_{\mb G^\vee} \circ \lambda_\pi, \psi )| \, 
q_K^{(f_{L/K} \mb a (\mf h^\vee) - \mb a (\mf g^\vee)) / 2} .
\end{equation}
From \eqref{eq:17.8} and \eqref{eq:17.12} we see that replacing  $\mb G (K)$ by $\mb H (L)$
adjusts both sides of \eqref{eq:17.19} by the same factor  
$q_K^{(f_{L/K} \mb a (\mf h^\vee) - \mb a (\mf g^\vee)) / 2}$. 

We come to the general case, now $\mb H$ can be any connected reductive $L$-group. We abbreviate 
$\mb G' = \mr{Res}_{L/K}(\mb H / Z(\mb H)_s)$ and  $\mb T = \mr{Res}_{L/K}(Z(\mb H)_s) / 
\mb Z (\mb G)_s$. Applying \eqref{eq:15.7} to $\mb H (L)$, we obtain a short exact sequence
\begin{equation}\label{eq:17.6}
1 \to Z(\mb H)_s (L) \to \mb G (K) \to \mb G' (K) \to 1 . 
\end{equation}
With Galois cohomology one checks that $\mb T (K) = Z(\mb H)_s (L) / Z(\mb G)_s (K)$,
just like \eqref{eq:15.7}. Plugging that into \eqref{eq:17.6}, we obtain a short exact sequence
\begin{equation}\label{eq:17.23}
1 \to \mb T (K) \to (\mb G / Z(\mb G)_s)(K) \to \mb G' (K) \to 1 . 
\end{equation}
Fix a unit vector $v$ in the Hilbert space on which $\pi$ is defined. 
The formal degree of $\pi$ is given in \cite{HII} as
\begin{equation}\label{eq:17.22}
\mr{fdeg}(\pi, d \mu_{\mb G / Z(\mb G)_s ,\psi})^{-1} = \int_{(\mb G / Z(\mb G)_s)(K)} 
| \langle \pi (g) v,v \rangle |^2 d \mu_{\mb G / Z(\mb G)_s ,\psi} (g) .
\end{equation}
As $\pi$ is square-integrable, its central character is unitary. Hence
$| \langle \pi (g) v,v \rangle |$ depends only on the image of $g$ in $\mb H(L) / Z(\mb H)_s (L) 
= \mb G' (K)$. With \eqref{eq:17.23} and Lemma \ref{lem:A.3} we see that \eqref{eq:17.22} equals
\[
\int_{\mb G'(K)} \int_{\mb T (K)} | \langle \pi (g) v,v \rangle |^2 d \mu_{\mb T,\psi} 
d \mu_{\mb G',\psi} (g) = \mu_{\mb T,\psi}(\mb T (K)) \int_{\mb G'(K)} 
| \langle \pi (g) v,v \rangle |^2 d \mu_{\mb G',\psi} (g) .
\]
By \eqref{eq:16.18} and \eqref{eq:17.2}, for any $\chi \in \Irr (\mb T (K))$:
\begin{equation}\label{eq:17.7}
\mu_{\mb T,\psi}(\mb T (K))^{-1} = \mr{fdeg}(\chi, d \mu_{\mb T,\psi}) =  
\dim (\rho_\chi) |S_{\lambda_\chi}^\sharp |^{-1} 
|\gamma (0,\mr{Ad}_{\mb T^\vee} \circ \lambda_\chi, \psi)| .
\end{equation}
We point out that $\mf h'^\vee := \mr{Lie} ((\mb H / Z(\mb H)_s)^\vee)$ is a representation of
Gal$(K_s / L)$. Its Artin conductor is understood as such. From \cite[Proposition 6.1.4]{GrGa} and 
\eqref{eq:17.18} with ord$(\psi) = \mr{ord}(\psi_{L,l}) = 0$  we see that
\begin{align*}
\int_{\mb G'(K)} | \langle \pi (g) v,v \rangle |^2 d \mu_{\mb G',\psi} (g) & = 
q_K^{-\mb a (\mf g'^\vee) /2} q_L^{\mb a (\mf h'^\vee) /2}  \int_{(\mb H / Z(\mb H)_s)(L)} 
\hspace{-7mm} | \langle \pi (g) v,v \rangle |^2 d \mu_{\mb H / Z(\mb H)_s,\psi_{L,l}} (g) \\
& = q_K^{(f_{L/K} \mb a (\mf h'^\vee) - \mb a (\mf g'^\vee)) /2}
\mr{fdeg}(\pi, d  \mu_{\mb H / Z(\mb H)_s,\psi_{L,l}})^{-1} . 
\end{align*}
When we put all this into \eqref{eq:17.22}, we find that
\begin{equation}\label{eq:17.25}
\begin{aligned}
& \mr{fdeg}(\pi, d \mu_{\mb G / Z(\mb G)_s ,\psi}) = \\
& q_K^{(\mb a (\mf g'^\vee) - f_{L/K} \mb a (\mf h'^\vee)) /2} \mr{fdeg}(\pi, 
d \mu_{\mb H / Z(\mb H)_s,\psi_{L,l}}) \mr{fdeg}(\chi, d \mu_{\mb T,\psi}) .
\end{aligned}
\end{equation}
By \eqref{eq:17.23} the representation space $\mr{Lie}(\mb G^\vee) / \mr{Lie}((Z(\mb G)_s)^\vee)$ 
for the adjoint $\gamma$-factor can be decomposed as $\mr{Lie}(\mb G'^\vee) \oplus \mr{Lie}(\mb T^\vee)$.
The action on the central Lie subalgebra $\mr{Lie}(\mb T^\vee)$ does not depend on the specific
L-parameter, it comes only from the canonical $\mb W_K$-action. The additivity of local factors
tells us that
\begin{equation}\label{eq:17.26}
\gamma (s,\mr{Ad}_{\mb G^\vee} \circ \lambda_\pi, \psi) =
\gamma (s,\mr{Ad}_{\mb G'^\vee} \circ \lambda_\pi, \psi) 
\gamma (s,\mr{Ad}_{\mb T^\vee} \circ \lambda_\chi, \psi) .
\end{equation}
From \eqref{eq:17.23} we also get a short exact sequence
\begin{equation}\label{eq:17.27}
1 \to Z_{\mb G'^\vee}(\lambda_\pi) \to Z_{(\mb G^\vee / Z(\mb G)_s )^\vee}(\lambda_\pi)
\to Z_{\mb T^\vee}(\lambda_\chi) \to 1 .
\end{equation}
The involved L-parameters are discrete, so all the groups in \eqref{eq:17.27} are finite. 
By definition, their component groups are the $S^\sharp_{\lambda_\pi}$ for the three cases.
Thus \eqref{eq:17.27} says that
\begin{equation}\label{eq:17.28}
S_{\lambda_\pi}^\sharp \big/ S_{\lambda'_\pi}^\sharp \cong S_{\lambda_\chi} \quad \text{and} 
\quad |S_{\lambda_\pi}^\sharp | = |S_{\lambda'_\pi}^\sharp| \, |S_{\lambda_\chi} | , 
\end{equation}
where the prime means that the group comes from $\mb G'$. Similar calculations show that 
$\mc A_{\lambda_\chi} = 1$ and $\mc A_{\lambda'_\pi} = \mc A_{\lambda_\pi}$. In particular 
$\rho_\pi$ is the same for $\mb G$ and for $\mb G'$. From that, \eqref{eq:17.26} and 
\eqref{eq:17.28} we deduce that
\begin{equation}\label{eq:17.29}
\frac{\dim (\rho_\pi)}{|S_{\lambda_\pi}^\sharp |} 
|\gamma (0,\mr{Ad}_{\mb G^\vee} \circ \lambda_\pi, \psi)| =
\frac{\dim (\rho_\pi)}{|S_{\lambda'_\pi}^\sharp |} 
|\gamma (0,\mr{Ad}_{\mb G'^\vee} \circ \lambda_\pi, \psi)|
\frac{|\gamma (0,\mr{Ad}_{\mb T^\vee} \circ \lambda_\chi, \psi)|}{|S_{\lambda_\chi}^\sharp |} 
\end{equation}
Using \eqref{eq:17.2} and our earlier findings for groups with anisotropic centre,
in particular \eqref{eq:17.12}, we can simplify \eqref{eq:17.29} to
\[
q_K^{(\mb a (\mf g'^\vee) - f_{L/K} \mb a (\mf h'^\vee)) /2}  
\frac{\dim (\rho_{\pi_{\mb H}})}{|S_{\lambda_{\pi_{\mb H}}}^\sharp |} 
\: |\gamma (0,\mr{Ad}_{\mb H^\vee} \circ \lambda_{\pi_{\mb H}}, \psi)| \:
\mr{fdeg}(\chi, d \mu_{\mb T,\psi}) .
\]
We compare that with \eqref{eq:17.25} and we note that replacing $\mb H (L)$ by $\mb G (K)$
adjusts both sides of the HII conjecture \eqref{eq:17.19} by a factor
\[
q_K^{(\mb a (\mf g'^\vee) - f_{L/K} \mb a (\mf h'^\vee)) /2}  
\mr{fdeg}(\chi, d \mu_{\mb T,\psi}) . \qedhere
\]
\end{proof}

In the body of the paper we only use Proposition \ref{prop:17.2} for unramified extensions
$L/K$. That case can be proven more elementarily, without Artin conductors.

\end{document}